\documentclass[11pt,english]{article}
\usepackage[T1]{fontenc}
\usepackage[utf8]{inputenc}
\usepackage{geometry}
\usepackage{color}
\usepackage{babel}
\usepackage{array}
\usepackage{prettyref}
\usepackage{units}
\usepackage{amsmath}
\usepackage{amsthm}
\usepackage{amssymb}
\usepackage{stackrel}
\usepackage{graphicx}
\usepackage[all]{xy}
\PassOptionsToPackage{normalem}{ulem}
\usepackage{ulem}
\usepackage[unicode=true,pdfusetitle,
 bookmarks=true,bookmarksnumbered=false,bookmarksopen=false,
 breaklinks=false,pdfborder={0 0 1},backref=page,colorlinks=true]
 {hyperref}
\hypersetup{
 linkcolor=blue,citecolor=blue}

\makeatletter

\providecommand{\tabularnewline}{\\}

\numberwithin{equation}{section}
\numberwithin{figure}{section}
\theoremstyle{plain}
\newtheorem{thm}{\protect\theoremname}[section]
\theoremstyle{definition}
\newtheorem{defn}[thm]{\protect\definitionname}
\theoremstyle{remark}
\newtheorem*{rem*}{\protect\remarkname}
\theoremstyle{plain}
\newtheorem{cor}[thm]{\protect\corollaryname}
\theoremstyle{plain}
\newtheorem{conjecture}[thm]{\protect\conjecturename}
\theoremstyle{remark}
\newtheorem{rem}[thm]{\protect\remarkname}
\theoremstyle{definition}
\newtheorem{example}[thm]{\protect\examplename}
\theoremstyle{plain}
\newtheorem{prop}[thm]{\protect\propositionname}
\theoremstyle{plain}
\newtheorem{lem}[thm]{\protect\lemmaname}

\usepackage{babel}

\usepackage{ytableau}

\makeatother

\providecommand{\conjecturename}{Conjecture}
\providecommand{\corollaryname}{Corollary}
\providecommand{\definitionname}{Definition}
\providecommand{\examplename}{Example}
\providecommand{\lemmaname}{Lemma}
\providecommand{\propositionname}{Proposition}
\providecommand{\remarkname}{Remark}
\providecommand{\theoremname}{Theorem}

\begin{document}
\global\long\def\F{\mathbb{\mathbb{\mathbf{F}}}}%
 
\global\long\def\rk{\mathbb{\mathrm{rk}}}%
 
\global\long\def\crit{\mathbb{\mathrm{Crit}}}%
 
\global\long\def\critak{\crit_{\ak}}%
 
\global\long\def\Hom{\mathrm{Hom}}%
 
\global\long\def\defi{\stackrel{\mathrm{def}}{=}}%
 
\global\long\def\kl{k_{1},\ldots,k_{\ell}}%
 
\global\long\def\ak{\alpha_{1},\ldots,\alpha_{k}}%
 
\global\long\def\qak{\xi_{1}^{\alpha_{1}}\cdots\xi_{k}^{\alpha_{k}}}%
 
\global\long\def\tr{{\cal T}r }%
 
\global\long\def\id{\mathrm{id}}%
 
\global\long\def\Aut{\mathrm{Aut}}%
 
\global\long\def\wl{w_{1},\ldots,w_{\ell}}%
 
\global\long\def\alg{\le_{\mathrm{alg}}}%
 
\global\long\def\ff{\stackrel{*}{\le}}%
 
\global\long\def\pmulti{\pi_{1}^{\mathrm{multi}}}%
 
\global\long\def\mocc{{\cal MOCC}}%
 
\global\long\def\mcg{{\cal M}\mathrm{u}{\cal C}{\cal G}}%
 
\global\long\def\decomp{{\cal D}\mathrm{ecomp}_{B}}%
 
\global\long\def\decompt{{\cal D}\mathrm{ecomp}_{B}^{3}}%
 
\global\long\def\algdecomp{{\cal D}\mathrm{ecomp}_{\mathrm{alg}}}%
 
\global\long\def\algdecompt{{\cal D}\mathrm{ecomp}_{\mathrm{alg}}^{3}}%
 
\global\long\def\A{{\cal A}}%
 
\global\long\def\z{{\cal \xi}}%
 
\global\long\def\mobius{M\dacute{o}bius}%
\global\long\def\chimax{\chi^{\mathrm{max}}}%
 
\global\long\def\chiak{\chi_{\ak}^{\mathrm{max}}}%
 
\global\long\def\uexp{\mathbb{E}_{\mathrm{unif}}}%
\global\long\def\symirr{\widehat{S_{\infty}}}%
 
\global\long\def\scl{\mathrm{scl}}%
 
\global\long\def\H{{\cal H}}%
 
\global\long\def\J{{\cal J}}%
 
\global\long\def\piol{\pi_{1}^{\mathrm{lab}}}%
 
\global\long\def\E{\overrightarrow{E}}%
 
\global\long\def\spec{\mathrm{Spec}}%
 
\global\long\def\cyr{\mathfrak{CR}}%

\title{Word Measures on Symmetric Groups}
\author{Liam Hanany~~~~and~~~~Doron Puder}
\maketitle
\begin{abstract}
Fix a word $w$ in a free group $\F$ on $r$ generators. A $w$-random
permutation in the symmetric group $S_{N}$ is obtained by sampling
$r$ independent uniformly random permutations $\sigma_{1},\ldots,\sigma_{r}\in S_{N}$
and evaluating $w\left(\sigma_{1},\ldots,\sigma_{r}\right)$. In \cite{Puder2014,PP15}
it was shown that the average number of fixed points in a $w$-random
permutation is $1+\theta\left(N^{1-\pi\left(w\right)}\right)$, where
$\pi\left(w\right)$ is the smallest rank of a subgroup $H\le\F$
containing $w$ as a non-primitive element. We show that $\pi\left(w\right)$
plays a role in estimates of all stable characters of symmetric groups.
In particular, we show that for all $t\ge2$, the average number of
$t$-cycles is $\frac{1}{t}+O\left(N^{-\pi\left(w\right)}\right)$.
As an application, we prove that for every $s$, every $\varepsilon>0$
and every large enough $r$, Schreier graphs with $r$ random generators
depicting the action of $S_{N}$ on $s$-tuples, have second eigenvalue
at most $2\sqrt{2r-1}+\varepsilon$ asymptotically almost surely.
An important ingredient in this work is a systematic study of not-necessarily
connected Stallings core graphs.
\end{abstract}
\tableofcontents{}

\section{Introduction\label{sec:Introduction}}

Fix $r\in\mathbb{Z}_{\ge1}$. Throughout this paper we let $\F$ denote
the free group on $r$ generators. A word $w\in\F$ induces a map
on any finite group, $w:G^{r}\to G$, by substituting the letters
of $w$ with elements of $G$. This map defines a distribution on
the group $G$: the push forward of the uniform distribution on $G^{r}$.
Equivalently, this distribution is the normalized number of times
each element in $G$ is obtained by a substitution in $w$. We call
this distribution \emph{the $w$-measure on $G$}. For example, if
$w=xyxy^{-2}$, a $w$-random element in $G$ is $ghgh^{-2}$ where
$g,h$ are independent, uniformly random elements of $G$.

More concretely, we study the expected value with respect to word
measures of certain class functions (functions invariant under conjugation).
Given a class function $f\colon G\to\mathbb{R}$, we analyze $\mathbb{E}_{w}\left[f\right]$,
the average value of this function under the $w$-measure on $G$.
Word measures are constant on conjugacy classes of $G$, i.e.~are
themselves class functions on the group $G$. Therefore, the expressions
$\mathbb{E}_{w}\left[f\right]$, running over a suitable family of
class functions (for example, all irreducible characters of $G$),
uniquely determine the $w$-measure on $G$. Several papers studying
word measures on various groups are motivated by questions from the
field of free probability, where the asymptotic statistics of such
measures on families of groups was analyzed. In recent years, different
works found more refined and deeper structure in these measures. We
mention some of these works in Section \ref{subsec:Similar-phenomena-in}.
The current work is the first one where non-trivial bounds are given
on all ``natural'' families of class functions on a given family
of groups, as we now explain.

Our focus in this paper is on word measures on the symmetric groups
$S_{N}$, and especially on the following class functions. For every
$k\in\mathbb{Z}_{\ge1}$, denote

\begin{equation}
\z_{k}\left(\sigma\right)\defi\#\mathrm{fix}\left(\sigma^{k}\right)\label{eq:xi k}
\end{equation}
where $\#\mathrm{fix}(\tau)$ is the number of fixed points of the
permutation $\tau$. We study the expected value under word measures
of products of the $\xi_{k}$'s in the form of $\xi_{1}^{\alpha_{1}}\xi_{2}^{\alpha_{2}}\cdots\xi_{k}^{\alpha_{k}}$
with $k\in\mathbb{Z}_{\ge1}$ and $\alpha_{1},\ldots,\alpha_{k}\in\mathbb{Z}_{\ge0}$
with $\sum\alpha_{k}\ge1$. When we write $\mathbb{E}_{w}\left[\xi_{1}^{\alpha_{1}}\cdots\xi_{k}^{\alpha_{k}}\right]$
there is a suppressed parameter $N$, namely, $\mathbb{E}_{w}\left[\xi_{1}^{\alpha_{1}}\cdots\xi_{k}^{\alpha_{k}}\right]$
is a map $\mathbb{Z}_{\ge1}\to\mathbb{Q}$, where $N$ is mapped to
the average value of this class function under the $w$-measure on
$S_{N}$.

\subsection{Main theorem}

For every word $w\in\F$ and $k,\ak$ as above, the expectation $\mathbb{E}_{w}\left[\qak\right]$
is a rational function of $N$, for large enough $N$: this is essentially
a result of \cite{nica1994number}, and see also Section 4 and especially
Remark 31 in \cite{Linial2010}. (It also follows from the analysis
in the current paper -- see Corollary \ref{cor:Phi, L, R, C all rational}.)
For example, $\mathbb{E}_{xyx^{-1}y^{-1}}\left[\xi_{1}\z_{2}\right]=3+\frac{4\left(N^{4}-9N^{3}+23N^{2}-13N-1\right)}{N\left(N-1\right)\left(N-2\right)\left(N-3\right)\left(N-5\right)}$
for all $N\ge6$. In particular, for large enough $N$, $\mathbb{E}_{w}\left[\qak\right]$
can be written as a Laurent series in $N$. Our main goal in this
paper is to estimate the leading terms of this Laurent series expansion.
The special case of $\z_{1}=\#\mathrm{fix}(\sigma)$, the average
number of fixed points, was studied in \cite{Puder2014,PP15}. These
papers show a connection between $\mathbb{E}_{w}\left[\z_{1}\right]$
and invariants of $w$ as an element of the free group.

In order to explain these invariants, we need a few notions from the
realm of free groups. A \emph{basis }of a free group is a free generating
set (or, equivalently for finitely generated free groups, a generating
set of minimal size). An element $w\in\F$ is called \emph{primitive}
if it belongs to a basis of $\F$. The rank of the free group $\F$,
denoted $\rk\F,$ is the size of a basis of $\F$. The classical Nielsen-Schreier
theorem states that subgroups of free groups are free. The primitivity
rank of a word, which plays an important role in this paper, was first
introduced in \cite{Puder2014}:
\begin{defn}
\label{def:primitivity-rank}The primitivity rank $\pi(w)$ of a word
$w\in\F$ is the minimal rank of a subgroup $H\leq\F$ containing
$w$ as a non-primitive element. If there are no such subgroups, set
$\pi\left(w\right)=\infty$. We also consider the set of \emph{critical
subgroups} of $w$ defined as
\[
\crit\left(w\right)=\left\{ H\le\F\,\middle|\,\rk H=\pi\left(w\right),H\ni w~\mathrm{and}~w~\mathrm{non-primitive~in}~H\right\} .
\]
\end{defn}

For example, $\pi\left(w\right)=0\iff w=1$ as the trivial word is
contained in the trivial subgroup but not as a primitive element.
Words with $\pi\left(w\right)=1$ are precisely proper powers and
if $u\in\F$ is not a proper power and $m\ge2$, then $\crit\left(u^{m}\right)=\left\{ \left\langle u^{d}\right\rangle \,\middle|\,d\mid m,~1\le d<m\right\} $.
Finally, $\pi\left(w\right)=\infty$ if and only if $w$ is primitive
in $\F$, and in any other case $\pi\left(w\right)\le r=\rk\F$ \cite[Lemma 4.1]{Puder2014}.
The set $\crit\left(w\right)$ is always finite \cite[Section 4]{PP15}.
We can now state the aforementioned result from \cite{PP15}.
\begin{thm}
\cite[Theorem 1.8]{PP15} \label{thm:PP15}For every word $w\in\F$

\[
\mathbb{E}_{w}\left[\#\mathrm{fix}(\sigma)\right]=1+\frac{|\crit(w)|}{N^{\pi(w)-1}}+O\left(\frac{1}{N^{\pi(w)}}\right).
\]
\end{thm}

Since the expected number of fixed point in a uniformly random permutation
is $1$, the theorem can be restated as

\[
\mathbb{E}_{w}\left[\z_{1}\right]=\uexp\left[\z_{1}\right]+\frac{\left|\crit\left(w\right)\right|}{N^{\pi(w)-1}}+O\left(\frac{1}{N^{\pi(w)}}\right),
\]
where $\mathbb{E}_{\mathrm{unif}}\left[f\right]$ is the expectation
of the function $f$ with respect to the uniform distribution on $S_{N}$.
The main result of this paper is the following generalization of Theorem
\ref{thm:PP15}. The quantity $\left\langle \qak,\z_{1}-1\right\rangle $
appearing in the statement is defined on page \pageref{inner product of class functions}
below.
\begin{thm}
\label{thm:Word-Measure-Character-Bound}For every non-power $w\in\F$,
and for every $k\in\mathbb{Z}_{\ge1}$ and $\ak\in\mathbb{Z}_{\ge0}$,
there exists a positive integer $C_{\ak}\in\mathbb{Z}_{\ge1}$ such
that
\begin{equation}
\mathbb{E}_{w}\left[\qak\right]=\uexp\left[\qak\right]+C_{\ak}\cdot\frac{|\crit(w)|}{N^{\pi(w)-1}}+O\left(\frac{1}{N^{\pi(w)}}\right).\label{eq:main thm}
\end{equation}
Moreover, the constant $C_{\ak}$ is equal to $\left\langle \qak,\z_{1}-1\right\rangle $. 
\end{thm}

In particular, $\mathbb{E}_{w}\left[\qak\right]\ge\uexp\left[\qak\right]$
for all large enough $N$. Note that the exclusion of powers in the
statement of the theorem is necessary: for example, let $x\in\F$
be a basis element. while $\mathbb{E}_{x^{3}}\left[\xi_{2}\right]=4$
for $N\ge6$, we have $\pi\left(x^{3}\right)=1$, $\crit\left(x^{3}\right)=\left\{ \left\langle x\right\rangle \right\} $,
$\uexp\left[\xi_{2}\right]=2$ for $N\ge2$ and $\left\langle \xi_{2},\xi_{1}-1\right\rangle =1$,
and so \eqref{eq:main thm} would give in this case $2+1\cdot\frac{1}{N^{0}}+O\left(\frac{1}{N}\right)=3+O\left(\frac{1}{N}\right)$,
which is incorrect. However, these expected values can still be understood.
Indeed, $\left(\qak\right)\left(\sigma^{t}\right)=\left(\xi_{t}^{\alpha_{1}}\cdots\xi_{kt}^{\alpha_{k}}\right)(\sigma)$.
Hence, we can still obtain an approximation for the expected value
of $\qak$ under a power-word, and see also Corollary \ref{cor:fixed points of power words}.
In Remark \ref{rem:expectaion in uniform measure of qak} we provide
a combinatorial formula for $\uexp\left[\qak\right]$.
\begin{rem*}
Throughout this paper, $N^{-\infty}$ should be interpreted as zero.
In particular, the results in this paper, Theorems \ref{thm:PP15}
and \ref{thm:Word-Measure-Character-Bound} included, hold for primitive
words as well, for which, as mentioned above, $\pi\left(w\right)=\infty$.
For example, in this case \eqref{eq:main thm} becomes $\mathbb{E}_{w}\left[\qak\right]=\mathbb{E}_{\mathrm{unif}}\left[\qak\right]$.
Indeed, primitive words induce the uniform measure on every finite
group \cite[Observation 1.2]{PP15}.
\end{rem*}
We perceive our results as interesting and elegant for their own sake.
We do, however, have further motivation for this study. One piece
of motivation comes from consequences of Theorem \ref{thm:Word-Measure-Character-Bound}
to the expansion of random Schreier graphs of $S_{N}$ as detailed
in Section \ref{subsec:Expansion-of-random} below. More motivation
comes from a growing body of evidence to an interplay between word
measures in general and primitivity rank of words in particular on
the one hand, and seemingly unrelated questions and challenges in
combinatorial and geometric group theory on the other hand. This is
illustrated by the works \cite{PP15,hanany2020some} dealing with
``profinite rigidity'' of words, and by the papers \cite{louder2018negative,louder2021uniform}
where the primitivity rank of a word is shown to have crucial consequences
for the one-relator group this word defines. We add here further evidence:
in Section \ref{subsec:Similar-phenomena-in} we explain how some
of the ideas in this paper are related to the notion of stable commutator
length, our proof of Theorem \ref{thm:Word-Measure-Character-Bound}
in Section \ref{sec:Proof-of-Theorem} suggests a deep connection
between the ``dependence theorems'' of \cite{louder2013scott,louder2018negative}
and word measures on $S_{N}$, and in Appendix \ref{sec:conjugacy-seperability}
we use our main result to get a new and simple proof of the conjugacy
separability of free groups.

\subsection{General \textquotedblleft stable\textquotedblright{} class functions}

Consider the abstract polynomial ring $\A=\mathbb{Q}\left[\z_{1},\z_{2},\ldots\right]$
\label{the ring A}in countably many variables. Every element $f\in\A$
corresponds to a class function in $S_{N}$ for all $N$. The elements
analyzed in Theorem \ref{thm:Word-Measure-Character-Bound} are precisely
the monomials in $\A$, and thus give a linear basis for $\A$. For
every class functions $f,g\in\A$ and every $N\in\mathbb{Z}_{\ge1}$,
we have the ordinary inner product in $S_{N}$ defined as 
\[
\left\langle f,g\right\rangle _{S_{N}}\defi\frac{1}{N!}\sum_{\sigma\in S_{N}}f\left(\sigma\right)\cdot\overline{g\left(\sigma\right)}.
\]
For every $f,g\in\A$, this inner product stabilizes for large enough
$N$ -- see Proposition \ref{prop:inner product stabilizes}. We
denote this constant value by\label{inner product of class functions}
$\left\langle f,g\right\rangle $\label{<f,g>}. In particular, note
that $\left\langle f,1\right\rangle =\mathbb{E}_{\mathrm{unif}}\left[f\right]$
for every large enough $N$. The following corollary thus follows
immediately from Theorem \ref{thm:Word-Measure-Character-Bound}.
As above, for $f\in\A$ we denote by $\mathbb{E}_{w}\left[f\right]$
a function $\mathbb{Z}_{\ge1}\to\mathbb{Q}$ which maps $N$ to the
average value of $f$ in $S_{N}$ under the $w$-measure.
\begin{cor}
\label{cor:general sym function}For every class function $f\in\A$
and every non-power $w\in\F$,
\[
\mathbb{E}_{w}\left[f\right]=\left\langle f,1\right\rangle +\left\langle f,\z_{1}-1\right\rangle \cdot\frac{\left|\crit\left(w\right)\right|}{N^{\pi\left(w\right)-1}}+O\left(\frac{1}{N^{\pi\left(w\right)}}\right).
\]
\end{cor}

Some elements of the ring $\A$ coincide with characters of families
of representations: such a family consists of a representation of
$S_{N}$ for every large enough $N$. These families are precisely
the families of \emph{stable} representations of $\left\{ S_{N}\right\} $,
in the sense of \cite{church2013representation}. These families were
studied in \cite{church2015fi} and subsequent works.

An interesting special case of Corollary \ref{cor:general sym function}
deals with statistics of short cycles in $S_{N}$. For $t\in\mathbb{Z}_{\ge1}$,
let $a_{t}\left(\sigma\right)$ denote the number of cycles of length
$t$ in the permutation $\sigma$. This is an element in $\A$: for
example, $a_{2}=\frac{\z_{2}-\z_{1}}{2}$. For $N\ge t,$the expected
number of $t$-cycles in a uniformly random permutation in $S_{N}$
is $\frac{1}{t}$. For $t\ge2$, $\left\langle a_{t},\z_{1}-1\right\rangle =0$
(see Appendix \ref{sec:class-functions} for more details). Therefore,
\begin{cor}
Let $t\geq2$. For every non-power $w\in\F$,

\[
\mathbb{E}_{w}\left[a_{t}\right]=\frac{1}{t}+O\left(\frac{1}{N^{\pi(w)}}\right).
\]
\end{cor}

Another special case of Theorem \ref{thm:Word-Measure-Character-Bound}
we mention explicitly is that of the functions $\z_{d}$, as they
relate to a general conjecture about word measures. This conjecture
asks whether two words $w_{1}$ and $w_{2}$ in $\F$ inducing the
same measure on every finite group are necessarily in the same orbit
of $\mathrm{Aut}\F$. It appears, for example, as \cite[Question 2.2]{Amit2011}
and \cite[Conjecture 4.2]{Shalev2013}, and see also \cite{Collins2019automorphism-invariant}.
(The converse, that two words in the same orbit induce the same measure
on every finite group, is a simple observation.) The special case
of this conjecture when $w_{1}$ is primitive, namely, in the orbit
of the word $x$, was settled in \cite{PP15}: it follows from Theorem
\ref{thm:PP15} that if $w_{2}$ induces the same measure as $x$,
namely, uniform measure, on $S_{N}$ for all $N$, then $w_{2}$ must
be primitive too. This was generalized in \cite[Theorem 1.4]{hanany2020some}
to show that the conjecture is true when $w_{1}=x^{d}$ or $w_{1}=\left[x,y\right]^{d}$
for arbitrary $d\in\mathbb{Z}_{\ge1}$, and in \cite{wilton2021profinite}
to all surface words $w_{1}=\left[x_{1},y_{1}\right]\cdots\left[x_{g},y_{g}\right]$
or $w_{1}=x_{1}^{~2}\cdots x_{g}^{~2}$ and their powers.

Consider the case $w_{1}=x^{d}$. The proof in \cite[Theorem 1.4]{hanany2020some}
has two steps: first, it can be shown that if $w_{2}$ induces the
same measures as $x^{d}$ then $w_{2}=u^{d}$ is a $d$-th power of
some non-power word $u$. Then it is shown that if $w_{2}$ is not
in the orbit of $x^{d}$, then for every large enough $N$, the average
number of fixed points in a $w_{2}$-random permutation in $S_{N}$
is strictly larger than that of $x^{d}$. Our results here give a
quantitative version of this step. For every $d\in\mathbb{Z}_{\ge1}$,
$\left\langle \z_{d},1\right\rangle =\tau\left(d\right)$, where $\tau\left(d\right)$
is the number of positive divisors of $d$, and $\left\langle \z_{d},\z_{1}-1\right\rangle =1$.
Hence,
\begin{cor}
\label{cor:fixed points of power words}Assume $1\ne u\in\F$ is a
non-power and let $w=u^{d}$ for some $d\in\mathbb{Z}_{\ge1}$. Then,
\[
\mathbb{E}_{w}\left[\z_{1}\right]=\mathbb{E}_{u}\left[\z_{d}\right]=\tau\left(d\right)+\frac{\left|\crit\left(u\right)\right|}{N^{\pi\left(u\right)-1}}+O\left(\frac{1}{N^{\pi\left(u\right)}}\right).
\]
\end{cor}

\subsection{Stable irreducible characters}

Arguably, the most important elements in the ring $\A$ of class functions
are the elements corresponding to families of irreducible characters
$\chi=\left\{ \chi_{N}\right\} _{N\ge N_{0}}$ ($\chi_{N}$ being
an irreducible character of $S_{N}$): these are precisely the characters
of \emph{stable} families of irreducible representations. For a partition
$\lambda=\left(\lambda_{1}\ge\ldots\ge\lambda_{\ell}>0\right)$, denote
$\left|\lambda\right|=\sum_{i=1}^{\ell}\lambda_{i}$, so $\lambda\vdash\left|\lambda\right|$.
For every $N\ge\left|\lambda\right|+\lambda_{1}$, consider the partition
\[
\lambda\cup\left\{ N-\left|\lambda\right|\right\} =\left(N-\left|\lambda\right|\ge\lambda_{1}\ge\ldots\ge\lambda_{\ell}\right)\vdash N.
\]
These partitions give rise to a family of irreducible characters $\chi=\left\{ \chi_{N}\right\} _{N\ge\left|\lambda\right|+\lambda_{1}}$,
one for every $N\ge\left|\lambda\right|+\lambda_{1}$. This family
corresponds to an element of $\A$ -- see Appendix \ref{sec:class-functions}.
Table \ref{tab:Some-families-of-irreps} describes the four ``simplest''
families of irreducible characters in $\A$.

\begin{table}
\begin{tabular}{ccccc}
\hline 
\noalign{\vskip2mm}
\multicolumn{1}{|c|}{Young Diagram of $\chi$} & \multicolumn{1}{c|}{$\lambda$} & \multicolumn{1}{c|}{Element of $\A$} & \multicolumn{1}{c|}{Poly in the $a_{t}$'s} & \multicolumn{1}{c|}{Dimension of $\chi_{N}$}\tabularnewline[2mm]
\hline 
\noalign{\vskip2mm}
\ytableausetup {mathmode, boxsize=2em} \begin{ytableau} ~ & ~ & ~ & ~ & \none[\dots] &  ~  \end{ytableau} & $\emptyset$ & 1 & $1$ & $1$\tabularnewline[2mm]
\noalign{\vskip2mm}
\ytableausetup {mathmode, boxsize=2em} \begin{ytableau} ~ & ~ & ~ & ~ & \none[\dots] &  ~ \\ ~ \end{ytableau} & $1$ & $\z_{1}-1$ & $a_{1}-1$ & $N-1$\tabularnewline[2mm]
\noalign{\vskip2mm}
\ytableausetup {mathmode, boxsize=2em} \begin{ytableau} ~ & ~ & ~ & ~ & \none[\dots] & ~ \\ ~ & ~ \end{ytableau} & $2$ & $\frac{\z_{1}^{~2}+\z_{2}}{2}-2\z_{1}$ & $\frac{a_{1}\left(a_{1}-3\right)}{2}+a_{2}$ & $\frac{N\left(N-3\right)}{2}$\tabularnewline[2mm]
\noalign{\vskip2mm}
\ytableausetup {mathmode, boxsize=2em} \begin{ytableau} ~ & ~ & ~ & ~ & \none[\dots] & ~ \\ ~ \\ ~ \end{ytableau} & $1,1$ & $\frac{\z_{1}^{~2}-\z_{2}}{2}-\z_{1}+1$ & $\frac{\left(a_{1}-1\right)\left(a_{1}-2\right)}{2}-a_{2}$ & $\frac{\left(N-1\right)\left(N-2\right)}{2}$\tabularnewline[2mm]
\end{tabular}\caption{\label{tab:Some-families-of-irreps}Some families of irreducible characters
belonging to $\protect\A$}
\end{table}

Denote the set of all such families of irreducible characters by \label{stable irreps}$\symirr$.
We may thus consider $\symirr$ as a subset of $\A$. For every $\chi\in\symirr$,
$\mathbb{E}_{w}\left[\chi\right]$ is defined for $N\ge N_{0}$ (or
for every $N\ge1$ if we consider the corresponding element of $\A$).
By orthogonality of irreducible characters, if $\chi\ne1,\z_{1}-1$
then $\left\langle \chi,1\right\rangle =\left\langle \chi,\xi_{1}-1\right\rangle =0$,
so
\begin{cor}
\label{cor:bound-for-irreducible-representations}Let $\chi\in\symirr$
so that $\chi\ne1,\xi_{1}-1$. Then, for non-powers $w\in\F$,

\[
\mathbb{E}_{w}\left[\chi\right]=O\left(\frac{1}{N^{\pi(w)}}\right).
\]
\end{cor}

In fact, the elements of $\symirr$ form, too, a linear basis of $\A$
-- see Proposition \ref{prop:irreps as linear basis of A}, and so
Corollary \ref{cor:bound-for-irreducible-representations} is \emph{equivalent}
to Theorem \ref{thm:Word-Measure-Character-Bound}. We conjecture
the following much stronger bound:
\begin{conjecture}
\label{conj:irreducible-characters}Let $\chi\in\symirr$. Then for
every $w\in\F$, 
\[
\mathbb{\mathbb{E}}_{w}\left[\chi\right]=O\left(\frac{1}{\left(\dim\chi\right)^{\pi(w)-1}}\right).
\]
\end{conjecture}

The dimension $\dim\chi$ is a polynomial function of $N$ (obtained
by substituting every $\xi_{j}$ in the corresponding polynomial in
$\A$ with $N$). \label{family of irreps has degree poly in N}The
degree of this polynomial is equal to the number of squares outside
the first row of the Young diagram, so if $\chi\neq1,\z_{1}-1$, the
degree is greater than $1$. Thus, Conjecture \ref{conj:irreducible-characters}
is stronger (for non-powers) than Corollary \ref{cor:bound-for-irreducible-representations}.
The conjecture holds for words of primitivity rank $1$, namely, for
proper powers: this follows from \cite{nica1994number} and from \cite[Section 4]{Linial2010}.
Another known special case of this conjecture is the commutator $\left[x,y\right]=xyx^{-1}y^{-1}$:
\label{simple commutator}indeed, $\pi\left(\left[x,y\right]\right)=2$
and already in 1896, Frobenius \cite{frobenius1896gruppencharaktere}
showed that $\mathbb{E}_{\left[x,y\right]}\left[\chi\right]=\frac{1}{\dim\chi}$
for every finite group $G$ and every irreducible character $\chi$
of $G$. Moreover, given two class functions $f_{1},f_{2}\colon G\to\mathbb{R}$
and an irreducible character $\chi$ of $G$, a simple application
of Schur's Lemma gives $\left\langle f_{1}\ast f_{2},\chi\right\rangle _{G}=\frac{\langle f_{1},\chi\rangle_{G}\langle f_{2},\chi\rangle_{G}}{\dim\chi}$.
If $w_{1}\in\F\left(x_{1},\ldots,x_{k}\right),w_{2}\in\F\left(x_{k+1},\ldots,x_{r}\right)$
are two words generated by disjoint sets of letters, then the $w_{1}w_{2}$-measure
on $G$ is the convolution of the $w_{1}$- and the $w_{2}$-measures,
and by the corollary of Schur's Lemma, $\mathbb{E}_{w_{1}w_{2}}\left[\chi\right]=\frac{\mathbb{E}_{w_{1}}\left[\chi\right]\cdot\mathbb{E}_{w_{2}}\left[\chi\right]}{\dim\chi}$.
On the other hand, $\pi(w_{1}w_{2})=\pi(w_{1})+\pi(w_{2})$ \cite[Lemma 6.8]{Puder2014}.
Hence, knowing the conjecture for two such words implies the claim
for their product. In particular, this implies the conjecture for
every product of disjoint commutators and powers, that is, for every
word of the form 
\[
w=\left[x_{1},y_{1}\right]\cdot[x_{2},y_{2}]\cdot\ldots\cdot[x_{r},y_{r}]\cdot z_{1}^{k_{1}}\cdot\ldots\cdot z_{m}^{k_{m}}\in\F\left(x_{1},\ldots,x_{r},y_{1},\ldots,y_{r},z_{1},\ldots,z_{m}\right),
\]
with $r,m\in\mathbb{Z}_{\ge0}$ and $k_{1},\ldots,k_{m}\in\mathbb{Z}$.
See Section \ref{subsec:Similar-phenomena-in} for generalizations
of Conjecture \ref{conj:irreducible-characters} for other families
of groups.

\subsection{Expansion of random Schreier graphs\label{subsec:Expansion-of-random}}

As an application of our results, we prove expansion properties of
random Schreier graphs of the symmetric group. If $G$ is a $d$-regular
graph on $n$ vertices, its adjacency matrix has eigenvalues 
\[
d=\mu_{1}\ge\mu_{2}\ge\ldots\ge\mu_{n}\ge-d,
\]
with $\mu_{1}=d$ considered as a trivial eigenvalue. Denote by $\mu\left(G\right)$
the largest absolute value of a non-trivial eigenvalue of the graph
$G$. Namely, $\mu\left(G\right)=\max\left(\mu_{2},-\mu_{n}\right)$.
An expander graph is a sparse graph with high connectivity. One standard
way to measure expansion is with $\mu\left(G\right)$ -- smaller
$\mu\left(G\right)$ implies better expansion (see \cite{hoory2006expander}
for a survey). Here we study random \emph{Schreier graphs} of the
groups $S_{N}$.
\begin{defn}
Given an action of a group $G$ on a set $X$, and a tuple $g_{1},\ldots,g_{r}\in G$,
the corresponding Schreier graph is the $2r$-regular graph with vertex
set $X$ and an edge $x\sim g_{i}\left(x\right)$ for every $x\in X$
and $i\in\left[r\right]$. Note that we allow multiple edges as well
as loops.
\end{defn}

The group $S_{N}$ acts naturally on the set of $s$-tuples of distinct
elements in $[N]\defi\left\{ 1,\ldots,N\right\} $. Choosing uniformly
at random a tuple of permutations $\sigma_{1},\ldots,\sigma_{r}\in S_{N}$,
consider the (random) Schreier graph corresponding to this action.
Denoting $d=2r$, this is a random $d$-regular graph on $\left(N\right)_{s}\defi N\cdot\ldots\cdot(N-s+1)$
vertices. The fact that for a fixed $s$, this family of random $d$-regular
graphs has a uniform spectral gap with high probability is known since
the work \cite{friedman1998action}. Theorem 2.1 therein states that
for every $\varepsilon>0$,
\[
\mu\left(G\right)\le\left(1+\varepsilon\right)d\left(\frac{2\sqrt{d-1}}{d}\right)^{1/\left(s+1\right)}
\]
asymptotically almost surely (namely, with probability tending to
1 as $N\to\infty$; a.a.s.~in short).

For $s=1,2$ much stronger bounds are known. Friedman \cite{Fri08}
famously proved a conjecture of Alon and showed that for $s=1$, a
random $d$-regular graph in this model is nearly Ramanujan in the
strong sense that for every $\varepsilon>0$, a.a.s.~$\mu(G)<2\sqrt{d-1}+\varepsilon$.
More recently, following Bordenave's new proof of Friedman's theorem
\cite{bordenave2020new}, Bordenave and Collins \cite{bordenave2019eigenvalues}
proved the same result for Schreier graphs on pairs of elements, namely,
for $s=2$. It is conjectured that the same result holds for any fixed
value of $s$ --- see, for instance, \cite[Conjecture 1.6]{Rivin19}\footnote{It is plausible that the method of proof in \cite{bordenave2019eigenvalues}
can be used to prove this conjecture in full.}. This conjecture, and progress towards it, may serve as steps towards
answering an even harder question: whether or not random Cayley graphs
of $S_{N}$ (namely, Cayley graphs with respect to a random set of
elements of some fixed size) are a.a.s.~nearly Ramanujan. In fact,
it is not even known whether these random Cayley graphs are uniformly
expanders. See \cite{breuillard2018expansion} for a recent survey.

In \cite{Puder2015}, using a different approach, the special case
of the action of $S_{N}$ on $[N]$ was studied. It was proved that
a.a.s.~$\mu(G)<2\sqrt{d-1}+0.84.$ The same approach was later improved
in \cite{friedman2020nonbacktracking} to give a.a.s.~

\begin{equation}
\mu(G)<2\sqrt{d-1}\cdot\exp\left(\frac{2}{e^{2}\left(d-1\right)}\right)<2\sqrt{d-1}+\frac{0.6}{\sqrt{d-1}}.\label{eq:Friedman-Puder bound}
\end{equation}
Here we generalize this method and prove the following bound for all
values of $s$:
\begin{thm}
\label{thm:Schreier-Graphs-Near-Ramanujan}Fix $s,r\in\mathbb{Z}_{\ge1}$
and let $d=2r$. Let $G$ be a random $d$-regular Schreier graph
depicting the action of $r$ random permutations on $s$-tuples of
distinct elements from $\left[N\right]$. Then a.a.s.~as $N\to\infty$,

\begin{equation}
\mu\left(G\right)<2\sqrt{d-1}\cdot\exp\left(\frac{2s^{2}}{e^{2}\left(d-1\right)}\right).\label{eq:schreier graphs bound}
\end{equation}
For fixed $s$ and growing $d$, this bound is 
\[
\mu\left(G\right)<2\sqrt{d-1}+\frac{4s^{2}}{e^{2}\sqrt{d-1}}+O\left(\frac{s^{4}}{\left(d-1\right)^{3/2}}\right).
\]
In particular, for every fixed $s$ and every $\varepsilon>0$, if
$d$ is large enough, \eqref{eq:schreier graphs bound} gives that
a.a.s.~$\mu\left(G\right)<2\sqrt{d-1}+\varepsilon$.
\end{thm}

\begin{rem}
The bound \eqref{eq:schreier graphs bound} is achieved by optimizing
our method for large values of $r$ (and $d$) and fixed $s$. For
specific, small values of r, the method gives better bounds. For example,
for $r=2$ (so $d=4$) and $s=1$, \eqref{eq:schreier graphs bound}
gives a bound of $\approx3.735$, while the method can actually yield
a bound of $\approx3.596$ (compare with $2\sqrt{3}\approx3.464$).
\end{rem}

\begin{rem}
Conjecture \ref{conj:irreducible-characters}, if true, yields that
the bound \eqref{eq:Friedman-Puder bound} holds as is also for the
Schreier graphs in Theorem \ref{thm:Schreier-Graphs-Near-Ramanujan},
namely, a bound which is independent of $s$. See Remark \ref{rem:how conjecture implies strong expansion, details}
for more details.
\end{rem}

\subsection{Overview of the paper}

\subsubsection*{Outline of the proof of Theorem \ref{thm:Word-Measure-Character-Bound}}

Let $1\ne w\in\F$ be a non-power. Consider the function $\mathbb{E}_{w}\left[\qak\right]\colon\mathbb{Z}_{\ge1}\to\mathbb{Q}$
from Theorem \ref{thm:Word-Measure-Character-Bound}. As this function
is invariant under replacing $w$ with a conjugate (automorphic image,
in fact), we may assume that $w$ is cyclically reduced. A natural
approach to study this function is to consider $\alpha_{1}+\ldots+\alpha_{k}$
cycles describing powers of $w$: $\alpha_{1}$ copies of $w$, $\alpha_{2}$
copies of $w^{2}$ and so on. This graph is denoted $\Gamma_{\ak}^{w}$
(see Example \ref{exa:main thm in terms of Phi}) and illustrated
in Figure \ref{fig:example of multi core graph}. Then $\mathbb{E}_{w}\left[\qak\right]\left(N\right)$
counts the average number of labelings of the vertices of $\Gamma_{\ak}^{w}$
which agree with the independent uniformly random permutations represented
by the $r$ generators of $\F$ ($x$ and $y$ in Figure \ref{fig:example of multi core graph}).
The fact that $\mathbb{E}_{w}\left[\qak\right]\left(N\right)$ is
given by a function in $\mathbb{Q}\left(N\right)$ (Corollary \ref{cor:Phi, L, R, C all rational})
follows by a simple argument (see also \cite[Section 5]{Puder2014}
for a more straightforward explanation of the technique).

It follows from (the arguments in) \cite{nica1994number,Linial2010}
that $\mathbb{E}_{w}\left[\qak\right]=\mathbb{E}_{\mathrm{unif}}\left[\qak\right]+O\left(\frac{1}{N}\right)$.
Our goal is to give a more precise estimate of the $O\left(\frac{1}{N}\right)$
term. First, we imitate the proof in \cite{PP15} of Theorem \ref{thm:PP15}.
This part starts with generalizing the function $\mathbb{E}_{w}\left[\qak\right]$
as follows. Consider the graph-moprhism $\eta_{\ak}^{w}\colon\Gamma_{\ak}^{w}\to X_{B}$
where $X_{B}$ is the bouquet (see Figure \ref{fig:example of multi core graph}):
$\mathbb{E}_{w}\left[\qak\right]\left(N\right)$ is then also equal
to the average number of lifts of this morphism to a random $N$-covering
of $X_{B}$ (Proposition \ref{prop:geometric meaning for Phi}). The
same average can be defined to any morphism $\eta$ between finite
graphs: and this is the essence of the map $\Phi_{\eta}$ in Definition
\ref{def:Phi}. In particular, $\Phi_{\eta_{\ak}^{w}}=\mathbb{E}_{w}\left[\qak\right]$.
The graphs we consider here have directed edges labeled by a fixed
basis $B$ of $\F$ and each connected component is a Stallings core
graph: we call such graphs \emph{multi core graphs} -- see Section
\ref{sec:Multi-core-graphs} for the precise definition. They correspond
to multisets of conjugacy classes of f.g.~subgroups of $\F$.

As in \cite{PP15}, we introduce Möbius inversions of the function
$\Phi$ (Definitions \ref{def:B-surjective Mob inversions} and \ref{def:alg-decompositions}).
While the left inversion $L^{B}$ is quite natural (corresponds to
\emph{injective }lifts of the morphism rather than arbitrary lifts
in $\Phi$ -- see Proposition \ref{prop:rational-function-for-L}),
the other two inversions $R^{B}$ and $C^{B}$ are more mysterious.
However, the crux of introducing these Möbius inversions is that their
analysis proves some non-trivial cancellations in the computation
of $\Phi$, culminating in Theorem \ref{thm:Phi-approximation}. The
flow of ideas in this part is very similar to \cite{PP15}, and the
reader is advised to read the overview there \cite[Section 2]{PP15}. 

Explaining the content of Theorem \ref{thm:Phi-approximation} requires
first to describe the notion of \emph{algebraic} morphism. For two
free groups $H\le J$, we say that $J$ is an algebraic extension
of $H$ if there is no intermediate subgroup of $J$ which is a proper
free factor of $J$. This notion was coined in \cite{miasnikov2007algebraic}
and it gives a notion of algebraicity of morphisms between \emph{connected}
core graphs. In Section \ref{sec:Free-and-algebraic} we generalize
this notion to morphisms between general multi core graphs. In Sections
\ref{sec:Free-and-algebraic} and \ref{sec:-Surjective-morphisms}
we also introduce pullbacks in the category of multi core graphs,
consider $B$-surjective morphisms and define norms of morphisms --
all these are required to prove Theorem \ref{thm:Phi-approximation}.
Many of the definitions around the category of multi core graphs are
not obvious and we think this part of the paper may be of independent
interest. 

Theorem \ref{thm:Phi-approximation} considers an arbitrary morphism
$\eta\colon\Gamma\to\Delta$ of multi core graphs. It follows from
Theorem \ref{thm:properties of algebraic morphisms}\eqref{enu:alg is surjective}
that there are finitely many decompositions $\Gamma\stackrel{_{\eta_{1}}}{\longrightarrow}\Sigma\stackrel{_{\eta_{2}}}{\longrightarrow}\Delta$
of $\eta$ with $\eta_{1}$ algebraic. We let $\chimax\left(\eta\right)$
denote the maximal Euler characteristic $\chi\left(\Sigma\right)$
of $\Sigma$ in such a decomposition with $\eta_{1}$ algebraic and
non-isomorphism, and $\crit\left(\eta\right)$ denote the set of such
decompositions with $\chi\left(\Sigma\right)$ maximal -- see Definition
\ref{def:chi-max}. Then Theorem \ref{thm:Phi-approximation} states
that
\begin{equation}
\Phi_{\eta}\left(N\right)=N^{\chi\left(\Gamma\right)}+\left|\crit(\eta)\right|\cdot N^{\chimax\left(\eta\right)}+O\left(N^{\chimax(\eta)-1}\right).\label{eq:content of naive Thm about Phi}
\end{equation}
When $\eta$ is the morphism from $\Gamma_{1}^{w}$, the core graph
of $\left\langle w\right\rangle $, to the bouquet $X_{B}$, \eqref{eq:content of naive Thm about Phi}
is precisely Theorem \ref{thm:PP15} -- the main result of \cite{PP15}.
However, when applied to the morphism $\eta_{\ak}^{w}\colon\Gamma_{\ak}^{w}\to X_{B}$
with $\sum i\alpha_{i}\ge2$, \eqref{eq:content of naive Thm about Phi}
does not yield anything new: it only recovers earlier results from
\cite{nica1994number,Linial2010}. Indeed, in these cases there is
an algebraic morphism $\Gamma_{\ak}^{w}\to\Gamma_{1}^{w}$, so $\chimax\left(\eta_{\ak}^{w}\right)=0$,
and \eqref{eq:content of naive Thm about Phi} only gives information
about the free term of the Laurent expansion of $\mathbb{E}_{w}\left[\qak\right]$. 

Section \ref{sec:Proof-of-Theorem}, using arguments from combinatorial
and geometric group theory, strengthens \eqref{eq:content of naive Thm about Phi}
for the morphisms $\eta_{\ak}^{w}$ and completes the proof of Theorem
\ref{thm:Word-Measure-Character-Bound}. First, we handle separately
algebraic morphisms from $\Gamma_{\ak}^{w}$ with codomain of Euler
characteristic $0$, and define $\chiak\left(w\right)$ to be the
maximal \emph{negative} Euler characteristic of (the codomain of)
an algebraic morphism from $\Gamma_{\ak}^{w}$, and $\crit_{\ak}\left(w\right)$
accordingly (see Definition \ref{def:chi-max for multiset of cycles}).
Theorem \ref{thm:character-first-bound}, which follows readily from
our analysis of the algebraic Möbius inversions of $\Phi$, gives
the following estimate for $\mathbb{E}_{w}\left[\qak\right]$: 
\[
\Phi_{\eta_{\ak}^{w}}\left(N\right)=\uexp\left[\qak\right]+\left|\critak\left(w\right)\right|\cdot N^{\chiak(w)}+O\left(N^{\chiak(w)-1}\right).
\]
It remains to show that $\chiak\left(w\right)=1-\pi\left(w\right)$
and that $\left|\crit_{\ak}\left(w\right)\right|=\left\langle \qak,\xi_{1}-1\right\rangle \cdot\left|\crit\left(w\right)\right|$.
The former equality is the content of the short Proposition \ref{prop:joint-primitivity-rank}.
The latter equality is the content of Section \ref{subsec:crit ak =00003D crit times C}.
It is quite straightforward to see that every critical subgroup of
$w$ corresponds to $\left\langle \qak,\xi_{1}-1\right\rangle $ distinct
critical morphisms in $\crit_{\ak}\left(w\right)$. The hard part
is to show that these are the \emph{only} elements of $\crit_{\ak}\left(w\right)$.
To this end, we use a ``dependence'' theorem of Louder \cite{louder2013scott}
-- Theorem \ref{thm:Louder} below -- concerning free quotients
of certain graphs of groups, from which we conclude that an algebraic
morphism from $\Gamma_{\ak}$ which, roughly, \emph{does not} factor
non-trivially through $\Gamma_{1}^{w}$ must be of Euler characteristic
strictly smaller then $1-\pi\left(w\right)$. In fact, this type of
dependence theorems of Louder (see also Louder and Wilton \cite{louder2018negative})
fits so well into our proof here, that it suggests a deeper connection
between these theorems and Conjecture \ref{conj:irreducible-characters}.
See also Section \ref{subsec:Some-remarks} for more points of intersection
between this paper and works of Louder and Wilton.

\subsubsection*{Paper organization}

We end the introduction in Section \ref{subsec:Similar-phenomena-in},
which describes some fascinating evidence that the phenomena we prove
and the phenomena we conjecture regarding the symmetric group are,
in fact, more universal. We stress that Section \ref{subsec:Similar-phenomena-in}
is completely orthogonal to the remaining sections and the reader
interested solely in our proven results may safely skip it. 

Sections \ref{sec:From-words-to-subgroups} through \ref{sec:-Surjective-morphisms}
and Appendix \ref{sec:norm-of-morphism-proof} are devoted to the
study of multi core graphs. We begin with the equivalent category
of multisets of conjugacy classes of finitely generated subgroups
of $\F$, introduced in Section \ref{sec:From-words-to-subgroups}.
Section \ref{sec:Multi-core-graphs} introduces the geometric counterpart
of the latter: multi core graphs and their morphisms, and also the
above-mentioned function $\Phi_{\eta}$. Section \ref{sec:Free-and-algebraic}
defines free and algebraic morphisms of multi core graphs, as well
as pullbacks (also known as fiber products). Section \ref{sec:-Surjective-morphisms}
and Appendix \ref{sec:norm-of-morphism-proof} deal with surjective
morphisms and with the norms of morphisms, generalizing analogous
concepts from \cite{Puder2014}. Section \ref{sec:M=0000F6bius-inversions}
introduces the Möbius inversions of the function $\Phi_{\eta}$ and
proves the above mentioned Theorem \ref{thm:Phi-approximation} --
a naive analogue of Theorem \ref{thm:PP15}. In Section \ref{sec:Proof-of-Theorem}
we strengthen Theorem \ref{thm:Phi-approximation} and prove our main
result, Theorem \ref{thm:Word-Measure-Character-Bound}.

Finally, Section \ref{sec:Expanders} contains the proof of Theorem
\ref{thm:Schreier-Graphs-Near-Ramanujan} about random Schreier graphs
of $S_{N}$, in Appendix \ref{sec:conjugacy-seperability} we use
our results in order to obtain a new and simple proof of the well-known
conjugacy separability of free groups, and Appendix \ref{sec:class-functions}
develops more formally the ring $\A$ of class functions introduced
above. We end with a glossary of the main notation used along the
paper.

\subsection{Similar phenomena in other families of groups\label{subsec:Similar-phenomena-in}}

Some of the phenomena discussed above regarding word measures on the
symmetric group have parallels, at least partially, in other families
of groups. The mere fact that for every natural family of class functions
$f$ and every word $w\in\F$, the expectation $\mathbb{E}_{w}\left[f\right]$
is a rational function in the running parameter (usually $N$) of
the family of groups, is true not only for symmetric groups \cite{nica1994number,Linial2010},
but also for unitary groups\footnote{Given a compact group $G$, the $w$-measure on $G$ is the push-forward,
via the word map $w\colon G^{r}\to G$, of the Haar measure on $G^{r}$.} \cite{Radulescu06,MSS07}, for orthogonal and compact symplectic
groups \cite{MP-On}, for natural families of class functions of $\mathrm{GL}_{N}\left(\mathbb{F}_{q}\right)$,
where $\mathbb{F}_{q}$ is a fixed finite field \cite{West21}, and
for generalized symmetric groups \cite{MP-surface-words,Ordo20,shomroni2021wreath}.

However, it seems there are deeper universal phenomena which are common
to all these families of groups. In each of the above-mentioned families,
there are also natural families of irreducible characters defined
analogously to those in $S_{N}$: these are the characters of \emph{stable
}families of irreducible representations in the sense of \cite{church2013representation}.
We elaborate a bit below in the sequel of this subsection. It seems
that the primitivity rank of a word plays a role in all the above
mentioned families of groups. More precisely, we conjecture the following
generalization of Conjecture \ref{conj:irreducible-characters}:
\begin{conjecture}
\label{conj:universal}Let $G=\left\{ G\left(N\right)\right\} _{N}$
be a natural family of groups as those mentioned above, and let $\chi=\left\{ \chi_{N}\right\} _{N\ge N_{0}}$
be a stable irreducible character with $\chi_{N}\in\widehat{G\left(N\right)}$.
Then for any word $w\in\F$, as $N\to\infty$,
\[
\mathbb{E}_{w}\left[\chi\right]=O\left(\left(\dim\chi\right)^{1-\pi\left(w\right)}\right).
\]
Here the implied constant may depend on $w$, on $G$ and on $\chi$.
\end{conjecture}

As stated, this conjecture may sound a bit vague, but it has a very
concrete meaning in each of the above mentioned families of groups.
Before elaborating on what this means for each of these families,
we mention that the conjecture is trivial for $w=1$, and is true
for proper powers, namely, $\mathbb{E}_{w}\left[\chi_{N}\right]=O\left(1\right)$,
in all cases studied in the works mentioned above. The conjecture
is also true for $w=\left[x,y\right]$ by \cite{frobenius1896gruppencharaktere},
and thus also for any word of the form $w=\left[x_{1},y_{1}\right]\cdot[x_{2},y_{2}]\cdot\ldots\cdot[x_{r},y_{r}]\cdot z_{1}^{k_{1}}\cdot\ldots\cdot z_{m}^{k_{m}}$,
as explained on page \pageref{simple commutator}. In fact, if true,
Conjecture \ref{conj:universal} may be seen as a generalization of
Frobenius' result about the commutator word $\left[x,y\right]$.

\subsubsection*{Unitary groups}

Consider the unitary groups $U\left(N\right)$. By analogy to $\xi_{k}$
from \eqref{eq:xi k}, define $\zeta_{k}\colon U\left(N\right)\to\mathbb{C}$
by $\zeta_{k}\left(B\right)=\mathrm{tr}\left(B^{k}\right)$, only
here $k\in\mathbb{Z}$ may also be negative, and define $\A^{U}=\mathbb{Q}\left[\ldots,\zeta_{-2},\zeta_{-1},\zeta_{1},\zeta_{2},\ldots\right]$.
In the case of $U\left(N\right)$, the natural families of irreducible
characters referred to in Conjecture \ref{conj:universal} are those
corresponding to elements in $\A^{U}$. In terms of highest weight
vectors\footnote{For the theory of highest weight vectors see, e.g., \cite[Chapters 24,25]{bump2004lie}.},
one starts with an arbitrary highest weight vector of length $N_{0}$
and adds $N-N_{0}$ zeros to obtain a character of $U\left(N\right)$
for all $N\ge N_{0}$.

The expected value of monomials in the $\zeta_{k}$'s under word measures
is the main object of study in \cite{MP-Un}, where these values are
given a topological interpretation in terms of surfaces and mapping
class groups. In particular, the defining character of $U\left(N\right)$
is $\zeta_{1}$, and it satisfies \cite[Corollary 1.8]{MP-Un}
\begin{equation}
\mathbb{E}_{w}\left[\zeta_{1}\right]=\mathbb{E}_{w}\left[\mathrm{tr}\left(B\right)\right]=O\left(N^{1-2\cdot\mathrm{cl}\left(w\right)}\right),\label{eq:Trw in U(N)}
\end{equation}
where $\mathrm{cl}\left(w\right)$ is the commutator length of $w$:
\[
\mathrm{cl}\left(w\right)\defi\min\left\{ g\,\middle|\,w=\left[u_{1},v_{1}\right]\cdots\left[u_{g},v_{g}\right]~\mathrm{with}~u_{i},v_{i}\in\F\right\} .
\]
(If $w\notin\left[\F,\F\right]$, we say that $\mathrm{cl}\left(w\right)=\infty$.)
Note that if $w=\left[u_{1},v_{1}\right]\cdots\left[u_{g},v_{g}\right]$,
then $w$ is non-primitive in the subgroup $\left\langle u_{1},v_{1},\ldots,u_{g},v_{g}\right\rangle $,
hence $\pi\left(w\right)\le2g$. Thus 
\[
\pi\left(w\right)\le2\cdot\mathrm{cl}\left(w\right),
\]
and \eqref{eq:Trw in U(N)} yields that Conjecture \ref{conj:universal}
holds for the irreducible character $\zeta_{1}$.

Moreover, there is a nice relation between general \emph{polynomial}
characters of $U\left(N\right)$ and an important invariant of words
called \emph{stable commutator length}, which is defined as 
\[
\mathrm{scl}\left(w\right)\defi\lim_{m\to\infty}\frac{\mathrm{cl}\left(w^{m}\right)}{m}.
\]
Indeed, from \cite[Theorem 1.7]{MP-Un} it follows that for every
$w\in\F$, $\ell>0$ and $j_{1},\ldots,j_{\ell}\in\mathbb{Z}$,
\begin{equation}
\mathbb{E}_{w}\left[\zeta_{j_{1}}\cdots\zeta_{j_{\ell}}\right]=\tr_{w^{j_{1}},\ldots,w^{j_{\ell}}}\left(N\right)=O\left(N^{\chi_{\max}\left(w^{j_{1}},\ldots,w^{j_{\ell}}\right)}\right),\label{eq:U(n) main thm}
\end{equation}
where $\chi_{\max}\left(w^{j_{1}},\ldots,w^{j_{\ell}}\right)$ is
the maximal Euler characteristic of a surface admissible for $w^{j_{1}},\ldots,w^{j_{\ell}}$
\cite[Definition 1.2]{MP-Un}. This is all very much related to Calegari's
works on stable commutator length. First, \cite[Lemma 2.6]{CALRATIONAL}
yields that if a surface $\Sigma$ is admissible for $w^{j_{1}},\ldots,w^{j_{\ell}}$,
then 
\begin{equation}
\scl\left(w\right)\le\frac{-\chi\left(\Sigma\right)}{2\left|j_{1}+\ldots+j_{\ell}\right|},\label{eq:scl as upper bound}
\end{equation}
so $\chi\left(\Sigma\right)\le-2\cdot\scl\left(w\right)\cdot\left|j_{1}+\ldots+j_{\ell}\right|$,
which, combined with \eqref{eq:U(n) main thm}, gives 
\begin{equation}
\mathbb{E}_{w}\left[\zeta_{j_{1}}\cdots\zeta_{j_{\ell}}\right]=O\left(N^{-2\cdot\scl\left(w\right)\cdot\left|j_{1}+\ldots+j_{\ell}\right|}\right).\label{eq:scl bound for monomials}
\end{equation}
Now consider the subring $\A_{\mathrm{poly}}^{U}\defi\mathbb{Q}\left[\zeta_{1},\zeta_{2},\ldots\right]$
of $\A^{U}$ generated by traces of positive powers of the matrices
in $U\left(N\right)$. The irreducible characters corresponding to
elements of $\A_{\mathrm{poly}}^{U}$ are families of characters of
\emph{polynomial} irreducible representations of $U\left(N\right)$.
In the language of highest weight vectors, these are irreducible characters
with non-negative weights. By the representation theory of $U\left(N\right)$,
every such character corresponds to some partition $\mu$ (these are
the positive weights). Let $\eta_{\mu}\in\A_{\mathrm{poly}}^{U}$
be the family of polynomial irreducible characters corresponding to
the partition $\mu$. There is a simple formula expressing $\eta_{\mu}$
as a linear combination of the monomials in $\A_{\mathrm{poly}}^{U}$.
For a partition $\lambda=\left(1^{\alpha_{1}}2^{\alpha_{2}}\ldots k^{\alpha_{k}}\right)$,
define the monomial $\zeta_{\lambda}\defi\zeta_{1}^{\alpha_{1}}\cdots\zeta_{k}^{\alpha_{k}}$.
In addition, let
\[
z_{\lambda}\defi\prod_{r}r^{\alpha_{r}}\cdot\alpha_{r}!.
\]
Note that $\frac{1}{z_{\lambda}}$ is the probability that a random
permutation in $S_{\left|\lambda\right|}$ has cycle structure $\lambda$.
Finally, given two partitions $\lambda,\rho$ of $m$, denote the
value of $\chi^{\rho}$ (the irreducible character of $S_{m}$ corresponding
to $\rho$) on a permutation with cycle structure $\lambda$ by $\chi^{\rho}\left(\lambda\right)$.
The formula for the polynomial character $\eta_{\mu}$ is 
\begin{equation}
\eta_{\mu}=\sum_{\lambda\vdash\left|\mu\right|}\frac{1}{z_{\lambda}}\chi^{\mu}\left(\lambda\right)\zeta_{\lambda}\label{eq:formula for schur poly}
\end{equation}
(this is basically a special case of \cite[Corollary 7.17.5]{Stanley1999enumerative2}).
For example, $\eta_{1,1,1}=\frac{1}{6}\zeta_{1}^{~3}-\frac{1}{2}\zeta_{2}\zeta_{1}+\frac{1}{3}\zeta_{3}$.
In particular, \eqref{eq:formula for schur poly} yields that $\dim\eta_{\mu}$
is a polynomial in $N$ of degree $\left|\mu\right|$. We conclude
from \eqref{eq:scl bound for monomials} that for every family $\eta_{\mu}$
of \emph{polynomial} irreducible characters,
\begin{equation}
\mathbb{E}_{w}\left[\eta_{\mu}\right]=O\left(N^{-2\cdot\scl\left(w\right)\cdot\left|\mu\right|}\right)=O\left(\left(\dim\eta_{\mu}\right)^{-2\cdot\scl\left(w\right)}\right).\label{eq:bound for poly chars of U(n)}
\end{equation}

More strikingly, in the same paper \cite{CALRATIONAL}, Calegari also
proves that every word $w$ admits \emph{extremal surfaces}: these
are surfaces admissible for $w^{j_{1}},\ldots,w^{j_{\ell}}$ for some
$\ell>0$ and $j_{1},\ldots,j_{\ell}>0$, such that there is equality
in \eqref{eq:scl as upper bound}. In particular, $\scl\left(w\right)$
is rational for every $w$, which is the main result of \cite{CALRATIONAL}.
As explained in \cite[Section 5.1]{MP-Un}, for such values of $j_{1},\ldots,j_{\ell}>0$
admitting extremal surfaces, \eqref{eq:U(n) main thm} becomes 
\[
\mathbb{E}_{w}\left[\zeta_{j_{1}},\ldots,\zeta_{j_{\ell}}\right]=\#\left\{ \mathrm{extremal~surfaces}\right\} \cdot N^{-2\cdot\scl\left(w\right)\left(j_{1}+\ldots+j_{\ell}\right)}\left(1+O\left(N^{-2}\right)\right).
\]
Now consider $\eta_{k}$, the irreducible polynomial character of
$U\left(N\right)$ corresponding to the partition $\left(k\right)$.
In this case, $\chi^{\left(k\right)}$ is the trivial character of
$S_{k}$, and so \eqref{eq:formula for schur poly} becomes
\[
\eta_{k}=\sum_{\lambda\vdash k}\frac{1}{z_{\lambda}}\zeta_{\lambda}.
\]
Because the coefficients here are all positive, the positive contribution
of extremal surfaces to $\mathbb{E}_{w}\left[\eta_{k}\right]$ cannot
be balanced out. So, if $w$ admits extremal surfaces with $j_{1}+\ldots+j_{\ell}=k$,
then\footnote{We use the notation $f=\Theta\left(g\right)$ if these are two functions
of $N\in\mathbb{Z}_{\ge1}$ satisfying $f=O\left(g\right)$ and $g=O\left(f\right)$.} 
\begin{equation}
\mathbb{E}_{w}\left[\eta_{k}\right]=\Theta\left(\left(\dim\eta_{k}\right)^{-2\cdot\scl\left(w\right)}\right).\label{eq:sym power of std}
\end{equation}
From \eqref{eq:bound for poly chars of U(n)} and \eqref{eq:sym power of std}
we conclude that in the case of families of \emph{polynomial }irreducible
characters of $U\left(N\right)$, Conjecture \ref{conj:universal}
is equivalent to the following one.
\begin{conjecture}
\label{conj:scl and pi}For any $w\in\F$,
\[
\pi\left(w\right)\le2\cdot\scl\left(w\right)+1.
\]
\end{conjecture}

This conjecture was verified numerically for various words -- see,
for instance, \cite[Proposition 4.4]{cashen2020short}. In fact, Heuer
arrived to the exact same conjecture independently \cite[Conjecture 6.3.2]{heuer2019constructions},
based entirely on computer experiments!

We stress that \cite{MP-Un} does not provide such sharp bounds for
general, non-polynomial, characters of $U\left(N\right)$. For example,
the irreducible character with highest weight vector $\left(1,0,\ldots,0,-1\right)$
is of dimension $N^{2}-1$ and is equal to $\zeta_{1}\zeta_{-1}-1$.
One can infer from the analysis of admissible surfaces of Euler characteristic
zero that for non-powers, $\mathbb{E}_{w}\left[\xi_{1}\xi_{-1}-1\right]=O\left(N^{-2}\right)$.
Yet Conjecture \ref{conj:universal} says, in this case, that it should
be of order $O\left(N^{2\left(1-\pi\left(w\right)\right)}\right)$,
which is open when $\pi\left(w\right)\ge3$.

\subsubsection*{Orthogonal and symplectic groups}

In the case of the orthogonal group $\mathrm{O}\left(N\right)$ and
compact symplectic group $\mathrm{Sp}\left(N\right)$, the defining
standard representation ($N$-dimensional in the case of $\mathrm{O}\left(N\right)$,
$2N$-dimensional for $\mathrm{Sp}\left(N\right)$) has real trace,
and for a matrix $B$ in the defining representation, $\mathrm{tr}\left(B^{-k}\right)=\mathrm{tr}\left(B^{k}\right)$.
So here the ring of class functions is $\mathbb{Q}\left[\zeta_{1},\zeta_{2},\ldots\right]$
with $\zeta_{k}\left(B\right)\defi\mathrm{tr}\left(B^{K}\right)$.
The paper \cite{MP-On} studies monomials in the $\zeta_{k}$'s and
describes their expected value under word measures in terms of, again,
surfaces and mapping class groups. There is one case where the general
result there translates into a concrete algebraic bound: the standard
character $\zeta_{1}$. Corollary 1.10 in \cite{MP-On} states that
for both $\mathrm{O}\left(N\right)$ and $\mathrm{Sp}\left(N\right)$,
\[
\mathbb{E}_{w}\left[\zeta_{1}\right]=O\left(N^{1-\min\left(\mathrm{sql}\left(w\right),2\cdot\mathrm{cl}\left(w\right)\right)}\right),
\]
where 
\[
\mathrm{sql}\left(w\right)\defi\min\left\{ g\,\middle|\,w=u_{1}^{~2}\cdots u_{g}^{2}~\mathrm{with}~u_{i}\in\F\right\} .
\]
As argued above, this shows that Conjecture \ref{conj:universal}
holds for this character. We do not have significant evidence towards
conjecture \ref{conj:universal} in the case of other characters.

\subsubsection*{Generalized symmetric group}

Consider either the groups $\left\{ C_{m}\wr S_{N}\right\} _{N}$
where $C_{m}$ is a fixed cyclic group of order $m\ge2$, or $\left\{ S^{1}\wr S_{N}\right\} _{N}$
where $S^{1}=\nicefrac{\mathbb{R}}{\mathbb{Z}}$. One can define here
too natural families of irreducible characters. The standard character
$\zeta_{1}$, that of the standard $N$-dimensional representation,
is irreducible. In \cite[Theorem 1.11]{MP-surface-words}, it is shown
that 
\[
\mathbb{E}_{w}\left[\zeta_{1}\right]=\begin{cases}
D_{~w}^{m}\cdot N^{\chi_{m}\left(w\right)}+O\left(N^{\chi_{m}\left(w\right)-1}\right) & \mathrm{if}~G\left(N\right)=C_{m}\wr S_{N}\\
D_{~w}^{\infty}\cdot N^{\chi_{\infty}\left(w\right)}+O\left(N^{\chi_{\infty}\left(w\right)-1}\right) & \mathrm{if}~G\left(N\right)=S^{1}\wr S_{N}.
\end{cases}
\]
Here, $\chi_{m}\left(w\right)$ is the maximal Euler characteristic\footnote{$\chi\left(H\right)=1-\rk H$.}
of a subgroup $H\le\F$ such that\linebreak{}
$w\in\ker\left(H\twoheadrightarrow C_{m}^{~\rk H}\right)$, and $D_{~w}^{m}$
is the number of such subgroups of maximal Euler characteristic. Similarly,
$\chi_{\infty}\left(w\right)$ is the maximal Euler characteristic
of a subgroup $H\le\F$ such that $w\in\left[H,H\right]$, and $D_{~w}^{\infty}$
is the number of such subgroups of maximal Euler characteristic. If
$w\in\ker\left(H\twoheadrightarrow C_{m}^{~\rk H}\right)$ or $w\in\left[H,H\right]$,
then $w$ is a non-primitive element of $H$. Thus, in all these cases
$\mathbb{E}_{w}\left[\zeta_{1}\right]=O\left(N^{1-\pi\left(w\right)}\right)$
and, again, Conjecture \ref{conj:universal} holds. More evidence
towards Conjecture \ref{conj:universal} in these families of groups
is found in \cite{Ordo20,shomroni2021wreath}.

\subsubsection*{Matrix Groups over finite fields}

Finally, fix a finite field $\mathbb{F}_{q}$ and consider a family
of groups such as $\left\{ \mathrm{GL}_{N}\left(\mathbb{F}_{q}\right)\right\} _{N}$.
The ring of class functions corresponding to this family of groups
can be constructed as follows. For every positive integer $k$ and
$A\in\mathrm{GL}_{k}\left(\mathbb{F}_{q}\right)$, define $\xi_{A}\colon\mathrm{GL}_{N}\left(\mathbb{F}_{q}\right)\to\mathbb{Z}_{\ge0}$
by
\[
\xi_{A}\left(B\right)=\#\left\{ M\in\mathrm{Mat}_{N\times k}\left(\mathbb{F}_{q}\right)\,\middle|\,BM=MA\right\} .
\]
This is indeed a class function. For example, if $A=\left(1\right)\in\mathrm{GL}_{1}\left(\mathbb{F}_{q}\right)$,
then $\xi_{A}\left(B\right)$ counts the number of fixed point in
the action of $B$ on $\mathbb{F}_{q}^{~N}$. If $A\sim A'$ are conjugates
in $\mathrm{GL}_{k}\left(\mathbb{F}_{q}\right)$, then $\xi_{A}=\xi_{A'}$.
These class functions are analogous to the monomials $\qak$ defined
on symmetric groups, and they linearly span the ring of class functions
for this family of groups: 
\[
\A\defi\mathrm{span}_{\mathbb{C}}\left(1,\left\{ \xi_{A}\right\} _{k\in\mathbb{Z}_{\ge1},A\in\mathrm{GL}_{k}\left(\mathbb{F}_{q}\right)}\right).
\]
Consider elements of $\A$ corresponding to irreducible characters
of $\left\{ \mathrm{GL}_{N}\left(\mathbb{F}_{q}\right)\right\} _{N}$
(for every large enough $N$). Such families of characters are the
ones Conjecture \ref{conj:universal} relates to in this case. As
an example, one such family of irreducible characters is given by
the permutation-representation given by the action of $\mathrm{GL}_{N}\left(\mathbb{F}_{q}\right)$
on the projective space $\mathbb{P}^{N-1}\left(\mathbb{F}_{q}\right)$
minus the trivial representation. This is a $\frac{q^{N}-q}{q-1}$-dimensional
representation. As an element of $\A$, it is given by $\chi=\left(\frac{1}{q-1}\sum_{A\in\mathrm{GL}_{1}\left(\mathbb{F}_{q}\right)\cong\mathbb{F}_{q}^{*}}\xi_{A}\right)-2$.
Conjecture \ref{conj:universal} says that in this case one should
have $\mathbb{E}_{w}\left[\chi\right]=O(\left(q^{N}\right)^{1-\pi\left(w\right)})$.
In \cite{West21} it is shown that $\mathbb{E}_{w}\left[f\right]$
is equal to a rational expression in $q^{N}$ for every $f\in\A$,
and partial evidence is given towards Conjecture \ref{conj:universal}
in the case of $\left\{ \mathrm{GL}_{N}\left(\mathbb{F}_{q}\right)\right\} _{N}$. 

\subsection*{Acknowledgments}

We thank the anonymous referees for valuable comments. This project
has received funding from the European Research Council (ERC) under
the European Union’s Horizon 2020 research and innovation programme
(grant agreement No 850956), and from the Israel Science Foundation:
ISF grant 1071/16.

\section{From words to subgroups\label{sec:From-words-to-subgroups}}

We now consider a few generalizations of our object of study that
will be crucial for the remainder of the paper. The quantities we
wish to study are of the form 
\[
\mathbb{E}_{w}\left[\qak\right]=\mathbb{E}_{\sigma_{1},\ldots,\sigma_{r}\in S_{N}}\left[\#\mathrm{fix}\left(w\left(\sigma_{1},\ldots,\sigma_{r}\right)\right)^{\alpha_{1}}\cdot\ldots\cdot\#\mathrm{fix}\left(w^{k}\left(\sigma_{1},\ldots,\sigma_{r}\right)\right)^{\alpha_{k}}\right].
\]
Assume that $w$ is written in the ordered basis $B=\left\{ b_{1},\ldots,b_{r}\right\} $
of $\F$. Choosing a uniformly random $r$-tuple of permutations from
$S_{N}$ is the same as choosing at random a homomorphism $\varphi\colon\F\to S_{N}$,
as $\varphi\left(b_{1}\right),\ldots,\varphi\left(b_{r}\right)$ is
a uniformly random $r$-tuple of permutations. Replacing the letters
of $w$ by the permutations $\varphi\left(b_{1}\right),\ldots,\varphi\left(b_{r}\right)$,
we obtain the permutation $\varphi\left(w\right)$. Hence,
\begin{equation}
\mathbb{E}_{w}\left[\qak\right]=\mathbb{E}_{\varphi\in\mathrm{Hom}\left(\F,S_{N}\right)}\left[\#\mathrm{fix}\left(\varphi\left(w\right)\right)^{\alpha_{1}}\cdot\ldots\cdot\#\mathrm{fix}\left(\varphi\left(w^{k}\right)\right)^{\alpha_{k}}\right].\label{eq:original object}
\end{equation}
Following \cite{PP15}, the first step in our analysis is to generalize
the function we study. This generalization is crucial for the next
steps. The most straightforward generalization is to consider quantities
of the form

\begin{equation}
\mathbb{E}_{\varphi\in\mathrm{Hom}\left(\F,S_{N}\right)}\left[\#\mathrm{fix}\left(\varphi\left(w_{1}\right)\right)\cdot\ldots\cdot\#\mathrm{fix}\left(\varphi\left(w_{\ell}\right)\right)\right],\label{eq:arbitrary words}
\end{equation}
for arbitrary words $w_{1},\ldots,w_{\ell}\in\F$. Next, we generalize
from fixed points of a word to \emph{common }fixed points of several
words, or, equivalently, to common fixed points of subgroups: note
that given a finite set of words $w_{1},\ldots,w_{t}\in\F$, an element
$i\in[N]$ is a \emph{common} fixed point of all the permutations
$\varphi\left(w_{1}\right),\ldots,\varphi\left(w_{t}\right)$ if and
only if it is a common fixed point of all the permutations in the
subgroup $\varphi\left(H\right)\le S_{N}$ where $H=\langle w_{1},\ldots,w_{t}\rangle\leq\F$.
For $H\le\F$ we denote by $\#\mathrm{fix}\left(\varphi\left(H\right)\right)$
the number of common fixed points of $\varphi\left(H\right)$. We
extend the function we wish to study to quantities of the form

\begin{equation}
\mathbb{E}_{\varphi\in\mathrm{Hom}\left(\F,S_{N}\right)}\left[\#\mathrm{fix}\left(\varphi\left(H_{1}\right)\right)\cdot\ldots\cdot\#\mathrm{fix}\left(\varphi\left(H_{\ell}\right)\right)\right],\label{eq:studying common fixed points of subgroups}
\end{equation}
where $H_{1},\ldots,H_{\ell}\leq\F$ are f.g.~(finitely generated)
subgroups of $\F$.

If $H,H'\le\F$ are conjugate subgroups then $\#\mathrm{fix}\left(\varphi\left(H\right)\right)=\#\mathrm{fix}\left(\varphi\left(H'\right)\right)$.
Therefore, \eqref{eq:studying common fixed points of subgroups} depends,
in fact, on a \emph{multiset of conjugacy classes of non-trivial f.g.~subgroups}
of $\F$. We shall work in the category of these objects, which we
denote $\mocc\left(\F\right)$.

Finally, assume that there are two multisets of non-trivial f.g.~subgroups
$H_{1},\ldots,H_{\ell}\le\F$ and $J_{1},\ldots,J_{m}\le\F$, and
that there is a map $f\colon\left[\ell\right]\to\left[m\right]$,
such that $H_{i}\le J_{f\left(i\right)}$ for all $1\le i\le\ell$.
Let $\left\{ \varphi_{j}\colon J_{j}\to S_{N}\right\} _{j=1}^{m}$
be independent, uniformly random homomorphisms. Our final generalization
of the object of study is to
\begin{equation}
\mathbb{E}_{\left\{ \varphi_{j}\in\Hom\left(J_{j},S_{N}\right)\right\} _{j=1}^{m}}\left[\#\mathrm{fix}\left(\varphi_{f\left(1\right)}\left(H_{1}\right)\right)\cdot\ldots\cdot\#\mathrm{fix}\left(\varphi_{f\left(\ell\right)}\left(H_{\ell}\right)\right)\right].\label{eq:final generalization of object of study}
\end{equation}

In the following section we will use the following formal definition
of morphisms in the category $\mocc\left(\F\right)$, which arises
naturally from the above-mentioned generalizations of our object of
study:
\begin{defn}
\label{def:morphism in MOCC}Let $\H=\left\{ H_{1}^{\F},\ldots,H_{\ell}^{\F}\right\} $
and $\J=\left\{ J_{1}^{\F},\ldots,J_{m}^{\F}\right\} $ be two elements
of $\mocc\left(\F\right)$. A morphism $\eta\colon\H\to\J$ consists
of a map $f\colon\left[\ell\right]\to\left[m\right]$ and a choice
of representatives $\overline{H_{1}}\in H_{1}^{\F},\ldots,\overline{H_{\ell}}\in H_{\ell}^{\F},\overline{J_{1}}\in J_{1}^{\F},\ldots,\overline{J_{m}}\in J_{m}^{\F}$
so that $\overline{H_{i}}\le\overline{J_{f\left(i\right)}}$ for all
$i\in\left[\ell\right]$.
\end{defn}

Given $f\colon\left[\ell\right]\to\left[m\right]$, two different
choices of representatives as in Definition \ref{def:morphism in MOCC}
may yield equivalent morphisms. We defer the exact definition of this
equivalence to the next section, where we give a geometric description
of the category $\mocc\left(\F\right)$ in terms of multi core graphs.
\begin{rem}
We made several non-obvious choices in our definitions in the current
section of the category $\mocc\left(\F\right)$, and in the equivalent
categories $\mcg_{B}\left(\F\right)$ defined in the next section.
For example, one could consider multi core graphs with $k$ ordered
basepoints for some fixed $k$. Even in the category of multi core
graphs without basepoints, we made the non-obvious choices of excluding
the trivial subgroup of $\F$, and not demanding in Definition \ref{def:morphism in MOCC}
that the map $f\colon\left[\ell\right]\to\left[m\right]$ be surjective.
There are good arguments for not making these choices. For example,
restricting to surjective maps $\left[\ell\right]\to\left[m\right]$
would simplify the statement of Proposition \ref{prop:properties of free morphisms}\eqref{enu:EC of free morphisms}
as well as of Definition \ref{def:basis independent norm}, and imply
that the norm $\left\Vert \eta\right\Vert $ of a morphism $\eta$
(see Definition \ref{def:basis independent norm}) is zero if and
only if it is an isomorphism. However, pullbacks exist as simply and
neatly as stated on page \pageref{pullback} only if the image of
a morphism may avoid some of the components in its codomain and if
the trivial group is excluded. In addition, non-surjective morphisms
allow us to have the empty set as an element in $\mocc\left(\F\right)$
with a unique morphism to any other element. Avoiding trivial subgroups
also means that the Euler characteristic of multi core graphs (see
Definition \ref{def:chi, c, rank} below) is always non-positive,
which is convenient. Notice that the affect of adding a trivial subgroup
to the multiset in \eqref{eq:studying common fixed points of subgroups}
would be a multiplication of the expectation by $N$.
\end{rem}

\section{Multi core graphs\label{sec:Multi-core-graphs}}

We use core graphs, and more generally \emph{multi} core graphs, as
a geometric picture of the multisets of subgroups considered above.

\subsection{Core graphs\label{subsec:Core-graphs}}

Let $B=\left\{ b_{1},\ldots,b_{r}\right\} $ be a basis of $\F$,
and consider the bouquet $X_{B}$ of $r$ circles with distinct labels
from $B$ and arbitrary orientations and with wedge point $o$. Then
$\pi_{1}\left(X_{B},o\right)$ is naturally identified with $\F$.
The notion of ($B$-labeled) core graphs, introduced in \cite{stallings1983topology},
refers to finite\footnote{One may include also infinite core graphs in the definition, but these
are not needed in the current paper.}, connected graphs with every vertex having degree at least two (so
no leaves and no isolated vertices), that come with a graph morphism
to $X_{B}$ which is an \emph{immersion}, namely, locally injective.
In other words, this is a finite connected graph with at least one
edge and no leaves, with edges that are directed and labeled by the
elements of $B$, such that for every vertex $v$ and every $b\in B$,
there is at most one incoming $b$-edge and at most one outgoing $b$-edge
at $v$. We stress that multiple edges between two vertices and loops
at vertices are allowed.

There is a natural one-to-one correspondence between finite $B$-labeled
core graphs and conjugacy classes of non-trivial f.g.~subgroups of
$\F$. Indeed, given a core graph $\Gamma$ as above, pick an arbitrary
vertex $v$ and consider the ``labeled fundamental group'' $\piol\left(\Gamma,v\right)$:
closed paths in a graph with oriented and $B$-labeled edges correspond
to words in the elements of $B$. In other words, if $p\colon\Gamma\to X_{B}$
is the immersion, then $\piol\left(\Gamma,v\right)$ is the subgroup
$p_{*}\left(\pi_{1}\left(\Gamma,v\right)\right)$ of $\pi_{1}\left(X_{B},o\right)=\F$.
The conjugacy class of $\piol\left(\Gamma,v\right)$ is independent
of the choice of $v$ and is the conjugacy class corresponding to
$\Gamma$. We denote it by $\piol\left(\Gamma\right)$.

Conversely, if $H\le\F$ is a non-trivial f.g.~subgroup, the conjugacy
class $H^{\F}$ corresponds to a finite core graph, denoted $\Gamma_{B}\left(H^{\F}\right)$,
which can be obtained in different manners. For example, let $\Upsilon$
be the topological covering space of $X_{B}$ corresponding to $H^{\F}$,
which is equal in this case to the Schreier graph depicting the action
of $\F$ on the right cosets of $H$ with respect to the generators
$B$. Then $\Gamma_{B}\left(H^{\F}\right)$ is obtained from $\Upsilon$
by 'pruning all hanging trees', or, equivalently, as the union of
all non-backtracking cycles in $\Upsilon$. One can also construct
$\Gamma_{B}\left(H^{\F}\right)$ from any finite generating set of
$H$ using ``Stallings foldings'' -- see \cite{stallings1983topology,kapovich2002stallings,Puder2014,PP15}
for more details about foldings and about core graphs in general.

\subsection{Multi core graphs and their morphisms\label{subsec:Multi-core-graphs}}

Here, we consider core graphs which are not necessarily connected:
\begin{defn}
\label{def:multi-core-graphs}Let $B$ be a basis of $\F$. A $B$-labeled
\emph{multi core graph }is a disjoint union of finitely many core
graphs. In other words, this is a finite graph, not necessarily connected,
with no leaves and no isolated vertices, and which comes with an immersion
to $X_{B}$. We denote the set of $B$-labeled multi core graphs by
$\mcg_{B}\left(\F\right)$.
\end{defn}

Because a connected core graph corresponds to a conjugacy class of
non-trivial f.g.~subgroups of $\F$, a multi core graph corresponds
to a multiset of such objects. Therefore, every basis $B$ of $\F$
gives rise to a one-to-one correspondence
\begin{equation}
\begin{gathered}\mcg_{B}\left(\F\right)=\\
\left\{ \begin{gathered}B\mathrm{-labeled~}\\
\mathrm{multi~core~graphs}
\end{gathered}
\right\} 
\end{gathered}
~~~\longleftrightarrow~~~\begin{gathered}\mocc\left(\F\right)=\\
\left\{ \begin{gathered}\mathrm{finite~multisets~of~conjugacy~classes}\\
\mathrm{of~non\textnormal{-}trivial~f.g.~subgroups~of}~\F
\end{gathered}
\right\} 
\end{gathered}
.\label{eq:1-to-1 correspondence of categories}
\end{equation}
For a multi core graph $\Gamma\in\mcg_{B}\left(\F\right)$ we let
$\piol\left(\Gamma\right)$ denote the corresponding multiset in $\mocc\left(\F\right)$,
and for a multiset $\H\in\mocc\left(\F\right)$ we let $\Gamma_{B}\left(\H\right)$
denote the corresponding multi core graph.
\begin{defn}
\label{def:morphism of multi CG}A morphism $\eta\colon\Gamma\to\Delta$
between $B$-labeled multi core graphs is a graph-morphism which commutes
with the immersions $p,q$ to $X_{B}$.
\[
\xymatrix{\Gamma\ar[rd]^{p}\ar[rr]^{\eta} &  & \Delta\ar[ld]_{q}\\
 & X_{B}
}
\]
\end{defn}

In particular, the morphism $\eta$ is itself an immersion, and it
preserves the orientations and labels of the edges. To get a description
of $\eta$ in terms of subgroups à la Definition \ref{def:morphism in MOCC},
assume that $\Gamma$ consists of $\ell$ components $\Gamma_{1},\ldots,\Gamma_{\ell}$
and that $\Delta$ consists of $m$ components $\Delta_{1},\ldots,\Delta_{m}$.
Let $f\colon\left[\ell\right]\to\left[m\right]$ be the induced map
on connected components, so $\eta\left(\Gamma_{i}\right)\subseteq\Delta_{f\left(i\right)}$.
For every $i\in\left[\ell\right]$, pick an arbitrary vertex $v_{i}\in\Gamma_{i}$
and let $H_{i}=\piol\left(\Gamma_{i},v_{i}\right)$. As $\eta$ is
an immersion, it induces \emph{injective} maps at the level of fundamental
groups: indeed, any non-backtracking cycle in $\Gamma$ is mapped
to a non-backtracking cycle in $\Delta$. Therefore, $\eta$ can be
thought of as the embedding, for all $i\in\left[\ell\right]$, 
\begin{equation}
H_{i}\hookrightarrow\pi_{1}\left(\Delta_{f\left(i\right)},\eta\left(v_{i}\right)\right).\label{eq:injection of pi1}
\end{equation}
We still need to conjugate the images in \eqref{eq:injection of pi1}
so that they all sit in the same subgroups in the conjugacy class
of subgroups of $\Delta_{j}$. Formally, pick an arbitrary vertex
$p_{k}\in\Delta_{k}$ for all $k\in\left[m\right]$ and let $J_{k}=\pi_{1}\left(\Delta_{k},p_{k}\right)$.
For every $i\in\left[\ell\right]$, let $u_{i}\in\F$ satisfy $u_{i}\left[\pi_{1}\left(\Delta_{f\left(i\right)},\eta\left(v_{i}\right)\right)\right]u_{i}^{-1}=J_{f\left(i\right)}$.
So now $u_{i}H_{i}u_{i}^{-1}\le J_{f\left(i\right)}$, and we get
a morphism as in Definition \ref{def:morphism in MOCC}.

Conversely, every embedding of subgroups $H\hookrightarrow J$ of
$\F$ gives rise to a morphism of core graphs $\Gamma_{B}\left(H\right)\to\Gamma_{B}\left(J\right)$.
Indeed, if one considers the entire covering space $\Upsilon_{H}$
of $X_{B}$ corresponding to $H$, there is certainly a morphism to
$\Upsilon_{J}$, the one corresponding to $J$. Because every non-backtracking
closed path in $\Upsilon_{H}$ is mapped to a non-backtracking closed
path in $\Upsilon_{J}$ (being non-backtracking is a local property),
we see that the image of $\Gamma_{B}\left(H\right)$ is contained
in $\Gamma_{B}\left(J\right)$. Thus any morphism in $\mocc\left(\F\right)$
as in Definition \ref{def:morphism in MOCC}, gives rise to a morphism
of the corresponding multi core graphs. We say that two morphisms
in $\mocc\left(\F\right)$ are identical if they induce the same morphism
of multi core graphs, up to a post-composition by an isomorphism of
the codomain. This equivalence of morphisms in $\mocc\left(\F\right)$
can also be defined in completely algebraic terms\footnote{This equivalence is the generalization of the fact that if $H\le J$
then so does $jHj^{-1}\le J$ for all $j\in J$ and the corresponding
morphism of core graphs is identical. We do not elaborate further
because we anyway use here only the ``geometric'' description of
this equivalence, which is more straightforward. }, and, in particular, it does not depend on the basis $B$.

Throughout the text we use the following three important invariants
of multi core graphs.
\begin{defn}
\label{def:chi, c, rank}Let $\Gamma\in\mcg_{B}\left(\F\right)$ be
a multi core graph and $\mathcal{H}=\piol\left(\Gamma\right)=\left\{ H_{1}^{\F},\ldots,H_{\ell}^{\F}\right\} $
the corresponding multiset in $\mocc\left(\F\right)$. We denote by
$\rk\mathcal{H}=\rk\Gamma$ the sum of ranks of $H_{1},\ldots,H_{\ell}$,
by $\chi\left(\Gamma\right)=\chi\left({\cal H}\right)\defi\#V\left(\Gamma\right)-\#E\left(\Gamma\right)$
the Euler characteristic of $\Gamma$, and by $c\left(\Gamma\right)=c\left({\cal H}\right)$
the number of connected components of $\Gamma$ (which is $\ell$
in the current notation). These three quantities are related by $\rk\Gamma+\chi\left(\Gamma\right)=c\left(\Gamma\right)$.
Note that as we excluded the trivial subgroup, $\chi\left(\Gamma\right)\le0$
for all $\Gamma\in\mcg_{B}\left(\F\right)$.
\end{defn}

We are now able to define another form of the function $\Phi$ as
in \eqref{eq:final generalization of object of study}, which depends
on a morphism of multi core graphs:
\begin{defn}
\label{def:Phi}Let $\eta\colon\Gamma\to\Delta$ be a morphism of
multi core graphs. Assume that $\piol\left(\Gamma\right)=\left\{ H_{1}^{\F},\ldots,H_{\ell}^{\F}\right\} $
and $\piol\left(\Delta\right)=\left\{ J_{1}^{\F},\ldots,J_{m}^{\F}\right\} $.
As above, let $f\colon\left[\ell\right]\to\left[m\right]$ correspond
to $\eta$, and assume that $u_{i}H_{i}u_{i}^{-1}\le J_{f\left(i\right)}$
for some $u_{i}\in\F$ for all $i\in\left[\ell\right]$ be the embedding
corresponding to $\eta$. Let $\left\{ \varphi_{k}\colon J_{k}\to S_{N}\right\} _{k=1}^{m}$
be independent, uniformly random homomorphisms. Define\vspace{-20bp}
\end{defn}

\[
\Phi_{\eta}\left(N\right)\defi\mathbb{E}_{\left\{ \varphi_{k}\in\Hom\left(J_{k},S_{N}\right)\right\} _{k=1}^{m}}\left[\#\mathrm{fix}\left(\varphi_{f\left(1\right)}\left(u_{1}H_{1}u_{1}^{-1}\right)\right)\cdot\ldots\cdot\#\mathrm{fix}\left(\varphi_{f\left(\ell\right)}\left(u_{\ell}H_{\ell}u_{\ell}^{-1}\right)\right)\right].
\]

\begin{example}
\label{exa:Phi for id}For $\id\colon\Gamma\to\Gamma$, we have $\Phi_{\id}\left(N\right)=N^{\chi\left(\Gamma\right)}$.
Indeed, the value of $\Phi$ is multiplicative with respect to the
different connected components of the codomain. In a component $\Gamma_{0}$
of rank $k$, the probability that $k$ independent permutations fix
some $i\in\left[N\right]$ is $\frac{1}{N^{k}}$, so the expected
number of common fixed points is $N^{1-k}=N^{\chi\left(\Gamma_{0}\right)}$.
\end{example}

\begin{example}
\label{exa:main thm in terms of Phi}To illustrate, we present the
geometric picture, in terms of multi core graphs, of the object of
study this paper began with -- $\mathbb{E}_{w}\left[\qak\right]$.
Here there is a multiset $\H\in\mocc\left(\F\right)$ of size $\alpha_{1}+\ldots+\alpha_{k}$,
with $\alpha_{1}$ copies of $\left\langle w\right\rangle ^{\F}$,
$\alpha_{2}$ copies of $\left\langle w^{2}\right\rangle ^{\F}$,
and so on. The corresponding multi core graph is denoted $\Gamma_{\ak}^{w}\defi\Gamma_{B}\left(\H\right)$.
It consists of $\alpha_{1}$ disjoint copies of a cycle of length
$\left|w\right|_{c}$ depicting $w$, together with $\alpha_{2}$
disjoint copies of a cycle of length $\left|w^{2}\right|_{c}$ depicting
$w^{2}$, and so on, where $\left|w\right|_{c}$ denotes the length
of the cyclic reduction of $w$. The second multiset $\J\in\mocc\left(\F\right)$
is the singleton $\left\{ \F^{\F}\right\} =\left\{ \left\{ \F\right\} \right\} $,
corresponding to the (multi) core graph $X_{B}$. There is only one
possible morphism between the two -- the immersion from Definition
\ref{def:multi-core-graphs} -- and we denote it by $\eta_{\ak}^{w}\colon\Gamma\to X_{B}$.
This is illustrated in Figure \ref{fig:example of multi core graph}.
In this case we have
\[
\Phi_{\eta_{\ak}^{w}}=\mathbb{E}_{w}\left[\qak\right].
\]
\begin{figure}
\includegraphics[viewport=0bp 0bp 774bp 387bp,clip,scale=0.6]{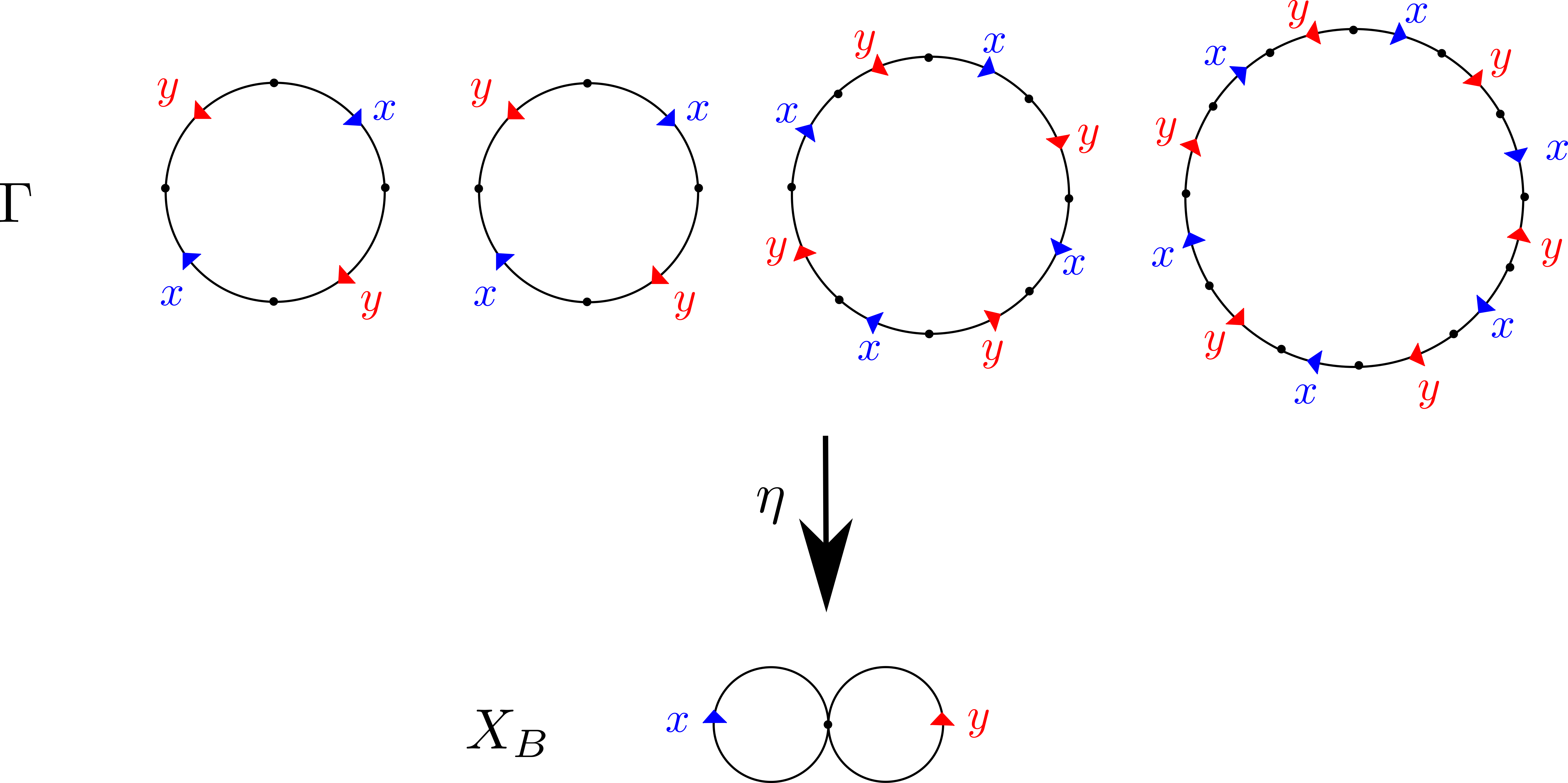}\caption{\label{fig:example of multi core graph}Let $\protect\F_{2}$ have
basis $B=\left\{ x,y\right\} $, and let $w=yxyxy^{-2}\in\protect\F_{2}$.
The multi core graph in the top part of the figure is $\Gamma=\Gamma_{B}\left(\protect\H\right)$
where $\protect\H=\left\{ \left\langle w\right\rangle ^{\protect\F_{2}},\left\langle w\right\rangle ^{\protect\F_{2}},\left\langle w^{2}\right\rangle ^{\protect\F_{2}},\left\langle w^{3}\right\rangle ^{\protect\F_{2}}\right\} $.
It is denoted $\Gamma_{2,1,1}^{w}$ in the notation from Example \ref{exa:main thm in terms of Phi}.
The bottom part shows the bouquet $X_{B}$. There is a single morphism
of multi core graphs between these two, and we denote it by $\eta_{2,1,1}^{w}$.
We have $\Phi_{\eta_{2,1,1}^{w}}=\mathbb{E}_{w}\left[\xi_{1}^{~2}\xi_{2}\xi_{3}\right]$,
where $\Phi_{\eta_{2,1,1}^{w}}$ is defined in Definition \ref{def:Phi}.}
\end{figure}
\end{example}

As explained in \cite[Section 6]{PP15} for the simpler case analyzed
there, $\Phi_{\eta}\left(N\right)$ can also be given the following
completely geometric interpretation. Let $\widehat{\Delta_{N}}$ be
a random $N$-sheeted covering space of $\Delta$, defined as follows.
Its vertex-set is $V\left(\Delta\right)\times\left[N\right]$. For
every directed edge $e=\left(u,v\right)\in E\left(\Delta\right)$,
choose uniformly at random a permutation $\sigma_{e}\in S_{N}$, and
introduce in $\widehat{\Delta_{N}}$ a directed edge $(u,i)\to(v,\sigma_{e}(i))$
with the same label as $e$ for every $i\in[N]$. This is indeed an
$N$-sheeted covering of $\Delta$ with the projection $\left(u,i\right)\mapsto u$
and $\left(\left(u,i\right),\left(v,\sigma_{e}\left(i\right)\right)\right)\to e$.
\begin{prop}
\label{prop:geometric meaning for Phi}Let $\eta\colon\Gamma\to\Delta$
be a morphism of multi core graphs. Then $\Phi_{\eta}\left(N\right)$
is equal to the average number of lifts of $\eta$ to the random $N$-covering
$\widehat{\Delta_{N}}$.
\[
\xymatrix{ & \widehat{\Delta_{N}}\ar@{->>}[d]^{p}\\
\Gamma\ar[r]_{\eta}\ar@{-->}[ur]^{\#} & \Delta
}
\]
\end{prop}

\begin{proof}
By multiplicativity of $\Phi_{\eta}$, it suffices to prove the claim
assuming that $\Delta$ is connected. Then the proposition practically
reduces to \cite[Lemma 6.2]{PP15}.
\end{proof}

\section{Free and algebraic morphisms\label{sec:Free-and-algebraic}}

A subgroup $H$ of a free group $\F$ is called a \emph{free factor}
of $\F$, and $\F$ a \emph{free extension} of $H$, denoted $H\ff\F$,
if it is generated by some subset of a basis of $\F$. Equivalently,
this means that there is another subgroup $K\leq\F$, such that $\F=H\ast K$.
The useful notion of an algebraic extensions of free groups is defined
as follows (see \cite{miasnikov2007algebraic} for a survey):
\begin{defn}
Let $H$ be a subgroup of the free group $\F$. Then $\F$ is an \emph{algebraic
extension }of $H$, denoted $H\alg\F$, if there is no intermediate
proper free factor of $\F$. Namely, if whenever $H\leq J\ff F$,
we have $J=\F$.
\end{defn}

Given a morphism of \emph{connected }core graphs, we may say it is
free (algebraic) if the induced map in the level of fundamental groups
gives a free (algebraic, respectively) extension of groups. A crucial
ingredient of our argument is to find the right generalizations of
these notions to morphisms of multi core graphs. We start with free
morphisms.

\subsection{Free morphisms\label{subsec:Free-morphisms}}
\begin{defn}
\label{def:free morphisms}If $H_{1},\ldots,H_{\ell}$ are subgroups
of the free group $J$, we say that $J$ is a \emph{free extension}
of the multiset $\left\{ H_{1},\ldots,H_{\ell}\right\} $, denoted
$\left\{ H_{1},\ldots,H_{\ell}\right\} \ff J$, if $J$ decomposes
as a free product 
\[
J=\left(\ast_{i=1}^{\ell}j_{i}H_{i}j_{i}^{-1}\right)*K
\]
for some conjugate subgroup $j_{i}H_{i}j_{i}^{-1}$ of $H_{i}$ (so
$j_{i}\in J$) and some subgroup $K\le J$.

Now let $\eta:\Gamma\to\Delta$ be a morphism of multi core graphs
with $\Delta$ \uline{connected}. As explained in Section \ref{subsec:Multi-core-graphs},
one can pick an arbitrary subgroup $J$ in the single conjugacy class
in $\piol\left(\Delta\right)$, and for every component $\Gamma_{1},\ldots,\Gamma_{\ell}$
of $\Gamma$, a suitable subgroup $H_{i}$ so that $H_{i}\le J$.
Say that $\eta$ is a \emph{free} morphism if $\left\{ H_{1},\ldots,H_{\ell}\right\} \ff J$.
Finally, say that a general morphism $\eta:\Gamma\to\Delta$ of multi
core graphs is \emph{free }if $\eta\Big|_{\eta^{-1}\left(\Delta'\right)}\colon\eta^{-1}\left(\Delta'\right)\to\Delta'$
is free for every connected component $\Delta'$ of $\Delta$.
\end{defn}

The definition of a free morphism does not depend on any of the choices
made: not on the choice of $J$ and not on the choice of the $H_{i}$'s.
The following theorem states some properties of free morphisms. In
particular, it shows that the set of multi core graphs together with
free morphisms form a valid category. By an injective morphism we
mean a morphism which is both edge-injective and vertex-injective. 
\begin{prop}
\label{prop:properties of free morphisms}
\begin{enumerate}
\item \label{enu:Every-injective-morphism is free}Every \emph{injective}
morphism of multi core graphs is free. In particular, the identity
morphism is free.
\item \label{enu:free is transitive}The composition of two free morphisms
is free.
\item \label{enu:EC of free morphisms}If $\eta\colon\Gamma\to\Delta$ is
a free morphism of multi core graphs, then $\chi\left(\Delta\right)\le\chi\left(\Gamma\right)$,
with equality if and only if $\left(i\right)$ $\eta$ induces an
isomorphism between $\Gamma$ and the connected components of $\Delta$
meeting $\mathrm{Im}\left(\eta\right)$, and $\left(ii\right)$ the
remaining connected components of $\Delta$ are cycles.
\item \label{enu:Phi can ignore free extensions}If $\xymatrix{\bullet\ar[r]_{\varphi}\ar@/^{0.5pc}/[rr]^{\eta} & \bullet\ar[r]_{\psi} & \bullet}
$ is a composition of morphisms with $\psi$ free, then $\Phi_{\varphi}=\Phi_{\eta}$.
\end{enumerate}
\end{prop}

\begin{proof}
Item \ref{enu:Every-injective-morphism is free} is a generalization
of the fact that if $\left(\Gamma_{1},v_{1}\right)\hookrightarrow\left(\Gamma_{2},v_{2}\right)$
is an embedding of connected, pointed graphs, then $\pi_{1}\left(\Gamma_{1},v_{1}\right)\ff\pi_{1}\left(\Gamma_{2},v_{2}\right)$
-- the proof in the case that several connected components in the
domain are mapped to a single component in the codomain is straightforward
(the simple idea of the proof also appears in the proof of Lemma \ref{lem:decomposition of a graph is free }
below). Item \ref{enu:free is transitive} is a straightforward generalization
of the transitivity of free extensions in free groups. Item \ref{enu:EC of free morphisms}
follows from the fact that if $\left\{ H_{1},\ldots,H_{\ell}\right\} \ff J$,
then $\sum_{i=1}^{\ell}\rk H_{i}\le\rk J$, and so 
\[
\chi\left(\Gamma_{B}\left(\left\{ H_{1}^{\F},\ldots,H_{\ell}^{\F}\right\} \right)\right)=\ell-\sum_{i=1}^{\ell}\rk H_{i}\ge1-\rk J=\chi\left(\Gamma_{B}\left(J^{\F}\right)\right),
\]
with equality if and only if $\ell=1$ and $H_{1}=J$, or $\ell=0$
and $\rk J=1$. Finally, item \ref{enu:Phi can ignore free extensions}
is a straightforward generalization of the corresponding claim for
connected core graphs -- see, e.g., \cite[Remark 5.2]{PP15}. 
\end{proof}
We end this subsection with two lemmas concerning free morphisms that
we need in Section \ref{subsec:Algebraic-morphisms}. Lemma \ref{lem:decomposition of a graph is free }
generalizes the fact that injective morphisms are free. Lemma \ref{lem:pullback of free is free}
generalizes the fact that if $H\le\F$ and $K\ff\F$ are all free
groups, then $H\cap K\ff H$ (e.g., \cite[Claim 3.9]{PP15}).
\begin{lem}
\label{lem:decomposition of a graph is free }Let $\Gamma$ be a multi
core graph. Let ${\cal P}$ be a partition of a subset of the edge-set
of $\Gamma$. For every block $\beta\in{\cal P}$, consider the multi
core graph $\Sigma_{\beta}$ obtained by deleting from $\Gamma$ the
edges outside $\beta$, and then recursively pruning all leaves and
deleting isolated vertices. Let $\Sigma=\bigsqcup_{\beta\in{\cal P}}\Sigma_{\beta}$
be the multi core graph obtained as the disjoint union of the $\Sigma_{\beta}$'s.
The embeddings $\Sigma_{\beta}\hookrightarrow\Gamma$ give rise to
a morphism $\eta\colon\Sigma\to\Gamma$. Then $\eta$ is free.
\end{lem}

\begin{proof}
Because freeness of morphisms is tested in every connected component
of the codomain separately, we may assume $\Gamma$ is connected.
Fix a basepoint $\otimes$ and a spanning tree $T$ in $\Gamma$.
Let $J=\piol\left(\Gamma,\otimes\right)$. An arbitrary orientation
of the edges of $\Gamma$ outside $T$ standardly gives rise to a
basis of $J$. We shall construct a similar basis which shows the
freeness of $\eta$.

Let $\Lambda$ be an arbitrary connected component of $\Sigma$, embedded
in $\Gamma$. Note that $T\cap\Lambda$ is a forest inside $\Lambda$,
which can be extended to a spanning tree $T_{\Lambda}$ of $\Lambda$.
Fix $T_{\Lambda}$ for every $\Lambda$. For every edge $e\in\left(\bigcup_{\Lambda}T_{\Lambda}\right)\setminus T$,
orient $e$ arbitrarily, and let $j_{e}=u_{1}\overrightarrow{e}u_{2}\in J$,
where $u_{1}\in\F$ corresponds to the path through $T$ from $\otimes$
to the tail of $e$ and $u_{2}\in\F$ to the path through $T$ from
the head of $e$ to $\otimes$. Let $C_{o}$ be the set of all such
elements of $J$ obtained from all edges in $\left(\bigcup_{\Lambda}T_{\Lambda}\right)\setminus T$.

For all $\Lambda$, let $\otimes_{\Lambda}$ be a fixed basepoint
of $\Lambda$ and $u_{\Lambda}\in\F$ be the path from $\otimes$
to $\otimes_{\Lambda}$ through $T$. Construct a basis $C'_{\Lambda}$
for $H_{\Lambda}\defi\piol\left(\Lambda,\otimes_{\Lambda}\right)$
using $T_{\Lambda}$, and note that $C_{\Lambda}\defi u_{\Lambda}C'_{\Lambda}u_{\Lambda}^{-1}\defi\left\{ u_{\Lambda}cu_{\Lambda}^{-1}\,\middle|\,c\in C'_{\Lambda}\right\} $
is a basis for $u_{\Lambda}H_{\Lambda}u_{\Lambda}^{-1}$, which is
a subgroup of $J$. Now $C_{o}\cup\bigcup_{\Lambda}C_{\Lambda}$ is
a basis of $J$ which shows that indeed
\[
\left\{ H_{\Lambda}\right\} _{\Lambda}\ff J,
\]
namely, $\eta$ is free.
\end{proof}
There is a natural notion of \emph{pullback} \label{pullback}in the
equivalent categories $\mocc\left(\F\right)$ and $\mcg_{B}\left(\F\right)$.
It is very similar to the well-established notion of pullback in the
case of connected core graphs (e.g., \cite[Sections 1.3 and 5.5]{stallings1983topology}).
If $\eta_{1}\colon\Gamma_{1}\to\Delta$ and $\eta_{2}\colon\Gamma_{2}\to\Delta$
are morphisms, the pullback is the multi core graph $\Sigma$ defined
as follows: begin with the graph $\Sigma'$ with vertex-set 
\[
\left\{ \left(u_{1},u_{2}\right)\,\middle|\,u_{i}~\mathrm{a~vertex~of}~\Gamma_{i},~\eta_{1}\left(u_{1}\right)=\eta_{2}\left(u_{2}\right)\right\} ,
\]
and edge-set defined analogously, where the edge $\left(e_{1},e_{2}\right)$
begins at the pair of tails and ends at the pair of heads. Then, recursively
remove all leaves and isolated vertices to obtain $\Sigma$. There
are natural morphisms $\sigma_{i}\colon\Sigma\to\Gamma_{i}$, $i=1,2$,
defined as the projection to the $i$-th coordinate. This pullback
satisfies the universal property of pullbacks: for every pair of morphisms
$\gamma_{1}\colon\Lambda\to\Gamma_{1}$, $\gamma_{2}\colon\Lambda\to\Gamma_{2}$
such that $\eta_{1}\circ\gamma_{1}=\eta_{2}\circ\gamma_{2}$, there
is a unique $\overline{\gamma}\colon\Lambda\to\Sigma$ such that the
following diagram commutes:

\[
\xymatrix{\Lambda\ar@{.>}[rd]^{\exists!\overline{\gamma}}\ar@/^{2pc}/[rrd]^{\gamma_{1}}\ar@/_{2pc}/[rdd]_{\gamma_{2}}\\
 & \Sigma\ar[r]^{\sigma_{1}}\ar[d]_{\sigma_{2}} & \Gamma_{1}\ar[d]^{\eta_{1}}\\
 & \Gamma_{2}\ar[r]_{\eta_{2}} & \Delta
}
\]
In algebraic terms, the connected component of $\left(u_{1},u_{2}\right)$
in the pullback $\Sigma$ corresponds to the conjugacy class in $\piol\left(\Delta,\eta_{1}\left(u_{1}\right)\right)$
of $\piol\left(\Gamma_{1},u_{1}\right)\cap\piol\left(\Gamma_{2},u_{2}\right)$.
In total, for every connected components $G_{1}$ of $\Gamma_{1}$
and $G_{2}$ of $\Gamma_{2}$ which are mapped to same component $D$
of $\Delta$, let $J$ be a representative of the conjugacy class
$\piol\left(D\right)$, and let $H_{i}$ be a representative of $\piol\left(G_{i}\right)$
such that $H_{1},H_{2}\le J$ agree with the morphisms $\eta_{1},\eta_{2}$.
Then there is one connected component in the pullback $\Sigma$ for
every non-trivial conjugacy class of subgroups of $J$ in the set
$\left\{ H_{1}\cap jH_{2}j^{-1}\right\} _{j\in J}$. Most importantly,
the pullback construction can be completely defined in the category
$\mocc\left(\F\right)$, and is thus basis-independent\footnote{Note that it is very much possible that the pullback $\Sigma$ be
empty, and recall that the empty multiset is an element of $\mocc\left(\F\right)$.}.
\begin{lem}
\label{lem:pullback of free is free}In the above notation, if $\eta_{1}\colon\Gamma_{1}\to\Delta$
and $\eta_{2}\colon\Gamma_{2}\to\Delta$ are morphisms \emph{with
$\eta_{2}$ free and $\left(\Sigma,\sigma_{1},\sigma_{2}\right)$
is the pullback, then $\sigma_{1}\colon\Sigma\to\Gamma_{1}$ is also
free.}
\begin{equation}
\xymatrix{\Sigma\ar[r]_{*}^{\sigma_{1}}\ar[d]_{\sigma_{2}} & \Gamma_{1}\ar[d]^{\eta_{1}}\\
\Gamma_{2}\ar[r]_{\eta_{2}}^{*} & \Delta
}
\label{eq:freeness diagram}
\end{equation}
\end{lem}

\begin{proof}
Note that in the diagram \eqref{eq:freeness diagram}, there is no
interaction between components of $\Gamma_{1},\Gamma_{2}$ and $\Sigma$
which are mapped to different components of $\Delta$, and recall
that freeness is tested independently in every connected component
of the codomain. Thus, we may assume that $\Delta$ is connected.
Because the pullback can be constructed in pure algebraic terms, we
may assume $J=\piol\left(\Delta,v\right)$ is the ambient free group
($v$ is some arbitrary vertex in $\Delta$), and pick a basis $Q$
of $J$ which extends a basis for the components of $\Gamma_{2}$.
Namely, every component of $\Gamma_{2}$ corresponds to the conjugacy
class of the subgroup generated by some subset of $Q$, and the different
subsets are disjoint. The geometric picture in this basis is that
$\Delta$ and every component of $\Gamma_{2}$ are bouquets (a single
vertex with several loops).

But now, for every component of $\Gamma_{2}$, let $\beta\subseteq Q$
be the corresponding subset of basis elements. This gives rise to
a partition of a subset of the edge-set of $\Gamma_{1}$, with one
block consisting of all edges colored by elements from $\beta$, for
every component of $\Gamma_{2}$. It is easy to see that the pullback
$\Sigma$ is identical to the construction described in Lemma \ref{lem:decomposition of a graph is free }
with respect to this partition of the edges of $\Gamma_{2}$, and
thus $\sigma_{1}$ is free.
\end{proof}

\subsection{Algebraic morphisms\label{subsec:Algebraic-morphisms}}

We turn to defining our generalization of the notion of algebraic
extensions.
\begin{defn}
\label{def:alg morphism}Let $\eta\colon\Gamma\to\Delta$ be a morphism
of multi core graphs. We say that $\eta$ is \emph{algebraic} if whenever
$\Gamma\stackrel{_{\eta_{1}}}{\longrightarrow}\Sigma\stackrel{_{\eta_{2}}}{\longrightarrow}\Delta$
is a decomposition of $\eta$ with $\eta_{2}$ free, we have that
$\eta_{2}$ is an isomorphism.
\end{defn}

Because the definition of a free morphism is basis-independent, so
is the definition of an algebraic morphism. The following theorem
lists some important properties of algebraic morphisms. In particular,
it shows that the set of multi core graphs together with algebraic
morphisms form a valid category.
\begin{thm}
\label{thm:properties of algebraic morphisms}
\begin{enumerate}
\item \label{enu:alg is surjective}Every algebraic morphism of multi core
graphs is surjective.
\item \label{enu:The-identity-morphism-is-alg}The identity morphism is
algebraic.
\item \label{enu:transitivity of algebraic morphisms}The composition of
two algebraic morphisms is algebraic. 
\end{enumerate}
\end{thm}

\begin{proof}
Every $\eta\colon\Gamma\to\Delta$ decomposes as $\Gamma\stackrel{_{\eta_{1}}}{\twoheadrightarrow}\Sigma\stackrel{_{\eta_{2}}}{\hookrightarrow}\Delta$,
where $\Sigma$ is the image of $\eta$ and $\eta_{2}$ is its embedding
in $\Delta$. Note that $\Sigma$ may contain multiple components
which are embedded in the same component of $\Delta$. By Proposition
\ref{prop:properties of free morphisms}\eqref{enu:Every-injective-morphism is free},
$\eta_{2}$ is free. Thus $\eta$ cannot be algebraic unless $\eta_{2}$
is an isomorphism, namely, $\eta$ is surjective. This shows item
\ref{enu:alg is surjective}.

For item \ref{enu:The-identity-morphism-is-alg}, note that any decomposition
of an identity morphism $\mathrm{id}\colon\Gamma\to\Gamma$ is through
an injective and surjective morphisms: $\Gamma\stackrel{\eta_{1}}{\hookrightarrow}\Sigma\stackrel{\eta_{2}}{\twoheadrightarrow}\Gamma$.
Then $\eta_{1}$ is free by Proposition \ref{prop:properties of free morphisms}\eqref{enu:Every-injective-morphism is free}.
If we assume that $\eta_{2}$ is also free, we obtain, by Proposition
\ref{prop:properties of free morphisms}\eqref{enu:EC of free morphisms},
that $\chi\left(\Sigma\right)=\chi\left(\Gamma\right)$, and that
both $\eta_{1}$ and $\eta_{2}$ are isomorphisms.

Finally, assume that $\Lambda\stackrel{\eta_{1}}{\to}\Gamma\stackrel{\eta_{2}}{\to}\Delta$
is a chain of two algebraic morphisms. Assume that there is a decomposition
of $\eta_{2}\circ\eta_{1}$ as $\Lambda\stackrel{\gamma_{1}}{\to}\Gamma'\stackrel{\gamma_{2}}{\to}\Delta$,
with $\gamma_{2}$ free. Let $\Sigma$ be the pullback of $\eta_{2}$
and $\gamma_{2}$ and $\overline{\gamma}\colon\Lambda\to\Sigma$ the
unique morphism so that Diagram \eqref{eq:proof-of-alg-is-transitive-diagram}
commutes. By Lemma \ref{lem:pullback of free is free}, $\sigma_{1}$
is free. In the notation of Diagram \eqref{eq:proof-of-alg-is-transitive-diagram},
we obtain that $\Lambda\stackrel{\overline{\gamma}}{\to}\Sigma\stackrel[*]{\sigma_{1}}{\to}\Gamma$
is a decomposition of the algebraic $\eta_{1}$, hence $\sigma_{1}$
is an isomorphism.
\begin{equation}
\xymatrix{\Lambda\ar[rd]^{\overline{\gamma}}\ar@/^{2pc}/[rrd]^{\eta_{1}}\ar@/_{2pc}/[rdd]_{\gamma_{1}}\\
 & \Sigma\ar[r]_{*}^{\sigma_{1}}\ar[d]_{\sigma_{2}} & \Gamma\ar[d]^{\eta_{2}}\\
 & \Gamma'\ar[r]_{\gamma_{2}}^{*} & \Delta
}
\label{eq:proof-of-alg-is-transitive-diagram}
\end{equation}
But now $\Sigma\stackrel{\eta_{2}\circ\sigma_{1}}{\to}\Delta$ is
algebraic (because $\eta_{2}$ is), and so in its decomposition $\Sigma\stackrel{\sigma_{2}}{\to}\Gamma'\stackrel[*]{\gamma_{2}}{\to}\Delta$,
$\gamma_{2}$ must be an isomorphism too. This proves that $\eta_{2}\circ\eta_{1}$
is algebraic. 
\end{proof}
\begin{rem}
\label{rem:relation between alg and surj}It is easy to come up with
surjective morphisms in $\mcg_{B}\left(\F\right)$ that are not algebraic:
consider, for instance, the morphism from $\xymatrix@1@C=15pt{\bullet\ar@/^{0.2pc}/[r]^{x} & \bullet\ar@/^{0.2pc}/[l]^{y}}$
to $\vphantom{\Big|}\xymatrix@1{\bullet\ar@(dl,ul)[]^{x}\ar@(dr,ur)[]_{y}}$.
Theorem \ref{thm:properties of algebraic morphisms}\eqref{enu:alg is surjective}
says that if $\eta\colon\H\to\J$ is an algebraic morphism in $\mocc\left(\F\right)$,
then it is surjective in \emph{$\mcg_{B}\left(\F\right)$ }for\emph{
any }basis $B$ of $\F$. It is a subtle matter to understand if this
has some converse -- see \cite{miasnikov2007algebraic,parzanchevski2014stallings,kolodner2021algebraic,ventura2021onto}.
\end{rem}

\subsection{The algebraic-free decomposition of morphisms\label{subsec:The-algebraic-free-decomposition}}
\begin{thm}
\label{thm:algebraic-free-decomposition}Let $\eta:\Gamma\to\Delta$
be a morphism of multi core graphs. Then there is a decomposition
\[
\xymatrix{\Gamma\ar[rr]_{\mathrm{algebraic}}^{\varphi}\ar@/^{2pc}/[rrrr]^{\eta} &  & \Sigma\ar[rr]_{\text{free}}^{\mathrm{\psi}} &  & \Delta}
\]
$\eta=\psi\circ\varphi$ such that $\varphi$ is algebraic and $\psi$
is a free. This decomposition is unique in the sense that if $\xymatrix{\Gamma\ar[r]_{\mathrm{alg}}^{\varphi'} & \Sigma'\ar[r]_{*}^{\mathrm{\psi'}} & \Delta}
$ is another such decomposition of $\eta$, then there is an isomorphism
of $\Sigma$ and $\Sigma'$ which commutes with the two decompositions
of $\eta$.

Moreover, this decomposition is universal in the following sense:
for every other decomposition $\xymatrix{\Gamma\ar[r]^{\varphi'} & \Sigma'\ar[r]^{\mathrm{\psi'}} & \Delta}
$ of $\eta$,
\begin{enumerate}
\item \label{enu:alg-free decomposition is universal for alg morphisms}if
$\varphi'$ is algebraic, $\varphi$ factors through $\varphi'$ (namely,
$\exists\theta$ with $\varphi=\theta\circ\varphi'$), and
\item if $\psi'$ is free, $\psi$ factors through $\psi'$ (namely, $\exists\theta'$
with $\psi=\psi'\circ\theta'$). 
\[
\xymatrix{ & \Sigma'\ar@/^{1pc}/[rrrd]\ar@{.>}[rd]^{\exists\theta}\\
\Gamma\ar[rr]_{\mathrm{alg}}^{\varphi}\ar@/^{1pc}/[ru]_{\mathrm{alg}}^{\varphi'}\ar@/_{1pc}/[rrrd] &  & \Sigma\ar[rr]_{*}^{\mathrm{\psi}}\ar@{.>}[rd]^{\exists\theta'} &  & \Delta\\
 &  &  & \Sigma''\ar@/_{1pc}/[ru]_{*}^{\psi'}
}
\]
\end{enumerate}
\end{thm}

Note that by definitions, in the first case $\theta$ must be algebraic,
and in the second case $\theta'$ must be free.
\begin{proof}
Assume we work in the category $\mcg_{B}\left(\F\right)$. Consider
all decompositions of $\eta$ as\linebreak{}
$\xymatrix{\Gamma\ar[r]^{\varphi} & \Sigma\ar[r]_{*}^{\mathrm{\psi}} & \Delta}
$ with $\psi$ free and $\varphi$ surjective (in the terminology of
Section \ref{sec:-Surjective-morphisms}, $\varphi$ is $B$-surjective).
There is at least one such decomposition: the decomposition of $\eta$
to a surjective and an injective morphisms. Euler characteristics
of multi core graphs are non-positive, and so among such decompositions,
we may pick one with $\chi\left(\Sigma\right)$ \emph{maximal}. We
fix this triple of $\Sigma,\varphi,\psi$. By Proposition \ref{prop:properties of free morphisms}\eqref{enu:EC of free morphisms},
$\varphi$ is algebraic. This proves the existence of the algebraic-free
decomposition.

Assume that $\xymatrix{\Gamma\ar[r]_{\mathrm{alg}}^{\varphi'} & \Sigma'\ar[r]_{*}^{\mathrm{\psi'}} & \Delta}
$ is another such decomposition. Let $\Lambda$ be the pullback of
$\psi$ and $\psi'$ as in the following diagram. By Lemma \ref{lem:pullback of free is free},
$\sigma$ and $\sigma'$ are free morphisms. But $\varphi$ and $\varphi'$
are algebraic, and so $\sigma$ and $\sigma'$ must be isomorphisms.
This proves the uniqueness of the algebraic-free decomposition.
\begin{equation}
\xymatrix{ &  & \Sigma'\ar@/^{1pc}/[rrd]_{\psi'}^{*}\\
\Gamma\ar@/_{1pc}/[rrd]_{\mathrm{alg}}^{\varphi}\ar@/^{1pc}/[rru]_{\varphi'}^{\mathrm{alg}}\ar@{.>}[r]^{\theta} & \Lambda\ar[ru]_{\sigma'}\ar[rd]^{\sigma} &  &  & \Delta\\
 &  & \Sigma\ar@/_{1pc}/[rru]_{*}^{\mathrm{\psi}}
}
\label{eq:diagram of proof of decomposition}
\end{equation}
Now let $\xymatrix{\Gamma\ar[r]_{\mathrm{alg}}^{\varphi'} & \Sigma'\ar[r]^{\mathrm{\psi'}} & \Delta}
$ be another decomposition of $\eta$ with $\varphi'$ algebraic, but
$\psi'$ not necessarily free. Let again $\left(\Lambda,\sigma,\sigma'\right)$
be the pullback of $\psi$ and $\psi'$ as in Diagram \eqref{eq:diagram of proof of decomposition}.
As $\psi$ is free, so is $\sigma'$, by Lemma \ref{lem:pullback of free is free}.
But $\varphi'$ is algebraic and so $\sigma'$ is an isomorphism,
and $\sigma\circ\left(\sigma'\right)^{-1}$ is the required morphism
$\Sigma'\to\Sigma$.

Finally, if $\xymatrix{\Gamma\ar[r]^{\varphi'} & \Sigma'\ar[r]_{*}^{\mathrm{\psi'}} & \Delta}
$ is another decomposition of $\eta$ with $\psi'$ free, but $\varphi'$
not necessarily algebraic, then by Lemma \ref{lem:pullback of free is free}
again, $\sigma$ (and $\sigma'$) are free. Let $\overline{\Lambda}$
be the image of $\theta$ in $\Lambda$ and $\overline{\sigma}$ the
restriction of $\sigma$ to $\overline{\Lambda}$. The embedding $\overline{\Lambda}\hookrightarrow\Lambda$
is free by Proposition \ref{prop:properties of free morphisms}\eqref{enu:Every-injective-morphism is free},
and so $\overline{\sigma}$, which is the composition of this embedding
with the free morphism $\sigma$, is free as well by \ref{prop:properties of free morphisms}\eqref{enu:free is transitive}.
As $\varphi$ is surjective, so is $\overline{\sigma}$. By the maximality
of $\chi\left(\Sigma\right)$ and Proposition \ref{prop:properties of free morphisms}\eqref{enu:EC of free morphisms},
$\overline{\sigma}$ must be an isomorphism, and so $\sigma'\circ\overline{\sigma}^{-1}$
is the sought-after morphism from $\Sigma$ to $\Sigma'$.
\end{proof}

\section{$B$-Surjective morphisms and norms of morphisms\label{sec:-Surjective-morphisms}}

Theorem \ref{thm:properties of algebraic morphisms}\eqref{enu:alg is surjective}
already manifested the importance of surjective morphisms of core
graphs. Yet, surjective morphisms play an even bigger role in this
work, and in the current section we develop some concepts that will
be useful in the following sections. Surjective morphisms of multi
core graphs is the analogue of the partial order ``$B$-cover''
defined in \cite[Definition 3.3]{PP15}. There, core graphs are connected
and have basepoints which must be mapped to basepoints by morphisms,
and so there is at most one morphism between two given subgroups $H,J\le\F$,
which exists if and only if $H\le J$. Thus, one can define a partial
order on subgroups of $\F$, which holds whenever $H\le J$ and the
corresponding morphism between core graphs is surjective. In contrast,
in the current paper, because multi core graphs are not necessarily
connected and do not have basepoints, there may be several different
morphisms between any two of them. Thus, being surjective is a property
of a morphism, and not of a pair of multi core graphs. As illustrated
in Remark \ref{rem:relation between alg and surj}, the property of
being surjective depends on the basis $B$. If one considers elements
in the category $\mocc\left(\F\right)$, one can say that a morphism
$\eta\colon\H\to\J$ is \emph{$B$-surjective }if the corresponding
morphism in $\mcg_{B}\left(\F\right)$ is surjective.

Let $\eta:\Gamma\to\Delta$ be a surjective morphism of multi core
graphs. Note that $\eta$ determines a partition $P$ of $V(\Gamma)$:
two vertices $v_{1},v_{2}$ are equivalent if and only if $\eta(v_{1})=\eta(v_{2})$.
This partition uniquely determines $\Delta$ and $\eta$: the vertices
of $\Delta$ correspond to the blocks of the partitions, and there
is a $b$-edge from block $\beta_{1}$ to block $\beta_{2}$ if and
only if there is a $b$-edge in $\Gamma$ from some vertex in $\beta_{1}$
to some vertex in $\beta_{2}$. It is clear how $\eta$ is defined.

However, not every partition of the vertex set of $\Gamma$ defines
a morphism, as the graph resulting from the above procedure may have
multiple edges of the same directed label incident to some vertex.
Yet, given such a partition $P$ of $V(\Gamma)$, we may define the
``generated'' multi core graph $\Delta$ and surjective morphism
$\eta:\Gamma\to\Delta$ via Stallings foldings, as follows. We start
the procedure with the $B$-labeled directed graph formed as above
by gluing together the vertices of $\Gamma$ according to the blocks
of $P$ and drawing edges between the blocks as above. Given two edges
of the same label and the same head, one may identify their tails
(if different) and the two edges. Similarly, given two edges of the
same label and same tail-vertex, one may identify their heads and
the two edges. Applying these identifications iteratively, a finite
number of times, yields a multi core graph $\Delta$, which is independent
of the choices made throughout the folding process\footnote{In our situation we never introduce leaves nor isolated vertices in
the process.} (see \cite{stallings1983topology}). This procedure yields a new
partition of the vertex set of $\Gamma$, which is coarser than $P$.
There is also only one reasonable way to define the morphism $\Gamma\to\Delta$,
by mapping every vertex to the block in the new partition it belongs
to, and every edge to the equally-labeled and equally-directed edge
between the corresponding blocks.
\begin{defn}
\label{def:norm of partition, surjective morphism}Let $P$ be a partition
of a finite set $X$. We define the norm of the partition as 
\[
||P||\defi\sum_{\beta\in P}\left(\left|\beta\right|-1\right).
\]
This is the minimal number of identifications of pairs of elements
of $X$ required in order to generate the partition.

Given a $B$-surjective morphism of multi core graphs $\eta\colon\Gamma\to\Delta$,
we define its \emph{$B$-norm}, denoted $\left\Vert \eta\right\Vert _{B}$,
to be the smallest norm of a partition generating it:
\begin{equation}
\left\Vert \eta\right\Vert _{B}\defi\min\left\{ \left\Vert P\right\Vert \,\middle|\,P~\mathrm{is~a~partition~of}~V\left(\Gamma\right)~\mathrm{generating}~\eta\right\} .\label{eq:B-norm of a morphism}
\end{equation}
One can also think of the $B$-norm as follows. Define a merging-step
of a multi core graph to be the gluing of two vertices of this graph
followed by folding. Then, $\left\Vert \eta\right\Vert _{B}$ is equal
to the smallest number of merging-steps which lead from $\Gamma$
to $\Delta$ to create $\eta$.
\end{defn}

We use here the notation $\left\Vert \cdot\right\Vert _{B}$ so that
we can use this notation also when $\eta$ is a ($B$-surjective)
morphism of two multisets of conjugacy classes of subgroups, in which
case the basis is not pre-determined. Because folding steps cannot
decrease the Euler characteristic of a graph and the gluing of two
vertices decreases the Euler characteristic by one, we obtain 
\begin{equation}
\left\Vert \eta\right\Vert _{B}\ge\chi\left(\Gamma\right)-\chi\left(\Delta\right).\label{eq:lower bound on combinatorial norm}
\end{equation}

There is also an ``algebraic'', basis-independent version of a norm
of morphisms in the categories $\mocc\left(\F\right)$ and $\mcg_{B}\left(\F\right)$.
We describe it gradually in the following lines.
\begin{defn}
\label{def:basis independent norm}Given a multiset $\H=\{H_{1}^{\F},\ldots,H_{\ell}^{\F}\}$
of conjugacy classes of subgroups of the free group $\F$ and a subgroup
$J\le\F$ such that $H_{i}\le J$ for all $i=1,\ldots,\ell$, let
$\eta\colon\H\to\left\{ J^{\F}\right\} $ be the corresponding morphism
in the category $\mocc\left(\F\right)$. Consider the two following
types of morphisms which we call \emph{immediate morphisms}:
\begin{enumerate}
\item Adding a generator from $J$ to one of the classes, namely, let $\H'$
be identical to $\H$ except for that some $H_{i}^{\F}$ is replaced
by $\left\langle H_{i},j\right\rangle ^{\F}$ for some $j\in J$.
The corresponding morphism $\varphi\colon\H\to\H'$ maps $H_{i}$
to $\left\langle H_{i},j\right\rangle $ and every other $H_{k}$
to itself. 
\item Merging together two of the classes, namely, let $\H'$ be identical
to $\H$ except for that for some $i\ne k$, $H_{i}^{\F}$ and $H_{k}^{\F}$
are replaced by $\left\langle jH_{i}j^{-1},H_{k}\right\rangle ^{\F}$
for some $j\in J$. The corresponding morphism $\varphi\colon\H\to\H'$
maps $jH_{i}j^{-1}$ and $H_{k}$ to $\left\langle jH_{i}j^{-1},H_{k}\right\rangle $
and every other $H_{t}$ to itself. 
\end{enumerate}
Note that in both cases $\eta$ factors through $\varphi$ by a unique
morphism $\eta'\colon\H'\to\J.$ Define the \emph{norm} of $\eta$,
denoted $\left\Vert \eta\right\Vert $, to be the smallest number
of immediate morphisms whose compositions gives $\eta$. If $\H=\emptyset$
is the empty multiset, we set\footnote{This convention makes sense in light of Lemma \ref{lem:distance-lower-bound}.}
$\left\Vert \eta\right\Vert =-\chi\left(J\right)=\rk J-1$.\\
Now let $\eta:\Gamma\to\Delta$ be a morphism of $B$-labeled multi
core graphs with $\Delta$ \uline{connected}. Define $\left\Vert \eta\right\Vert _{\mathrm{}}$
to be the norm of the corresponding morphism $\piol\left(\Gamma\right)\to\piol\left(\Delta\right)$.
Finally, for the general case where $\Delta$ is not necessarily
connected, let $\Delta_{1},\ldots,\Delta_{m}$ be the connected components
of $\Delta$, and for $k\in\left[m\right]$, let $\eta_{k}$ denote
the morphism $\eta^{-1}\left(\Delta_{k}\right)\to\Delta_{k}$ obtained
by restricting $\eta$ to $\eta^{-1}\left(\Delta_{k}\right)$. Define
$\left\Vert \eta\right\Vert \defi\sum_{k=1}^{m}\left\Vert \eta_{k}\right\Vert _{\mathrm{}}$.
\end{defn}

Note that $\left\Vert \eta\right\Vert $ is a non-negative integer,
but may be zero even if $\eta$ is not an isomorphism. Indeed, $\left\Vert \eta\right\Vert =0$
whenever $\eta\colon\Gamma\to\Delta$ induces an isomorphism between
$\Gamma$ and the connected components of $\Delta$ meeting the image
of $\eta$ and the remaining connected components of $\Delta$ correspond
to cyclic subgroups. One can give an equivalent definition for free
morphisms using the norm:
\begin{lem}
\label{lem:distance-lower-bound}Let $\eta\colon\Gamma\to\Delta$
be a morphism of multi core graphs. Then $\left\Vert \eta\right\Vert \geq\chi\left(\Gamma\right)-\chi\left(\Delta\right)$,
with equality if and only if $\eta$ is free.
\end{lem}

\begin{proof}
Any component $D$ of $\Delta$ not meeting $\eta\left(\Gamma\right)$
is trivially a free extension of its preimage, and contributes $-\chi\left(D\right)$
to $\left\Vert \eta\right\Vert $. Thus we may assume every component
of $\Delta$ meets $\eta\left(\Gamma\right)$. Now, the two types
of identifications from Definition \ref{def:basis independent norm}
cannot decrease the Euler characteristic of the element of $\mocc\left(\F\right)$
by more than one. Hence $\left\Vert \eta\right\Vert \ge\chi\left(\Gamma\right)-\chi\left(\Delta\right)$
with equality if and only if there is a set of identifications each
decreasing the Euler characteristic by exactly one. The first identification
decreases the Euler characteristic by one if and only if $\langle H_{i},j\rangle=H_{i}\ast\langle j\rangle$.
The second identification decreases the Euler characteristic by one
if and only if $\left\langle jH_{i}j^{-1},H_{k}\right\rangle =jH_{i}j^{-1}\ast H_{k}$.
In both cases, the corresponding morphism describing the step is free.
By transitivity of free morphisms, the claim follows.
\end{proof}
It is easy to see that a combinatorial merging-step as above where
we glue together two vertices from the same component, corresponds
to a step of the first kind from Definition \ref{def:basis independent norm},
and two vertices from different components to a step of the second
kind. Therefore, if $\eta\colon\Gamma\to\Delta$ is a surjective morphism
in $\mcg_{B}\left(\F\right)$ then 
\begin{equation}
\left\Vert \eta\right\Vert \le\left\Vert \eta\right\Vert _{B}.\label{eq:ineq of norms for surj morphism}
\end{equation}
An extension of the arguments from \cite[Section 3]{Puder2014} leads
to the following stronger statement:
\begin{thm}
\label{thm:alg distance  =00003D comb distance}Let $\eta\colon\Gamma\to\Delta$
be a morphism of $B$-labeled multi core graphs. Let $\Sigma=\mathrm{Image}\left(\eta\right)$
denote the image of $\eta$ in $\Delta$, and let $\xymatrix{\Gamma\ar@{->>}[r]^{\overline{\eta}} & \Sigma\ar@{^{(}->}[r]^{\iota} & \Delta}
$ be the decomposition of $\eta$ to a surjective and an injective
morphisms. Then 
\[
\left\Vert \eta\right\Vert =\left\Vert \overline{\eta}\right\Vert _{B}+\left[\chi\left(\Sigma\right)-\chi\left(\Delta\right)\right].
\]
In particular, if $\eta$ is $B$-surjective, then
\[
\left\Vert \eta\right\Vert =\left\Vert \eta\right\Vert _{B}.
\]
\end{thm}

Theorem \ref{thm:alg distance  =00003D comb distance} has some nice
corollaries we mention next. However, the theorem and its corollaries
are not needed for proving our main results from Section \ref{sec:Introduction},
and so we defer its proof to Appendix \ref{sec:norm-of-morphism-proof}.
The first corollary is immediate from Theorem \ref{thm:alg distance  =00003D comb distance}
together with Lemma \ref{lem:distance-lower-bound} and Proposition
\ref{prop:properties of free morphisms}\eqref{enu:Every-injective-morphism is free}.
\begin{cor}
Let $\eta\colon\Gamma\to\Delta$ be a morphism of multi core graphs.
Then, in the notation of Theorem \ref{thm:alg distance  =00003D comb distance},
\[
\eta~\mathrm{is~free}~~\iff~~\left\Vert \eta\right\Vert =\chi\left(\Gamma\right)-\chi\left(\Delta\right)~~\iff~~\left\Vert \overline{\eta}\right\Vert _{B}=\chi\left(\Gamma\right)-\chi\left(\Sigma\right)~~\iff~~\overline{\eta}~\mathrm{is~free}.
\]
\end{cor}

The second corollary concerns computability:
\begin{cor}
\label{cor:norm, freeness, algebraicity are computable}Given a morphism
$\eta\colon\Gamma\to\Delta$ of multi core graphs, there is an algorithm
to compute its norm $\left\Vert \eta\right\Vert $, and to determine
whether it is free and whether it is algebraic.
\end{cor}

\begin{proof}
By Theorem \ref{thm:alg distance  =00003D comb distance}, it is enough
to compute $\left\Vert \overline{\eta}\right\Vert _{B}$ to obtain
$\left\Vert \eta\right\Vert $, and the $B$-norm of a morphism is
obviously computable. By Lemma \ref{lem:distance-lower-bound}, it
is straightforward to determine whether $\eta$ is free given $\left\Vert \eta\right\Vert $.
Finally, if $\Gamma\stackrel{_{\eta_{1}}}{\longrightarrow}\Sigma\stackrel{_{\eta_{2}}}{\longrightarrow}\Delta$
is a decomposition of $\eta$ with $\eta_{2}$ free, we may assume
without loss of generality that $\eta_{1}$ is surjective (otherwise,
replace $\Sigma$ with $\eta_{1}\left(\Gamma\right)$, as $\eta_{2}\circ\left(\eta_{1}\left(\Sigma\right)\hookrightarrow\Sigma\right)$
is free as well). Because there are only finitely many surjective
morphisms from $\Gamma$, let alone decompositions $\Gamma\stackrel{_{\eta_{1}}}{\longrightarrow}\Sigma\stackrel{_{\eta_{2}}}{\longrightarrow}\Delta$
of $\eta$ with $\eta_{1}$ surjective, we may go through all of them
(such that $\eta_{2}$ is not an isomorphism) and test whether $\eta_{2}$
is free. Then $\eta$ is algebraic if and only if there is no such
decomposition with $\eta_{1}$ surjective and $\eta_{2}$ free.
\end{proof}
We end this section with an upper bound on the $B$-norm of morphisms,
to be used in Section \ref{sec:Expanders}.
\begin{prop}
\label{prop:upper bound on B-norm}The $B$-norm of a $B$-surjective
morphism $\eta\colon\Gamma\to\Delta$ satisfies 
\[
\left\Vert \eta\right\Vert _{B}\le\left[c\left(\Gamma\right)-c\left(\Delta\right)\right]+\rk\left(\Delta\right)=c\left(\Gamma\right)-\chi\left(\Delta\right).
\]
\end{prop}

\begin{proof}
Clearly, the $B$-norm is additive if we consider the different components
of $\Delta$, and so is the bound we give. So it is enough to prove
that when $\Delta$ is connected, $\left\Vert \eta\right\Vert _{B}\le c\left(\Gamma\right)-1+\rk\left(\Delta\right)$.
We prove this by induction on $c\left(\Gamma\right)$. If $c\left(\Gamma\right)=1$,
we are in the situation of \cite[Lemma 3.3]{Puder2014} which says
that $\left\Vert \eta\right\Vert _{B}\le\rk\left(\Delta\right)$.
If $c=c\left(\Gamma\right)\ge2$, then because $\eta$ is onto, there
must be a vertex in $\Delta$ which is in the image of at least two
different components of $\Gamma$. Merge such two vertices in two
different components of $\Gamma$ to obtain a graph $\Gamma'$ with
$c-1$ components and $\eta'\colon\Gamma'\to\Delta$. By induction,
\[
\left\Vert \eta\right\Vert \le1+\left\Vert \eta'\right\Vert \le1+\left(c-1-1\right)+\rk\left(\Delta\right)=c-1+\rk\left(\Delta\right).
\]
\end{proof}

\section{Möbius inversions and the leading terms of $\Phi$\label{sec:M=0000F6bius-inversions}}

Recall Definition \ref{def:Phi} of $\Phi_{\eta}\left(N\right)$,
and that our goal is to estimate $\Phi_{\eta}\left(N\right)$ for
certain morphisms of multi core graphs as in Example \ref{exa:main thm in terms of Phi}.
The main result of this section is Theorem \ref{thm:Phi-approximation}.
\begin{defn}
\label{def:chi-max}Let $\eta:\Gamma\to\Delta$ be a morphism of multi
core graphs. Denote 
\begin{equation}
\chimax\left(\eta\right)\defi\max\left\{ \chi\left(\Sigma\right)\,\middle|\,\begin{gathered}\Gamma\stackrel{_{\eta_{1}}}{\longrightarrow}\Sigma\stackrel{_{\eta_{2}}}{\longrightarrow}\Delta~\mathrm{is~a~decomposition~of~\eta}\\
\mathrm{with}~\eta_{1}~\mathrm{algebraic~and~non}\textnormal{-}\mathrm{isomorphism}
\end{gathered}
\right\} .\label{eq:def of chi-max}
\end{equation}
Every decomposition as in \eqref{eq:def of chi-max} and with $\chi\left(\Sigma\right)=\chimax\left(\eta\right)$
maximal, is called critical. Let $\crit(\eta)$ denote the set of
critical decompositions of $\eta$ up to equivalence as in Theorem
\ref{thm:algebraic-free-decomposition} (or as in Definition \ref{def:set of decompositions}
below).
\end{defn}

If $\Gamma$ is connected, $H\le\F$ a representative of the conjugacy
class $\piol\left(\Gamma\right)$, and $\eta\colon\Gamma\to X_{B}$,
then $\chimax\left(\eta\right)$ is equal to $1-\pi\left(H\right)$,
where $\pi\left(H\right)$ is the primitivity rank of $H$ (\cite[Definition 1.7]{PP15}).
Because algebraic morphisms are surjective (Theorem \ref{thm:properties of algebraic morphisms}\eqref{enu:alg is surjective})
and there are finitely many surjective morphisms with domain $\Gamma$,
$\crit\left(\eta\right)$ is always a finite set.
\begin{thm}
\label{thm:Phi-approximation}Let $\eta\colon\Gamma\to\Delta$ be
a morphism of multi core graphs. Then

\[
\Phi_{\eta}\left(N\right)=N^{\chi\left(\Gamma\right)}+\left|\crit(\eta)\right|\cdot N^{\chimax\left(\eta\right)}+O\left(N^{\chimax(\eta)-1}\right).
\]
\end{thm}

Theorem \ref{thm:Phi-approximation} is proven at the end of this
Section \ref{sec:M=0000F6bius-inversions}. When $\Gamma$ and $\Delta$
are connected, Theorem \ref{thm:Phi-approximation} reduces to \cite[Theorem 1.8]{PP15}
which is written in purely algebraic terms. As in the argument in
\cite{PP15}, the remaining ingredient of the proof is the definition
of certain Möbius inversions of the function $\Phi$ and the study
of their properties. The current section is devoted to an extension
of the arguments of \cite{PP15} to multi core graphs. In Section
\ref{subsec:Basis-dependent-M=0000F6bius}, we study Möbius inversions
based on decompositions of $B$-surjective morphisms. While an extension
of the Möbius inversions in \cite{PP15}, this is a bit unusual as
Möbius inversions are usually defined in terms of posets (partially
ordered sets) -- and see Remark \ref{rem:partial order on Decomp}.
In Section \ref{subsec:Algebraic-M=0000F6bius-Inversion} we introduce
Möbius inversions in decompositions in the category of algebraic morphisms,
which has no analogue in \cite{PP15}.
\begin{rem}
Note that the number $\chimax\left(\eta\right)$ and the set $\crit\left(\eta\right)$
from Definition \ref{def:chi-max} are algorithmically computable.
Indeed, every algebraic morphism is surjective (Proposition \ref{prop:C-algebraic-basis-formula}\eqref{enu:alg is surjective})
and so it is enough to go over the finitely many decompositions $\Gamma\stackrel{_{\eta_{1}}}{\longrightarrow}\Sigma\stackrel{_{\eta_{2}}}{\longrightarrow}\Delta$
of $\eta$ with $\eta_{1}$ surjective. By Corollary \ref{cor:norm, freeness, algebraicity are computable},
it is possible to determine whether $\eta_{1}$ is algebraic. Given
the finite set of such decompositions with $\eta_{1}$ algebraic,
it is straightforward to compute $\chimax\left(\eta\right)$ and $\crit\left(\eta\right)$.
See also Remark \ref{rem:a1...ak -pure algebraic morphisms} for some
more information about algebraic morphisms from $\Gamma_{\ak}^{w}$.
\end{rem}

\subsection{Basis dependent Möbius inversions\label{subsec:Basis-dependent-M=0000F6bius}}
\begin{defn}
\label{def:set of decompositions}Let $\eta\colon\Gamma\twoheadrightarrow\Delta$
be a $B$-surjective morphism. Denote by $\decomp\left(\eta\right)$
the set of decompositions of $\eta$ into two surjective morphisms
$\xymatrix{\Gamma\ar@{->>}[r]^{\eta_{1}} & \Sigma\ar@{->>}[r]^{\eta_{2}} & \Delta}
$, where the latter decomposition is considered identical to $\xymatrix{\Gamma\ar@{->>}[r]^{\eta'_{1}} & \Sigma'\ar@{->>}[r]^{\eta'_{2}} & \Delta}
$ if there is an isomorphism $\Sigma\cong\Sigma'$ which commutes with
both decompositions. 
\[
\xymatrix{\Gamma\ar@{->>}[r]^{\eta{}_{1}}\ar@{->>}[rd]_{\eta'{}_{1}} & \Sigma\ar@{<->}[d]_{\cong}^{\theta}\ar@{->>}[rd]^{\eta_{2}}\\
 & \Sigma'\ar@{->>}[r]_{\eta'_{2}} & \Delta
}
\]
Similarly, let $\decompt\left(\eta\right)$ denote the set of decompositions
$\xymatrix{\Gamma\ar@{->>}[r]^{\eta_{1}} & \Sigma_{1}\ar@{->>}[r]^{\eta_{2}} & \Sigma_{2}\ar@{->>}[r]^{\eta_{3}} & \Delta}
$ of $\eta$ into three surjective morphisms. Again, two such decompositions
are considered equivalent (and therefore the same element in $\decompt\left(\eta\right)$)
if there are isomorphisms $\Sigma_{i}\cong\Sigma'_{i}$, $i=1,2$,
which commute with the decompositions.
\end{defn}

Note that $\decomp\left(\eta\right)$ and $\decompt\left(\eta\right)$
are finite sets, as the multi core graph $\Gamma$ is finite. In another
point of view, $\decomp\left(\eta\right)$ can be thought of as the
set of all partitions of $V\left(\Gamma\right)$, the vertex-set of
$\Gamma$, which give rise (without folding) to valid multi core graphs,
and which are finer than (or equal to) the partition induced by $\eta$.
We remark that two distinct decompositions of $\eta$ may have isomorphic
$\Sigma$'s, as the morphisms $\Gamma\twoheadrightarrow\Sigma$ could
be distinct. Moreover, two distinct decompositions may be equivalent
up to an automorphism of $\Gamma$, as the following example illustrates.
\begin{example}
Let $H\leq\F$ be any non-trivial f.g.~subgroup. Consider the multi
core graphs $\Gamma_{n}$ consisting of $n$ disjoint copies of $\Gamma_{B}\left(H\right)$.
Denote by $\eta_{n}$ the unique morphism $\Gamma_{n}\twoheadrightarrow X_{B}$.
There are at least $\binom{n}{2}$ distinct decompositions in $\decomp(\eta_{n})$
with intermediate multi core graph $\Gamma_{n-1}$, corresponding
to a choice of a pair of components of $\Gamma_{n}$ which are identified
to a single component in $\Gamma_{n-1}$ (there may be more if $H\le uHu^{-1}$,
or equivalently $H=uHu^{-1}$, for some $u\in\F\setminus H$). All
of these decompositions are related by an automorphism of $\Gamma_{n}$,
permuting the connected components of $\Gamma_{n}$, but they are
distinct elements of $\decomp(\eta_{n})$.
\end{example}

We now define three different Möbius inversions of the function $\Phi$
which are defined on surjective morphisms in $\mcg_{B}\left(\F\right)$.
The following definition may seem a bit puzzling at first glance,
but as we explain right afterwords, it is indeed a valuable definition
of the functions $L^{B}$, $R^{B}$ and $C^{B}$. Recall the function
$\Phi$ from Definition \ref{def:Phi} and Proposition \ref{prop:geometric meaning for Phi},
which is defined on every morphism of multi core graphs.
\begin{defn}
\label{def:B-surjective Mob inversions}We define the \emph{left inversion}
of $\Phi$ on $B$-surjective morphisms, denoted $L^{B}$, by the
following equation that holds for every $B$-surjective morphism $\eta$,

\begin{equation}
\Phi_{\eta}=\sum_{\left(\eta_{1},\eta_{2}\right)\in\decomp\left(\eta\right)}L_{\eta_{2}}^{B}.\label{eq:def of L^B}
\end{equation}
Similarly, we define the right Möbius inversion $R^{B}$ by the following
equation holding for every $B$-surjective morphism $\eta$ 
\begin{equation}
\Phi_{\eta}=\sum_{\left(\eta_{1},\eta_{2}\right)\in\decomp\left(\eta\right)}R_{\eta_{1}}^{B}.\label{eq:def of R^B}
\end{equation}
Finally, the following equation for all $B$-surjective $\eta$ defines
the two-sided inversion $C^{B}$ of $\Phi$:
\begin{equation}
\Phi_{\eta}=\sum_{\left(\eta_{1},\eta_{2},\eta_{3}\right)\in\decompt\left(\eta\right)}C_{\eta_{2}}^{B}.\label{eq:def of C^B}
\end{equation}
\end{defn}

Indeed, \eqref{eq:def of L^B} well defines a map $L_{\eta}^{B}\colon\mathbb{Z}_{\ge1}\to\mathbb{Q}$
for every $B$-surjective morphism $\eta$ by induction on the size
of $\decomp\left(\eta\right)$. The base case is $L_{\mathrm{id}}^{B}\left(N\right)=\Phi_{\mathrm{id}}\left(N\right)=N^{\chi\left(\Gamma\right)}$.
For a general $B$-surjective $\eta$, note that $\left(\id,\eta\right)\in\decomp\left(\eta\right)$
and so 
\begin{equation}
L_{\eta}^{B}=\Phi_{\eta}-\sum_{\left(\eta_{1},\eta_{2}\right)\in\decomp\left(\eta\right)\setminus\left\{ \left(\mathrm{id},\eta\right)\right\} }L_{\eta_{2}}^{B}.\label{eq:L^B as a difference}
\end{equation}
For every element $\left(\eta_{1},\eta_{2}\right)\in\decomp\left(\eta\right)$
other than $\left(\mathrm{id},\eta\right)$, $\eta_{1}$ is not an
isomorphism, and so $\left|\decomp\left(\eta_{2}\right)\right|<\left|\decomp\left(\eta\right)\right|$.
Thus, the summation on the right hand side of \eqref{eq:L^B as a difference}
is on morphisms with a strictly smaller set of decompositions, and
the terms are well-defined by the induction hypothesis. 

A similar argument shows that \eqref{eq:def of R^B} and \eqref{eq:def of C^B}
well define $R^{B}$ and $C^{B}$, respectively. Note that $C^{B}$
is the right inversion of $L^{B}$ and the left inversion of $R^{B}$:
\begin{equation}
L_{\eta}^{B}=\sum_{\left(\eta_{1},\eta_{2}\right)\in\decomp\left(\eta\right)}C_{\eta_{1}}^{B}~~~~~~~~~R_{\eta}^{B}=\sum_{\left(\eta_{1},\eta_{2}\right)\in\decomp\left(\eta\right)}C_{\eta_{2}}^{B}\label{eq:C as left/right inversions of R/L}
\end{equation}

\begin{rem}
\label{rem:partial order on Decomp}One could define a partial order
on $\decomp\left(\eta\right)$ by setting $\left(\eta_{1},\eta_{2}\right)\le\left(\eta'_{1},\eta'_{2}\right)$
whenever there is a (necessarily surjective) morphism $\theta\colon\Sigma\to\Sigma'$
which makes the following diagram commute. 
\[
\xymatrix{\Gamma\ar@{->>}[r]^{\eta{}_{1}}\ar@{->>}[rd]_{\eta'{}_{1}} & \Sigma\ar@{->>}[d]^{\theta}\ar@{->>}[rd]^{\eta_{2}}\\
 & \Sigma'\ar@{->>}[r]_{\eta'_{2}} & \Delta
}
\]
Using this partial order, one could define the maps $L^{B},R^{B},C^{B}$
as Möbius inversions of a map defined on pairs of comparable elements
in a locally-finite \emph{poset}. This is the ordinary manner of defining
Möbius inversions. We chose a different language here which seems
more elegant.
\end{rem}

We turn to the study of $\Phi$ using the three Möbius inversions
$L^{B},R^{B},C^{B}$. Recall the geometric interpretation of $\Phi$
in Proposition \ref{prop:geometric meaning for Phi}. This gives rise
to a similar interpretation for $L^{B}$. Recall that $\left(N\right)_{s}=N\left(N-1\right)\ldots\left(N-s+1\right)$.
\begin{prop}
\label{prop:rational-function-for-L}Let $\eta\colon\Gamma\to\Delta$
be a $B$-surjective morphism. In the notation of Proposition \ref{prop:geometric meaning for Phi},
$L_{\eta}^{B}(N)$ is equal to the average number of \emph{injective}
lifts $\hat{\eta}\colon\Gamma\hookrightarrow\widehat{\Delta_{N}}$
of $\eta$. Moreover, for every large enough $N$, 
\begin{equation}
L_{\eta}^{B}(N)=\frac{\prod_{v\in V\left(\Delta\right)}\left(N\right)_{\left|\eta^{-1}\left(v\right)\right|}}{\prod_{e\in E\left(\Delta\right)}\left(N\right)_{\left|\eta^{-1}\left(e\right)\right|}}.\label{eq:rational exp for L}
\end{equation}
\end{prop}

\begin{proof}
By Proposition \ref{prop:geometric meaning for Phi}, $\Phi_{\eta}\left(N\right)$
is equal to the average number of lifts $\hat{\eta}\colon\Gamma\to\widehat{\Delta_{N}}$
of $\eta$. Every such lift can be written as the composition of a
surjection and an embedding.
\[
\xymatrix{ &  & \widehat{\Delta_{N}}\ar@{->>}[dd]^{p}\\
 & Im\hat{\eta}\ar@{^{(}->}[ru]^{\iota}\\
\Gamma\ar@{->>}[rr]_{\eta}\ar@/^{3pc}/[uurr]^{\hat{\eta}}\ar@{->>}[ru]^{\overline{\hat{\eta}}} &  & \Delta
}
\]
Note that the image $Im\hat{\eta}$ is a multi core graph. Because
$\left(p\circ\iota\right)\circ\overline{\hat{\eta}}$ is a decomposition
of the surjective $\eta$, the morphism $\left(p\circ\iota\right)$
is surjective too, and so $\left(\overline{\hat{\eta}},p\circ\iota\right)\in\decomp\left(\eta\right)$.
There is thus a one-to-one correspondence between the lifts $\hat{\eta}$
of $\eta$, and the union over all decompositions $\left(\eta_{1},\eta_{2}\right)\in\decomp\left(\eta\right)$
of injective lifts of $\eta_{2}$. Therefore,

\[
\Phi_{\eta}(N)=\mathbb{E}\left[\#\;\mathrm{lifts}\;\mathrm{of}\;\eta\right]=\sum_{\left(\eta_{1},\eta_{2}\right)\in\decomp\left(\eta\right)}\mathbb{E}\left[\mathrm{\#\;injective}\;\mathrm{lifts}\,\mathrm{of}\,\eta_{2}\right],
\]
and we conclude that, indeed, $L_{\eta_{2}}^{B}(N)$ is equal to the
number of injective lifts of $\eta_{2}$ to $\widehat{\Delta_{N}}$.

It remains to prove the right hand side of \eqref{eq:rational exp for L}
gives the average number of injective lifts of $\eta\colon\Gamma\to\Delta$.
First we embed the vertices of $\Gamma$ in $\widehat{\Delta_{N}}$.
For every $v\in V\left(\Delta\right)$, the fiber $\eta^{-1}\left(v\right)$
should be embedded in the fiber $p^{-1}\left(v\right)$ which is of
size $N$, and there are $(N)_{|\eta^{-1}(v)|}$ choices for such
an embedding. Second, for every edge $e\in E(\Delta)$, we obtain
$\eta^{-1}(e)$ restrictions on the permutation $\sigma_{e}$. Such
a set of conditions occurs with probability $\frac{\left(N-\left|\eta^{-1}\left(e\right)\right|\right)!}{N!}=\frac{1}{\left(N\right)_{\left|\eta^{-1}\left(e\right)\right|}}$.
This implies the claim.
\end{proof}
\begin{cor}
\label{cor:Phi, L, R, C all rational}If $\eta\colon\Gamma\to\Delta$
is a $B$-surjective morphism, then $\Phi_{\eta}$, $L_{\eta}^{B}$,
$R_{\eta}^{B}$ and $C_{\eta}^{B}$ are all rational functions in
$N$ for every large enough $N$.
\end{cor}

\begin{proof}
For $L_{\eta}^{B}$ this follows directly from Proposition \ref{prop:rational-function-for-L}.
The other three functions are equal to finite sums of $L_{\psi}^{B}$
with certain $B$-surjective morphisms $\psi$.
\end{proof}
We now develop an alternate expression for the right hand side of
\eqref{eq:rational exp for L}, in order to obtain an expression for
the double sided Möbius inversion $C_{\eta}^{B}$. 
\begin{defn}
\label{def:norm-of-perm, number-of-preserving-perms}Let $X$ be a
finite set. Define the norm $||\sigma||$ of a permutation $\sigma\in Sym(X)$
as the minimal length of a product of transpositions which gives $\sigma$.
Equivalently, $\left\Vert \sigma\right\Vert =\sum_{c}\left[\mathrm{len}\left(c\right)-1\right]$,
the sum being on the cycles of $\sigma$. Also, $\left\Vert \sigma\right\Vert $
is equal to the norm (as in Definition \ref{def:norm of partition, surjective morphism})
of the partition of $X$ induced by the cycles of $\sigma$.

If, in addition, $Y$ is also a set and $\varphi:X\to Y$ some map,
let 
\[
\left[X\right]_{j}^{\varphi}\defi\left|\left\{ \sigma\in Sym(X)\,\middle|\,\varphi\circ\sigma=\varphi,||\sigma||=j\right\} \right|
\]
denote the number of $\varphi$-preserving permutations of $X$ of
norm $j$. Note that this number depends only on the partition induced
by $\varphi$ on $X$. 
\end{defn}

\begin{prop}
\label{prop:L-cover}Let $\eta\colon\Gamma\to\Delta$ be a $B$-surjective
morphism. Then 
\begin{equation}
L_{\eta}^{B}\left(N\right)=\sum_{t\geq0}\sum_{\substack{j_{0}\geq0\\
j_{1},...,j_{t}\geq1
}
}(-1)^{t+\sum_{i=0}^{t}j_{i}}\left[V(\Gamma)\right]_{j_{0}}^{\eta}\cdot\left[E(\Gamma)\right]_{j_{1}}^{\eta}\cdot...\cdot\left[E(\Gamma)\right]_{j_{t}}^{\eta}N^{\chi(\Gamma)-\sum j_{i}}\label{eq:Laurent for L}
\end{equation}
\end{prop}

\begin{proof}
This is the same as \cite[Section 7.1]{PP15} -- we repeat here briefly
a sketch of the argument. We use the identity $\left(N\right)_{k}=N^{k}\sum_{j=0}^{k}\left(-1\right)^{j}\left[k\right]_{j}N^{-j}$,
where $\left[k\right]_{j}$ denotes the number of permutations in
$S_{k}$ with norm $j$. Multiplying this identity for the sets $\eta^{-1}(v)$,
we obtain 
\[
\prod_{v\in V\left(\Delta\right)}\left(N\right)_{\left|\eta^{-1}\left(v\right)\right|}=N^{\left|V\left(\Gamma\right)\right|}\sum_{j=0}^{\left|V\left(\Gamma\right)\right|}\left(-1\right)^{j}\left[V\left(\Gamma\right)\right]_{j}^{\eta}N^{-j},
\]
since an $\eta$-preserving permutation decomposes uniquely as a product
of permutations in the $\eta$-fibers. Similarly, 
\[
\prod_{e\in E\left(\Delta\right)}\left(N\right)_{\left|\eta^{-1}\left(e\right)\right|}=N^{\left|E\left(\Gamma\right)\right|}\sum_{j=0}^{\left|E\left(\Gamma\right)\right|}\left(-1\right)^{j}\left[E\left(\Gamma\right)\right]_{j}^{\eta}N^{-j}.
\]
Combined with \eqref{eq:rational exp for L}, we get 
\[
L_{\eta}^{B}(N)=N^{\chi\left(\Gamma\right)}\frac{\sum_{j=0}^{\left|V\left(\Gamma\right)\right|}\left(-1\right)^{j}\left[V\left(\Gamma\right)\right]_{j}^{\eta}N^{-j}}{\sum_{j=0}^{\left|E\left(\Gamma\right)\right|}\left(-1\right)^{j}\left[E\left(\Gamma\right)\right]_{j}^{\eta}N^{-j}}.
\]
Using the fact that
\[
\frac{1}{1+\sum_{i\geq1}a_{i}N^{-i}}=\sum_{t\geq0}\left(-\sum a_{i}N^{-i}\right)^{t}=\sum_{t\geq0}\left(-1\right)^{t}\sum_{j_{1},...,j_{t}\geq1}a_{j_{1}}\cdot...\cdot a_{j_{t}}N^{-\sum j_{i}},
\]
the claim follows.
\end{proof}
We use this expression in order to obtain a combinatorial interpretation
for $C_{\eta}^{B}(N)$, which then implies the following Theorem. 
\begin{thm}
\label{thm:C-upper-bound}Let $\eta\colon\Gamma\to\Delta$ be a $B$-surjective
morphism. Then 
\[
C_{\eta}^{B}(N)=O\left(N^{\chi(\Gamma)-\left\Vert \eta\right\Vert _{B}}\right).
\]
\end{thm}

\begin{proof}
We give a sketch of the analysis carried out in more detail in \cite[Section 7.1]{PP15}.
Recall that $C^{B}$ is the right Möbius inversion of $L^{B}$. Our
starting point is the expression \eqref{eq:Laurent for L} for $L_{\eta}^{B}$.
A permutation of the vertex set $V\left(\Gamma\right)$ induces a
partition of the vertex set. Identifying the blocks of this partition
and folding, gives rise to a $B$-surjective morphism $\eta_{1}:\Gamma\to\Sigma$.
If the permutation is $\eta$-preserving, $\eta_{1}$ defines a partition
which refines the partition of $\eta$, and then there is a $B$-surjective
morphism $\eta_{2}\colon\Sigma\to\Delta$ with $\eta=\eta_{2}\circ\eta_{1}$.

Similarly, an $\eta$-preserving permutation of the edge-set $E\left(\Gamma\right)$
induces a natural $B$-surjective morphism which is the first half
of a decomposition of $\eta$. This can be seen by gluing every two
edges in the same cycle of the permutation, and then folding. This
gluing is equivalent to gluing together the origins of the two edges,
or equivalently their termini, since the permutation is $\eta$-preserving
and in particular preserves edge labels and directions.

Given both a vertex permutation and a sequence of edge permutations
of $\Gamma$, we may glue along all of these permutations, and then
fold in order to obtain a $B$-surjective morphism $\eta_{1}:\Gamma\to\Sigma$
corresponding to a partition refining the $\eta$-partition of $V\left(\Gamma\right)$.
This implies that every term in \eqref{eq:Laurent for L} can be attributed
to an element of $\decomp(\eta)$. By collecting all the terms from
\eqref{eq:Laurent for L} corresponding to every $\left(\eta_{1},\eta_{2}\right)\in\decomp\left(\eta\right)$
and denoting their sum by $\tilde{C}_{\eta,\eta_{1}}^{B}(N)$, we
obtain an expression 
\begin{align*}
L_{\eta}^{B}(N) & =\sum_{\left(\eta_{1},\eta_{2}\right)\in\decomp(\eta)}\tilde{C}_{\eta,\eta_{1}}^{B}(N).
\end{align*}
We will prove that $\tilde{C}_{\eta,\eta_{1}}^{B}(N)$ depends only
on $\eta_{1}$ (and not on $\eta$), implying that $C_{\eta,\eta_{1}}^{B}$
is the right Möbius inversion of $L^{B}$ as in \eqref{eq:C as left/right inversions of R/L},
and therefore equal to the central inversion $C_{\eta_{1}}^{B}$.
On the other hand, the expression $\tilde{C}_{\eta,\eta_{1}}^{B}(N)$
was obtained as a signed sum of expressions of the form $N^{\chi(\Gamma)-\sum j_{i}}$
where $\sum j_{i}\geq\left\Vert \eta_{1}\right\Vert _{B}$, since
this number of identifications yields $\eta_{1}$. This will imply
the claim.

It remains to prove that $\tilde{C}_{\eta,\eta_{1}}^{B}(N)$ is indeed
independent of $\eta$. This $\tilde{C}_{\eta,\eta_{1}}^{B}(N)$ was
obtained as a sum over sequences of $\eta$-preserving vertex and
edge permutations generating $\eta_{1}$ (with Stallings foldings).
Note that such a sequence of permutations is then also $\eta_{1}$-preserving.
This implies that we can equivalently describe this set of permutations
as sequences of $\eta_{1}$-preserving permutations generating $\eta_{1}$
after folding. Hence, the expression depends only on $\eta_{1}$.
\end{proof}

\subsection{Algebraic Möbius Inversion\label{subsec:Algebraic-M=0000F6bius-Inversion}}

We also work with Möbius inversion based on algebraic morphisms. This
has no direct parallel in \cite{PP15}.
\begin{defn}
\label{def:alg-decompositions}For an algebraic morphism $\eta\colon\Gamma\to\Delta$
in $\mcg_{B}\left(\F\right)$ or, equivalently, in $\mocc\left(\F\right)$,
denote by $\algdecomp(\eta)$ and $\algdecompt\left(\eta\right)$
the set of decompositions of $\eta$ into two (three, respectively)
algebraic morphisms, with the same identifications as in Definition
\ref{def:set of decompositions}. We also define the algebraic left,
right and central Möbius inversions of $\Phi$ (restricted to algebraic
morphisms), denoted $L_{\eta}^{\mathrm{alg}}(N),R_{\eta}^{\mathrm{alg}}(N),C_{\eta}^{\mathrm{alg}}(N)$,
respectively, by analogy with Definition \ref{def:B-surjective Mob inversions}.
For instance, $L^{\mathrm{alg}}$ is defined by
\[
\Phi_{\eta}=\sum_{\left(\eta_{1},\eta_{2}\right)\in\algdecomp\left(\eta\right)}L_{\eta_{2}}^{\mathrm{alg}}.
\]
\end{defn}

Recall, by Theorem \ref{thm:properties of algebraic morphisms}\eqref{enu:alg is surjective},
that if $\eta$ is algebraic, then $\algdecomp\left(\eta\right)\subseteq\decomp\left(\eta\right)$.
\begin{prop}
\label{prop:R-basis-independence}The right Möbius inversion $R^{B}$
is supported on algebraic morphisms, and on those it is equal to $R^{\mathrm{alg}}$
(in particular, it is independent of the basis $B$).
\end{prop}

\begin{proof}
We prove all claims together by induction on the size of $\decomp(\eta)$.
Note that the base case is $\id\colon\Gamma\to\Gamma$, which is algebraic,
and $R_{\id}^{B}(N)=R_{\id}^{\mathrm{alg}}\left(N\right)=\Phi_{\id}(N)=N^{\chi(\Gamma)}$,
which is basis independent. For the general case,
\begin{eqnarray}
R_{\eta}^{B}(N) & = & \Phi_{\eta}(N)-\sum_{\left(\eta_{1},\eta_{2}\right)\in\decomp(\eta):~\eta_{2}~\mathrm{non-isomorphism}}R_{\eta_{1}}^{B}(N)\nonumber \\
 & = & \Phi_{\eta}(N)-\sum_{\left(\eta_{1},\eta_{2}\right)\in\decomp(\eta):~\eta_{1}\ \mathrm{algebraic,\ \eta_{2}~\mathrm{non}-\mathrm{isomorphism}}}R_{\eta_{1}}^{\mathrm{alg}}(N),\label{eq:induction for R}
\end{eqnarray}
where the second equality is by the induction hypothesis. If $\eta$
is algebraic, the pairs $\left(\eta_{1},\eta_{2}\right)$ in the summation
in right hand side of \eqref{eq:induction for R} are exactly those
in $\algdecomp\left(\eta\right)\setminus\left\{ \left(\eta,\id\right)\right\} $
and so the entire summation is equal to $R_{\eta}^{\mathrm{alg}}\left(N\right)$.

Finally, assume that $\eta$ is not algebraic. Consider the unique
decomposition $\xymatrix{\Gamma\ar[r]_{\mathrm{alg}}^{\varphi} & \Sigma\ar[r]_{*}^{\mathrm{\psi}} & \Delta}
$ of $\eta$ into an algebraic morphism $\varphi$ and a free one $\psi$,
as in Theorem \ref{thm:algebraic-free-decomposition}. As $\eta$
is $B$-surjective, so is $\psi$. By the same Theorem \ref{thm:algebraic-free-decomposition},
for every $\left(\eta_{1},\eta_{2}\right)\in\decomp(\eta)$ with $\eta_{1}$
algebraic, there is (a unique) $\overline{\eta_{2}}$ so that $\left(\eta_{1},\overline{\eta_{2}}\right)\in\algdecomp\left(\varphi\right)$.
Thus from \eqref{eq:induction for R} we derive
\begin{eqnarray*}
R_{\eta}^{B}(N) & = & \Phi_{\eta}(N)-\sum_{(\eta_{1},\overline{\eta_{2}})\in\algdecomp\left(\varphi\right)}R_{\eta_{1}}^{\mathrm{alg}}(N)\\
 & \stackrel{\mathrm{Proposition}~\ref{prop:properties of free morphisms}\eqref{enu:Phi can ignore free extensions}}{=} & \Phi_{\varphi}(N)-\sum_{(\eta_{1},\overline{\eta_{2}})\in\algdecomp\left(\varphi\right)}R_{\eta_{1}}^{\mathrm{alg}}(N)=0.
\end{eqnarray*}
\end{proof}
\begin{prop}
\label{prop:C-algebraic-basis-formula}Let $\eta$ be an algebraic
morphism. Then for every basis $B$,
\[
C_{\eta}^{\mathrm{alg}}(N)=\sum_{(\eta_{1},\eta_{2})\in\decomp(\eta):\ \eta_{1}\ \mathrm{is~free}}C_{\eta_{2}}^{B}(N).
\]
\end{prop}

\begin{proof}
Denote the right hand side of the above equality by $F_{\eta}(N)$
for the course of the proof. For a morphism $\gamma$ denote by $\mathrm{alg}\left(\gamma\right)$
and $\mathrm{free}\left(\gamma\right)$ the morphisms in the unique
decomposition of $\gamma$ into algebraic and free morphisms given
by Theorem \ref{thm:algebraic-free-decomposition}. Let $\eta$ be
algebraic. For every $\left(\eta_{1},\eta_{2}\right)\in\decomp(\eta)$,
consider the decomposition $\xymatrix{\bullet\ar[r]_{\mathrm{alg}}^{\mathrm{alg}\left(\eta_{1}\right)} & \bullet\ar[r]_{*}^{\mathrm{free}\left(\eta_{1}\right)} & \bullet\ar[r]^{\eta_{2}} & \bullet}
$ of $\eta$. Because $\eta$ is algebraic, so is $\beta\defi\eta_{2}\circ\mathrm{free}\left(\eta_{1}\right)$.
Thus,
\begin{eqnarray*}
R_{\eta}^{\mathrm{alg}}(N) & \stackrel{\mathrm{Prop.}~\ref{prop:R-basis-independence}}{=} & R_{\eta}^{B}(N)=\sum_{\left(\eta_{1},\eta_{2}\right)\in\decomp(\eta)}C_{\eta_{2}}^{B}(N)\\
 & = & \sum_{\left(\alpha,\beta\right)\in\algdecomp\left(\eta\right)}\left[\sum_{\substack{\left(\eta_{1},\eta_{2}\right)\in\decomp(\eta):\\
\eta_{2}\circ\mathrm{free}\left(\eta_{1}\right)=\beta
}
}C_{\eta_{2}}^{B}(N)\right]\\
 & = & \sum_{\left(\alpha,\beta\right)\in\algdecomp\left(\eta\right)}\left[\sum_{\substack{\left(\eta'_{1},\eta_{2}\right)\in\decomp(\beta):\\
\eta'_{1}~\mathrm{is~free}
}
}C_{\eta_{2}}^{B}(N)\right]=\sum_{\left(\alpha,\beta\right)\in\algdecomp\left(\eta\right)}F_{\beta}\left(N\right).
\end{eqnarray*}
This implies the claim, by definition of the Möbius inversion.
\end{proof}
\begin{cor}
\label{cor:C-algebraic-upper-bound}Let $\eta:\Gamma\to\Delta$ be
an algebraic morphism. Then
\[
C_{\eta}^{\mathrm{alg}}\left(N\right)=\begin{cases}
N^{\chi\left(\Gamma\right)} & \mathrm{if}~\eta~\mathrm{is~an~isomorphism,}\\
O\left(N^{\chi\left(\Gamma\right)-\left\Vert \eta\right\Vert }\right)\le O\left(N^{\chi\left(\Delta\right)-1}\right) & \mathrm{otherwise}.
\end{cases}
\]
\end{cor}

\begin{proof}
If $\eta$ is an isomorphism, then $C_{\eta}^{\mathrm{alg}}\left(N\right)=C_{\id}^{\mathrm{alg}}\left(N\right)=\Phi_{\id}\left(N\right)=N^{\chi\left(\Gamma\right)}$,
as noted in Example \ref{exa:Phi for id}. Otherwise, 
\begin{eqnarray}
C_{\eta}^{\mathrm{alg}}(N) & \stackrel{\mathrm{Prop.}~\ref{prop:C-algebraic-basis-formula}}{=} & \sum_{\left(\eta_{1},\eta_{2}\right)\in\decomp(\eta),\ \eta_{1}\ \mathrm{free}}C_{\eta_{2}}^{B}(N)\nonumber \\
 & \stackrel{\mathrm{Thm.}~\ref{thm:C-upper-bound}}{=} & \sum_{\left(\eta_{1},\eta_{2}\right)\in\decomp(\eta),\ \eta_{1}\ \mathrm{free}}O\left(N^{\chi\left(\mathrm{Im}\left(\eta_{1}\right)\right)-\left\Vert \eta_{2}\right\Vert _{B}}\right)\nonumber \\
 & \stackrel{\mathrm{Lemma}~\ref{lem:distance-lower-bound}}{=} & \sum_{\left(\eta_{1},\eta_{2}\right)\in\decomp(\eta),\ \eta_{1}\ \mathrm{free}}O\left(N^{\chi\left(\Gamma\right)-\left\Vert \eta_{1}\right\Vert -\left\Vert \eta_{2}\right\Vert _{B}}\right)\nonumber \\
 & \stackrel{\eqref{eq:ineq of norms for surj morphism}}{=} & \sum_{\left(\eta_{1},\eta_{2}\right)\in\decomp(\eta),\ \eta_{1}\ \mathrm{free}}O\left(N^{\chi\left(\Gamma\right)-\left\Vert \eta_{1}\right\Vert -\left\Vert \eta_{2}\right\Vert }\right).\label{eq:upper bound on C-alg inside proof}
\end{eqnarray}
Finally, by the very definition of the basis-independent norm of morphisms
in Definition \ref{def:basis independent norm}, it is clear that
$\left\Vert \eta\right\Vert \le\left\Vert \eta_{1}\right\Vert +\left\Vert \eta_{2}\right\Vert $
for every $\left(\eta_{1},\eta_{2}\right)\in\decomp(\eta)$. Hence
every term in \eqref{eq:upper bound on C-alg inside proof} is at
most $O\left(N^{\chi\left(\Gamma\right)-\left\Vert \eta\right\Vert }\right)$,
which is at most $O\left(N^{\chi\left(\Delta\right)-1}\right)$ by
Lemma \ref{lem:distance-lower-bound} as $\eta$ is not free. This
completes the proof as the summation in \eqref{eq:upper bound on C-alg inside proof}
is finite.
\end{proof}
We can now prove the main result of this section. Recall that $\eta\colon\Gamma\to\Delta$,
and we ought to show that 
\begin{equation}
\Phi_{\eta}\left(N\right)=N^{\chi\left(\Gamma\right)}+\left|\crit(\eta)\right|\cdot N^{\chimax\left(\eta\right)}+O\left(N^{\chimax(\eta)-1}\right).\label{eq:what we need to show in approx of Phi}
\end{equation}

\begin{proof}[Proof of Theorem \ref{thm:Phi-approximation}]
 We may assume without loss of generality that $\eta$ is algebraic.
Otherwise, the decomposition of $\eta$ into an algebraic morphism
$\varphi$ and a free morphism $\psi$ given by Theorem \ref{thm:algebraic-free-decomposition},
satisfies that $\Phi_{\eta}=\Phi_{\varphi}$ by Proposition \ref{prop:properties of free morphisms}\eqref{enu:Phi can ignore free extensions},
and that $\chimax\left(\eta\right)=\chimax\left(\varphi\right)$ and
$\left|\crit\left(\eta\right)\right|=\left|\crit\left(\varphi\right)\right|$
by Theorem \ref{thm:algebraic-free-decomposition}.

So assume that $\eta$ is algebraic. We have $\Phi_{\eta}(N)=\sum_{\left(\eta_{1},\eta_{2},\eta_{3}\right)\in\algdecompt\left(\eta\right)}C_{\eta_{2}}^{\mathrm{alg}}(N)$.
If $\eta_{2}=\id$ the contribution is $C_{\id\colon\mathrm{Im}\left(\eta_{1}\right)\to\mathrm{Im}\left(\eta_{1}\right)}^{\mathrm{alg}}\left(N\right)=N^{\chi\left(\mathrm{Im}\left(\eta_{1}\right)\right)}$,
and these contributions give rise to the first two terms in \eqref{eq:what we need to show in approx of Phi}
plus $O\left(N^{\chimax(\eta)-1}\right)$. In any other decomposition
$\left(\eta_{1},\eta_{2},\eta_{3}\right)\in\algdecompt\left(\eta\right)$,
Corollary \ref{cor:C-algebraic-upper-bound} yields that
\[
C_{\eta_{2}}^{\mathrm{alg}}\left(N\right)=O\left(N^{\chi\left(\mathrm{Im}\left(\eta_{2}\right)\right)-1}\right)=O\left(N^{\chi\left(\mathrm{Im}\left(\eta_{2}\circ\eta_{1}\right)\right)-1}\right)=O\left(N^{\chimax\left(\eta\right)-1}\right).
\]
\end{proof}
We end this section with the following full analysis of algebraic
and $B$-surjective morphisms in rank one free group.
\begin{lem}
\label{lem:Phi-rank-one}Let $\F_{1}\cong\mathbb{Z}$ with basis $B=\left\{ b\right\} $.
Let $\eta:\mathcal{H}\to\mathcal{J}$ be a morphism in $\mocc\left(\F_{1}\right)$
such that the image of $\eta$ meets every element of the multiset
$\J$. Then $\eta$ is algebraic, and for all large enough $N$,

\begin{equation}
C_{\eta}^{\mathrm{alg}}(N)=\begin{cases}
1 & \mathrm{if}~\eta=\id,\\
0 & \mathrm{otherwise}
\end{cases},\label{eq:C in F_1}
\end{equation}
$L_{\eta}^{\mathrm{alg}}(N)=1$ and $\Phi_{\eta}(N)=|\algdecomp(\eta)|=|\decomp(\eta)|$.
\end{lem}

\begin{proof}
Recall that the elements in the multisets in $\mocc\left(\F_{1}\right)$
are conjugacy classes of non-trivial subgroups. Every non-trivial
subgroup of $\F_{1}$ is of rank $1$. Hence, by Proposition \ref{prop:properties of free morphisms}\eqref{enu:EC of free morphisms},
there are no free morphisms in $\mocc\left(\F_{1}\right)$ in which
the image meets every component of the codomain, except for isomorphisms.
The definition of algebraic morphisms now implies that every morphism
in $\mocc\left(\F_{1}\right)$ with image meeting every component
of the codomain is algebraic. As every algebraic morphism is $B$-surjective,
we obtain that $\decomp(\eta)=\algdecomp(\eta)$ and so $L_{\eta}^{\mathrm{alg}}(N)=L_{\eta}^{B}(N)$
and $C_{\eta}^{\mathrm{alg}}(N)=C_{\eta}^{B}(N)$.

Next we prove that $L_{\eta}^{B}(N)=1$. By Proposition \ref{prop:rational-function-for-L},
$L^{B}$ is multiplicative on the elements of $\J$ and it is thus
enough to prove that $L_{\eta}^{B}(N)=1$ when $\J$ is a singleton.
But then $\J=\{\left\langle b^{j}\right\rangle ^{\F_{1}}\}$ for some
$j$, and every component of $\H$ is $\langle b^{jm}\rangle^{\F_{1}}$
for some $m\in\mathbb{Z}_{\ge1}$. In particular, the morphism of
$B$-labeled multi core graphs is a topological covering map, and
every vertex and every edge in $\Gamma_{B}\left(\J\right)$ have fiber
of the same size. It now follows from \eqref{eq:rational exp for L}
that $L_{\eta}^{B}(N)=1$.

That $\Phi_{\eta}(N)=|\algdecomp(\eta)|=|\decomp(\eta)|$ follows
immediately, and \eqref{eq:C in F_1} follows by considering $C^{\mathrm{alg}}$
as the right Möbius inversion of $L^{\mathrm{alg}}$ and a simple
induction on $\left|\decomp\left(\eta\right)\right|$.
\end{proof}

\section{The proof of Theorem \ref{thm:Word-Measure-Character-Bound}\label{sec:Proof-of-Theorem}}

Throughout this section, we fix a non-power $1\ne w\in\F$. Recall
from Example \ref{exa:main thm in terms of Phi} that the function
$\mathbb{E}_{w}\left[\qak\right]$ in the center of Theorem \ref{thm:Word-Measure-Character-Bound}
is equal to $\Phi_{\eta_{\ak}^{w}}$, where $\eta_{\ak}^{w}\colon\Gamma_{\ak}^{w}\to X_{B}$.

When $k=1$ and $\alpha_{1}=1$, this is the special case from Theorem
\ref{thm:PP15}, that was proven in \cite{PP15}, and is immediate
from Theorem \ref{thm:Phi-approximation}. However, in any other case,
Theorem \ref{thm:Phi-approximation} as is does not teach us anything
new. Indeed, if $\alpha_{1}+2\alpha_{2}+\ldots+k\alpha_{k}\ge2$,
then there is $\left(\varphi_{1},\varphi_{2}\right)\in\decomp\left(\eta_{\ak}^{w}\right)$
with $\varphi_{1}$ algebraic (by Lemma \ref{lem:Phi-rank-one}) and
non-isomorphism, so that $\mathrm{Im}\left(\varphi_{1}\right)=\Gamma_{B}(\left\langle w\right\rangle ^{\F})$.
In particular, $\chimax\left(\eta_{\ak}^{w}\right)=0$, so Theorem
\ref{thm:Phi-approximation} says only that 
\[
\mathbb{E}_{w}\left[\qak\right]=\left[1+\crit\left(\eta_{\ak}^{w}\right)\right]+O\left(N^{-1}\right).
\]
This agrees with the statement of Theorem \ref{thm:Word-Measure-Character-Bound}:
as we explain below, $\left\langle \qak,1\right\rangle =1+\crit\left(\eta_{\ak}^{w}\right)$.
But this is only the easier part of this theorem (that also follows
from \cite{nica1994number,Linial2010}). To establish Theorem \ref{thm:Word-Measure-Character-Bound}
in full, we need some more machinery. We start with the following
twist on Definition \ref{def:chi-max} of $\chimax$ and of $\crit$,
which considers only \textbf{negative} Euler characteristics. Because
the codomain of $\eta_{\ak}^{w}$ is $X_{B}$, any morphism $\Gamma_{\ak}^{w}\to\Sigma$
is part of a decomposition of $\eta_{\ak}^{w}$.
\begin{defn}
\label{def:chi-max for multiset of cycles}For a non-power $1\ne w\in\F$
and $k\ge1$, $\ak\ge0$, let $\Gamma_{\ak}^{w}$ and $\eta_{\ak}^{w}$
denote the corresponding multi core graph and morphism as above, and
let
\[
\chiak\left(w\right)\defi\max\left\{ \chi\left(\Sigma\right)\,\middle|\,\Gamma_{\ak}^{w}\stackrel{_{\varphi}}{\longrightarrow}\Sigma~\mathrm{is~algebraic~with}~\chi\left(\Sigma\right)<0\right\} .
\]
Denote by $\critak\left(w\right)$ the set of algebraic morphisms
with domain $\Gamma_{\ak}^{w}$ and codomain of Euler characteristic
$\chiak\left(w\right)$, up to equivalence as in Theorem \ref{thm:algebraic-free-decomposition}.
\end{defn}

The proof of Theorem \ref{thm:Word-Measure-Character-Bound} consists
of $\left(i\right)$ the following theorem which is an analogue of
Theorem \ref{thm:Phi-approximation}, $\left(ii\right)$ showing that
$\chiak\left(w\right)=\chimax\left(\eta_{1}^{w}\right)=1-\pi\left(w\right)$
-- this is done in Section \ref{subsec:chi max ak is 1-pi}, and
$\left(iii\right)$ showing that $\left|\critak\left(w\right)\right|=\left\langle \qak,\xi_{1}-1\right\rangle \cdot\left|\crit\left(w\right)\right|$,
which is done in Section \ref{subsec:crit ak =00003D crit times C}.
\begin{thm}
\label{thm:character-first-bound}Let $1\ne w\in\F$ be a non-power.
Then 
\[
\mathbb{E}_{w}\left[\qak\right]=\uexp\left[\qak\right]+\left|\critak\left(w\right)\right|\cdot N^{\chiak(w)}+O\left(N^{\chiak(w)-1}\right).
\]

\end{thm}

\begin{proof}
Recall that $\eta_{\ak}^{w}\colon\Gamma_{\ak}^{w}\to X_{B}$, and
that our goal is to prove the stated approximation of $\Phi_{\eta_{\ak}^{w}}\left(N\right)=\mathbb{E}_{w}\left[\qak\right]$.
The notation $\Gamma_{\ak}^{w}$ specializes to $\Gamma_{1}^{w}$
being the core graph $\Gamma_{B}(\left\langle w\right\rangle ^{\F})$
which is a cycle representing $\left\langle w\right\rangle ^{\F}$.
Let $\Gamma_{\ak}^{w}\stackrel{\varphi}{\to}\Sigma\stackrel{\psi}{\to}X_{B}$
be the unique decomposition of $\eta_{\ak}^{w}$ to an algebraic $\varphi$
and a free $\psi$ from Theorem \ref{thm:algebraic-free-decomposition}.
Because $\eta_{\ak}^{w}$ has a decomposition $\left(\omega_{1},\omega_{2}\right)$
with $\omega_{1}\colon\Gamma_{\ak}^{w}\to\Gamma_{1}^{w}$ algebraic
(by Lemma \ref{lem:Phi-rank-one}), and in particular with $\mathrm{Im}\left(\omega_{1}\right)$
connected, $\Sigma$ must be connected as well (by Theorem \ref{thm:algebraic-free-decomposition}\eqref{enu:alg-free decomposition is universal for alg morphisms}).
Let $J$ be the group in the conjugacy class $\piol\left(\Sigma\right)$
to which $\left\langle w\right\rangle $ is mapped as a subgroup.
We have again that $w$ is a non-power in $J$, and as in the proof
of Theorem \ref{thm:Phi-approximation}, the quantities $\chiak\left(w\right)$
and $\critak\left(w\right)$ are the same in $J$ as in $\F$. Thus
we may assume without loss of generality that $\Sigma=X_{B}$ and
$J=\F$, namely, that $\eta_{\ak}^{w}$ is algebraic.

Using the algebraic Möbius Inversions from Section \ref{subsec:Algebraic-M=0000F6bius-Inversion}
we have 
\begin{equation}
\Phi_{\eta_{\ak}^{w}}(N)=\sum_{\left(\beta_{1},\beta_{2},\beta_{3}\right)\in\algdecompt\left(\eta_{\ak}^{w}\right)}C_{\beta_{2}}^{alg}(N).\label{eq:Phi of eta ak as sum of C-alg}
\end{equation}
By definition of $\chiak\left(w\right)$, any summand in \eqref{eq:Phi of eta ak as sum of C-alg}
satisfies $\chi\left(\mathrm{Im}\left(\beta_{2}\right)\right)=0$
or $\chi\left(\mathrm{Im}\left(\beta_{2}\right)\right)\le\chiak\left(w\right)$.
Consider the following cases:

\paragraph*{Case I: $\chi\left(\mathrm{Im}\left(\beta_{2}\right)\right)=0$}

~~~~As $w$ is a non-power, the only cyclic subgroups of $\F$
containing $w^{m}$ are $\left\langle w^{d}\right\rangle $ with $d\mid m$.
This shows that $\left(\beta_{1},\beta_{2}\right)$ is part of a decomposition
of $\omega_{1}$, and everything takes place inside the ambient group
$\left\langle w\right\rangle \cong\mathbb{Z}$. Conversely, every
decomposition $\left(\beta_{1},\beta_{2},\beta_{3}\right)\in\algdecompt\left(\omega_{1}\right)$
satisfies $\chi\left(\mathrm{Im}\left(\beta_{2}\right)\right)=0$.
The entire category of algebraic morphisms inside $\mocc\left(\left\langle w\right\rangle \right)$
are identical to that inside $\mocc\left(\mathbb{Z}\right)$. And
so 
\[
\sum_{\left(\beta_{1},\beta_{2},\beta_{3}\right)\in\algdecompt\left(\eta_{\ak}^{w}\right):~\chi\left(\mathrm{Im}\left(\beta_{2}\right)\right)=0}C_{\beta_{2}}^{alg}(N)=\Phi_{\omega_{1}}\left(N\right)=\mathbb{E}_{x}\left[\qak\right]=\mathbb{E}_{\mathrm{unif}}\left[\qak\right].
\]

\paragraph*{Case II: $\chi\left(\mathrm{Im}\left(\beta_{2}\right)\right)=\protect\chiak\left(w\right)$
and $\beta_{2}$ is an isomorphism}

~~~~By Corollary \ref{cor:C-algebraic-upper-bound}, we have in
this case $C_{\beta_{2}}^{alg}(N)=N^{\chi\left(\mathrm{Im}\left(\beta_{1}\right)\right)}=N^{\chi\left(\mathrm{Im}\left(\beta_{2}\right)\right)}=N^{\chiak\left(w\right)}$.
There is exactly one such decomposition in $\algdecompt\left(\eta_{\ak}^{w}\right)$
for every $\beta_{1}\in\critak\left(w\right)$, and so the total contribution
of these summands in \eqref{eq:Phi of eta ak as sum of C-alg} is
$\left|\critak\left(w\right)\right|\cdot N^{\chiak\left(w\right)}$.

\paragraph*{All remaining terms in \eqref{eq:Phi of eta ak as sum of C-alg}:}

In every other case, either $\chi\left(\mathrm{Im}\left(\beta_{2}\right)\right)<\chiak\left(w\right)$
or $\chi\left(\mathrm{Im}\left(\beta_{2}\right)\right)=\chiak\left(w\right)$
but $\beta_{2}$ is not an isomorphism, and Corollary \ref{cor:C-algebraic-upper-bound}
yields that $C_{\beta_{2}}^{alg}(N)=O\left(N^{\chiak(w)-1}\right)$.
This completes the proof of the theorem.

\end{proof}
\begin{rem}
\label{rem:expectaion in uniform measure of qak}The analysis above
readily leads to a precise formula for $\mathbb{E}_{\mathrm{unif}}\left[\qak\right]=\mathbb{E}_{x}\left[\qak\right]$.
Denote by $\eta_{\qak}^{x}$ the morphism corresponding to the single-letter
word $x$. By Lemma \ref{lem:Phi-rank-one}, 
\[
\mathbb{E}_{x}\left[\qak\right]=\Phi_{\eta_{\ak}^{x}}\left(N\right)=\left|\algdecomp\left(\eta_{\ak}^{x}\right)\right|.
\]
Every such decomposition induces a partition on the components of
$\Gamma_{\ak}^{w}$ (by their image in the intermediate multi core
graph). If a block in the partition consists of cycles corresponding
to $\left\langle x^{k_{1}}\right\rangle ,\ldots,\left\langle x^{k_{m}}\right\rangle $,
the connected image corresponds to $\left\langle x^{d}\right\rangle $
for some $d\mid\gcd\left(k_{1},\ldots,k_{m}\right)$, and there are
$d^{m-1}$ non-equivalent morphisms to such a cycle. So if $S$ is
a multiset with $\alpha_{1}$ $1$'s, $\alpha_{2}$ $2$'s and so
on, then

\[
\mathbb{E}_{\mathrm{unif}}\left[\qak\right]=\sum_{{\cal P}\in\mathrm{Partitions}(S)}\left[\prod_{A\in P}\left[\sum_{d|gcd\left(\left\{ k\in A\right\} \right)}d^{|A|-1}\right]\right].
\]
\end{rem}

\subsection{Maximal Euler characteristic\label{subsec:chi max ak is 1-pi}}
\begin{lem}
\label{lem:distance-one-algebraic-morphisms}Let $J\leq\F$ be a f.g.~subgroup
and let $u\in\F$. If $\rk\left(\left\langle J,u\right\rangle \right)\leq\rk J$,
then $J\leq_{alg}\left\langle J,u\right\rangle $. 
\end{lem}

\begin{proof}
This is \cite[Corollary 3.14]{miasnikov2007algebraic}, but we give
the short proof here for completeness. Assume by contradiction that
$J\leq L\overset{\ast}{\lneq}\langle J,u\rangle$. Then $\langle L,u\rangle=\langle J,u\rangle$
and $\rk L+1\leq\rk\langle J,u\rangle$ and hence $L\ast\langle u\rangle=\langle J,u\rangle$.
Since $J$ is a subgroup of $L$, and this is a free product, it follows
that $J=L$. Therefore, $J\ast\langle u\rangle=\langle J,u\rangle$,
which contradicts the rank inequality.
\end{proof}
\begin{prop}
\label{prop:joint-primitivity-rank}Let $w\in\F$ be a non-power as
above. Then for every $k\ge1$ and $\ak\ge0$ not all zeros, 
\[
\chiak(w)=\chimax(\eta_{1}^{w})=1-\pi(w).
\]
\end{prop}

\begin{proof}
We start by proving that $\chiak(w)\geq\chimax(\eta_{1}^{w})$. Because
$w$ is not a power, $\pi\left(w\right)\ge2$, namely, $\chimax\left(\eta_{1}^{w}\right)<0$.
Let $\beta\colon\Gamma_{1}^{w}\to\Sigma$ be a critical morphism of
$\eta_{1}^{w}$ (as in Definition \ref{eq:def of chi-max}), so $\chi\left(\Sigma\right)=\chimax\left(\eta_{1}^{w}\right)<0$.
By Lemma \ref{lem:Phi-rank-one}, the natural morphism $\Gamma_{\ak}^{w}\stackrel{\omega_{1}}{\to}\Gamma_{1}^{w}$
is algebraic. Therefore, the composition $\Gamma_{\ak}^{w}\stackrel{\omega_{1}}{\to}\Gamma_{1}^{w}\stackrel{\beta}{\to}\Sigma$
is also algebraic, and so $\chiak\left(w\right)\ge\chi\left(\Sigma\right)=\chimax\left(\eta_{1}^{w}\right)$.

To prove the converse inequality, let $\Gamma_{\ak}^{w}\stackrel{_{\varphi}}{\longrightarrow}\Sigma$
be algebraic with $\chi(\Sigma)<0$. We need to prove that $\chi\left(\Sigma\right)\le\chimax\left(\eta_{1}^{w}\right)$.
Let us restrict our attention to a component $\Sigma_{o}$ of $\Sigma$
with negative Euler characteristic. This gives an algebraic extension
of the corresponding powers of $w$, i.e., the components of $\Gamma_{\ak}^{w}$
mapped to $\Sigma_{o}$, and it is enough to show that $\chi\left(\Sigma_{o}\right)\le\chimax\left(\eta_{1}^{w}\right)$.
So we assume without loss of generality that $\Sigma$ is connected.

Assume that $\piol\left(\Sigma\right)=M^{\F}$ for some f.g.~$M\le\F$.
Assume that $\Gamma_{\ak}^{w}$ has $s=\sum\alpha_{i}$ connected
components corresponding to $\langle w^{k_{1}}\rangle^{\F},\ldots,\langle w^{k_{s}}\rangle^{\F}$,
and that the morphism $\varphi$ maps $\left\langle w^{k_{i}}\right\rangle \to u_{i}Mu_{i}^{-1}$
for some $u_{i}\in\F$ with $u_{1}=1$ (conjugate $M$ if needed).
Showing that $\chi\left(\Sigma\right)\le\chimax\left(\eta_{1}^{w}\right)$
is equivalent to that $\rk M\ge\pi\left(w\right)$.

Consider the subgroups $J_{1}\le J_{2}\le\ldots\le J_{s+1}$ defined
by gradually adding $u_{2},\ldots,u_{s},w$ to $M$:
\[
J_{1}\defi\left\langle M\right\rangle ,~~~~J_{2}\defi\left\langle J_{1},u_{2}\right\rangle ,~~~~\ldots,~~~~J_{s}=\left\langle J_{s-1},u_{s}\right\rangle ,~~~~J_{s+1}=\left\langle J_{s},w\right\rangle .
\]
Note that the extensions $J_{i}\le J_{i+1}$ are not free. Indeed,
for $i=1,\ldots,s-1$, $u_{i+1}Mu_{i+1}^{~-1}$ and $M$ both contain
$w^{\mathrm{lcm}\left(k_{1},k_{i+1}\right)}$, so $J_{i+1}=\left\langle J_{i},u_{i+1}\right\rangle $
is not a free extension of $J_{i}$. For $i=s$, $J_{s+1}=\left\langle J_{s},w\right\rangle $,
but $w$ has powers contained in $J_{s}$, so once again this is not
a free extension. Hence $\rk J_{i+1}\le\rk J_{i}$ for $i=1,\ldots,s$
and so $\rk J_{s+1}\le\rk M$. By Lemma \ref{lem:distance-one-algebraic-morphisms},
these are all algebraic extensions, and by transitivity of algebraic
extensions, $M\alg J_{s+1}$, corresponding to an algebraic morphism
$\Sigma\stackrel{\psi}{\to}\Delta\defi\Gamma_{B}\left(J_{s+1}^{~~\F}\right)$.
Note that the composition $\Gamma_{\ak}^{w}\stackrel{\varphi}{\to}\Sigma\stackrel{\psi}{\to}\Delta$
maps the subgroups $\left\langle w^{k_{1}}\right\rangle ,\ldots,\left\langle w^{k_{s}}\right\rangle $
to $J_{s+1}$ itself (the latter contains, in particular, $w$). Hence
this composition factors through $\Gamma_{1}^{w}$, as in the following
diagram:
\[
\xymatrix{\Gamma_{\ak}^{w}\ar[rr]_{\mathrm{alg}}^{\varphi}\ar@{->>}[d]_{\omega_{1}}^{\mathrm{alg}} &  & \Sigma\ar[d]_{\mathrm{alg}}^{\psi}\\
\Gamma_{1}^{w}\ar[rr]_{\alpha} &  & \Delta
}
\]
Because $\psi\circ\varphi$ is algebraic, so is $\alpha$. As $M$
was not cyclic, $J_{s+1}$ is not cyclic, so $\chi\left(\Delta\right)<0$.
We deduce that $\left\langle w\right\rangle \lneq_{alg}J_{s+1}$.
Hence $\pi\left(w\right)\le\rk J_{s+1}\le\rk M$.
\end{proof}

\subsection{The set of critical morphisms\label{subsec:crit ak =00003D crit times C}}

The previous results already show that for a non-power $w$,
\begin{equation}
\mathbb{E}_{w}\left[\qak\right]=\left\langle \qak,1\right\rangle +\left|\critak\left(w\right)\right|\cdot N^{1-\pi\left(w\right)}+O\left(N^{-\pi\left(w\right)}\right).\label{eq:almost main thm}
\end{equation}
In order to prove Theorem \ref{thm:Word-Measure-Character-Bound},
it remains to show that all morphisms in $\critak\left(w\right)$
are obtained from $\crit\left(w\right)$, or, equivalently, from $\crit\left(\eta_{1}^{w}\right)$,
in the following straightforward way. If $\varphi\colon\Gamma_{\ak}^{w}\to\Sigma$
is a critical morphism in $\critak\left(w\right)$, then $\Sigma$
has one component $\Sigma_{o}$ with $\chi\left(\Sigma_{o}\right)=\chiak\left(w\right)=\chimax\left(\eta_{1}^{w}\right)=1-\pi\left(w\right)$,
and the remaining components are cycles corresponding to powers of
$w$. Let $\varphi_{o}\colon\Gamma_{o}\to\Sigma_{o}$ be the restriction
of $\varphi$ to the disjoint union $\Gamma_{o}$ of components of
$\Gamma_{\ak}^{w}$ mapped to $\Sigma_{o}$. Proposition \ref{prop:critical-tuple}
below shows that $\varphi_{o}$ can always be factored through a unique
$\beta\colon\Gamma_{1}^{w}\to\Sigma_{o}$ in $\crit\left(\eta_{1}^{w}\right)$. 

On the other hand, it is clear that the number of $\varphi\in\critak\left(w\right)$
corresponding to a given $\beta\in\crit\left(\eta_{1}^{w}\right)$
as above, depends only on $\ak$ and not on $w$ nor on $\beta$.
With notation as in Remark \ref{rem:expectaion in uniform measure of qak},
this number is equal to 
\[
\sum_{{\cal P}\in\mathrm{Partitions}(S)}\left[\sum_{A\in{\cal P}}\left[\prod_{B\in P\setminus\left\{ A\right\} }\left[\sum_{d|gcd\left(\left\{ k_{i}\in B\right\} \right)}d^{|B|-1}\right]\right]\right],
\]
and in Proposition \ref{prop:critical morphisms coming from Crit(w)}
below, we prove it is equal to $\left\langle \qak,\xi_{1}-1\right\rangle $.

Proposition \ref{prop:critical-tuple} crucially relies on the following
theorem of Louder, as stated in \cite[page 553]{louder2018negative}.
\begin{thm}
\label{thm:Louder}\cite{louder2013scott} Consider the following
graph of groups $\Delta$ in the shape of a star, with $\ell$ vertices
around a central vertex, and $m_{i}\ge1$ edges between the center
and the $i$-th peripheral vertex -- see Figure \ref{fig:louder}.
The center vertex-group is $\mathbb{Z}$ with generator denoted $w$.
The peripheral vertex-groups are free groups $H_{1},\ldots,H_{\ell}$
and all edge-groups are infinite cyclic. For every $i\in\left[\ell\right],j\in\left[m_{i}\right]$,
there is an element $v_{i,j}\in H_{i}$ and a positive integer $n_{i,j}$
so that the $j$-th edge between $H_{i}$ and $\left\langle w\right\rangle $
attaches $v_{i,j}$ to $w^{n_{i,j}}$. We assume further that
\begin{enumerate}
\item For all $i,j$, $v_{i,j}\ne1$ and is not a proper power.
\item For all $i$, $\left\langle v_{i,1}\right\rangle ^{H_{\text{i}}},\ldots,\left\langle v_{i,m_{i}}\right\rangle ^{H_{i}}$
are distinct conjugacy classes.
\item There exists no free splitting of any of the groups $H_{i}$, $H_{i}=H_{i}'\ast\langle v_{i,k}\rangle$,
such that all the remaining elements $v_{i,j},j\neq k$ are conjugate
into $H_{i}'$.
\item $\sum_{i,j}n_{i,j}\ge2$.
\end{enumerate}
Let $\pi_{1}\left(\Delta\right)$ denote the corresponding group.
Namely, 
\[
\pi_{1}\left(\Delta\right)=\left\langle H_{1},\ldots,H_{\ell},w,\left\{ t_{i,j}\right\} _{i\in\left[\ell\right],j\in\left[m_{i}\right]}\,\middle|\,t_{i,j}v_{i,j}t_{i,j}^{-1}=w^{n_{i,j}},\left\{ t_{i,1}=1\right\} _{i\in\left[\ell\right]}\right\rangle .
\]
Then, if $f\colon\pi_{1}\left(\Delta\right)\twoheadrightarrow J$
is a surjective homomorphism onto a free group $J$, and $f\Big|_{H_{i}}$
is injective for every $i$, then
\[
\rk J-1<\sum_{i=1}^{\ell}\left(\rk H_{i}-1\right).
\]
\end{thm}

\begin{figure}
\includegraphics[viewport=-359.181bp 0bp 193.9578bp 155.7198bp,clip,scale=0.5]{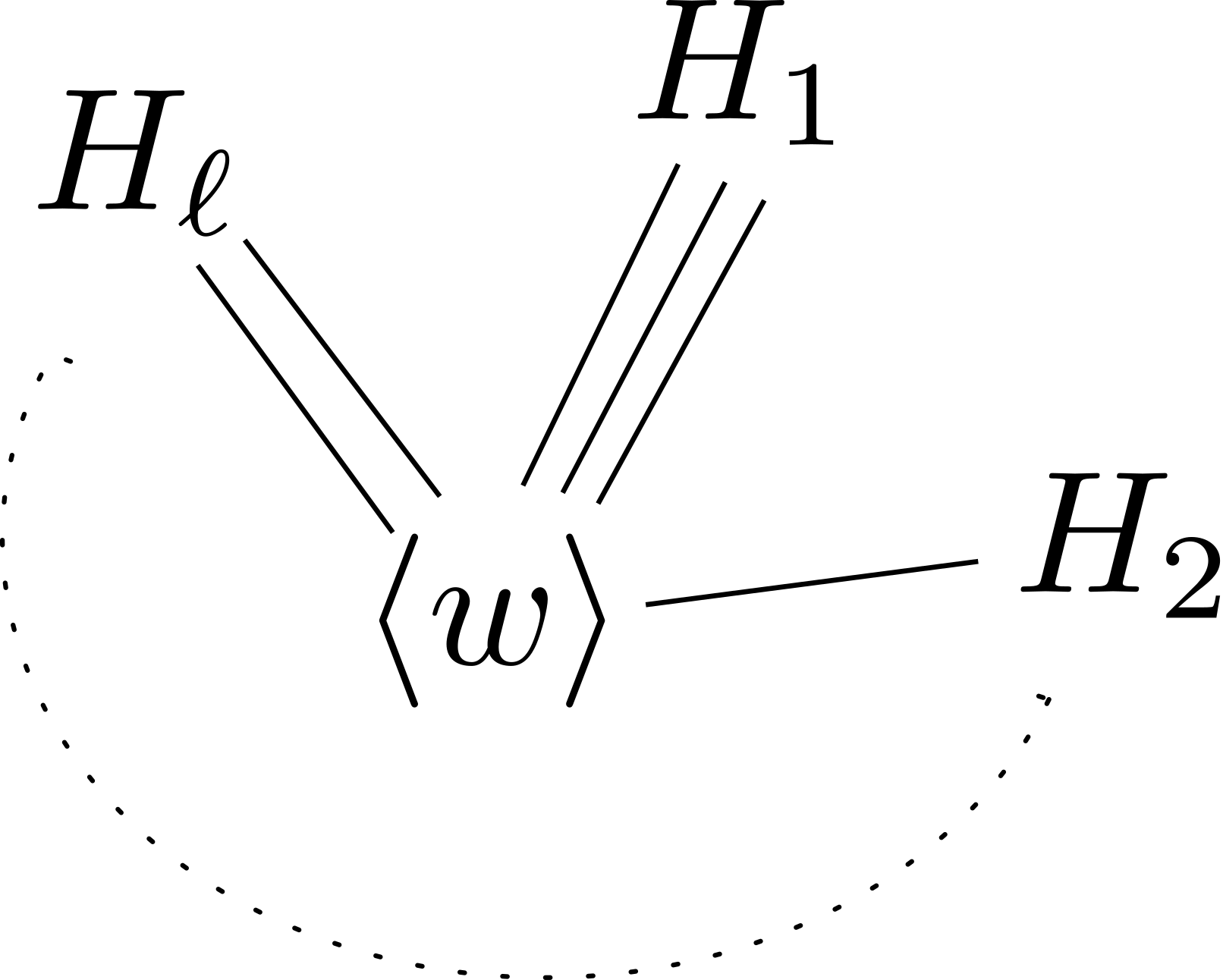}\caption{\label{fig:louder}The graph of groups in Louder's Theorem \ref{thm:Louder}.}
\end{figure}

\begin{prop}
\label{prop:critical-tuple}Let $1\ne w\in\F$ be a non-power. Assume,
as above, that $\varphi_{o}\colon\Gamma_{\ak}^{w}\to\Sigma_{o}$ is
a critical morphism in $\critak\left(w\right)$ with $\Sigma_{o}$
connected and $\chi\left(\Sigma_{o}\right)<0$ (and so $\chi\left(\Sigma_{o}\right)=\chiak\left(w\right)=\chimax\left(\eta_{1}^{w}\right)=1-\pi\left(w\right)$).
Then $\varphi_{o}$ factors through a unique $\beta\colon\Gamma_{1}^{w}\to\Sigma_{o}$
in $\crit\left(\eta_{1}^{w}\right)$.
\end{prop}

\begin{proof}
Because $w$ is not a power, there is a unique morphism $\omega_{1}\colon\Gamma_{\ak}^{w}\to\Gamma_{1}^{w}$,
and therefore at most one decomposition of $\varphi_{o}$ through
some $\beta\colon\Gamma_{1}^{w}\to\Sigma_{o}$ in $\crit\left(\eta_{1}^{w}\right)$.
It remains to show such a factorization exists.

We freely use notation from the proof of Proposition \ref{prop:joint-primitivity-rank}.
Recall the notion of pullback in the category $\mocc\left(\F\right)$
from page \pageref{pullback}. Let $\left(\Lambda,\sigma_{1},\sigma_{2}\right)$
be the pullback of $\psi_{1}\colon\Sigma_{o}\to X_{B}$ and of $\psi_{2}\colon\Gamma_{1}^{w}\to X_{B}$,
and let $\beta\colon\Gamma_{\ak}^{w}\to\Lambda$ be the unique morphism
making the following diagram commute.
\[
\xymatrix{\Gamma_{\ak}^{w}\ar@{.>}[rd]^{\beta}\ar@/^{2pc}/[rrd]^{\varphi_{0}}\ar@/_{2pc}/[rdd]_{\omega_{1}}\\
 & \Lambda\ar[r]^{\sigma_{1}}\ar[d]_{\sigma_{2}} & \Sigma_{o}\ar[d]^{\psi_{1}}\\
 & \Gamma_{1}^{w}\ar[r]_{\psi_{2}} & X_{B}
}
\]
As every pullback involving $\Gamma_{1}^{w}$, the multi core graph
$\Lambda$ is a union of cycles corresponding to powers of $w$. Our
goal is to show that the image of $\beta$ meets a sole component
of $\Lambda$ which is isomorphic to $\Gamma_{1}^{w}$, and so $\sigma_{1}$,
restricted to this component, is in $\crit\left(\eta_{1}^{w}\right)$.
Assume to the contrary that this is not the case, so the image of
$\beta$ consists of some $\Gamma_{\alpha_{1}',\ldots,\alpha_{k}'}$
with $\sum i\alpha'_{i}\ge2$. Without loss of generality, assume
the image is $\Gamma_{\ak}^{w}$ itself and $\sum i\alpha_{i}\ge2$,
and so $\sigma_{1}=\varphi_{o}$. In particular, we assume that $\left(\Gamma_{\ak}^{w},\varphi_{o},\omega_{1}\right)$
is itself the pullback.

Let $s=\sum\alpha_{i}$ and $k_{1},\ldots,k_{s}$ be the powers of
$w$ in the components of $\Gamma_{\ak}^{w}$. Let $M\le\F$ be a
f.g.~subgroup with $M^{\F}=\piol\left(\Sigma_{o}\right)$, such that
$\varphi_{o}$ is given by $w^{k_{1}},u_{2}w^{k_{2}}u_{2}^{-1},\ldots,u_{s}w^{k_{s}}u_{s}^{-1}\in M$,
for some $u_{2},\ldots,u_{s}\in\F$. In the notation of Louder's Theorem
\ref{thm:Louder} with $\ell=1$, consider the graph of groups $\Delta$
with two vertices: $H_{1}=M$ and $\left\langle w\right\rangle $,
and $m_{1}=s$ parallel edges with $\mathbb{Z}$ as edge-group between
them. We denote $v_{1,i}$ by $v_{i}$ and $n_{1,i}$ by $n_{i}$,
and set $v_{1}=w^{k_{1}},v_{2}=u_{2}w^{k_{2}}u_{2}^{-1},\ldots,v_{s}=u_{s}w^{k_{s}}u_{s}^{-1}$,
and $n_{i}=k_{i}$ for $i\in\left[s\right]$. We claim that these
choices satisfy the four assumptions in Theorem \ref{thm:Louder}.
Indeed:
\begin{itemize}
\item If some $v_{i}=t^{d}$ was a proper power in $M$ (so $d\ge2$), then
this equation is also valid in $\F$, so that $d\mid k_{i}$. But
then both $\omega_{1}$ and $\varphi_{o}$ factor through the morphism
$\Gamma_{\ak}^{w}\to\Delta$ which maps $\langle w^{k_{i}}\rangle$
to $\langle w^{k_{i}/d}\rangle$ (and leaves all other elements unchanged),
in contradiction to $\left(\Gamma_{\ak}^{w},\varphi_{o},\omega_{1}\right)$
being the pullback.
\item Assume that $\left\langle v_{i}\right\rangle $ and $\left\langle v_{j}\right\rangle $
are conjugate subgroups of $M$ for some $i\ne j$, say without loss
of generality that $v_{i}=mv_{j}m^{-1}$ (the other possibility being
$v_{i}=mv_{j}^{-1}m^{-1}$) with $m\in M$. Then both $\omega_{1}$
and $\varphi_{o}$ factor through the morphism $\Gamma_{\ak}^{w}\to\Delta$
which maps $\langle w^{k_{j}}\rangle$ isomorphically to $\langle w^{k_{j}}\rangle$,
and $\langle w^{k_{i}}\rangle$ to the same component, as 
\[
\left\langle w^{k_{i}}\right\rangle =\left\langle u_{i}^{-1}v_{i}u_{i}\right\rangle =\left\langle u_{i}^{-1}mv_{j}m^{-1}u_{i}\right\rangle =u_{i}^{-1}m\left\langle v_{j}\right\rangle m^{-1}u_{i}.
\]
This, again, contradicts our assumption that $\left(\Gamma_{\ak}^{w},\varphi_{o},\omega_{1}\right)$
is itself the pullback.
\item Assume there is a free splitting $M=M'*\left\langle v_{i}\right\rangle $
with all other $v_{j}$'s conjugate into $M'$. But then $\varphi_{o}$
factors through the free factor $\{M'^{\F},\langle v_{i}\rangle^{\F}\}\to\{M^{\F}\}$,
in contradiction to $\varphi_{o}$ being algebraic.
\item Finally, our assumption that $\sum i\alpha_{i}\ge2$ is equivalent
to $\sum n_{i}\ge2$.
\end{itemize}
Let $\pi_{1}\left(\Delta\right)$ be the fundamental group of this
graph of groups, namely,
\[
\pi_{1}\left(\Delta\right)=\left\langle M,w,t_{1},\ldots,t_{s}\,\middle|\,t_{i}v_{i}t_{i}^{-1}=w^{k_{i}},t_{1}=1\right\rangle .
\]
Consider, as in the proof of Proposition \ref{prop:joint-primitivity-rank},
the extension $J_{s+1}=\left\langle M,u_{2},\ldots,u_{s},w\right\rangle $
of $M$ inside $\F$. The same argument as in that proof applies to
show that $J_{s+1}$ is an algebraic extension of $M$. In particular,
the composition of $\varphi_{o}$ with this algebraic extension $M\alg J_{s+1}$
gives an algebraic morphism from $\Gamma_{\ak}^{w}$ to $\Gamma_{B}(J_{s+1}^{~~~\F})$,
which also factors through $\omega_{1}\colon\Gamma_{\ak}^{w}\to\Gamma_{1}^{w}$.
Thus $J_{s+1}$ is a proper algebraic extension of $w$, and hence
$\rk J_{s+1}\ge\pi\left(w\right)$.

On the other hand, $J_{s+1}$ is a free quotient of $\pi_{1}\left(\Delta\right)$
which extends an embedding of $M$: this is obtained by mapping $w$
to itself, $t_{1}\mapsto1$ and $t_{i}\mapsto u_{i}^{-1}$ for $i\ge2$.
By Louder's Theorem \ref{thm:Louder}, we have $\rk J_{s+1}<\rk M$,
and so $\rk M>\rk J_{s+1}\ge\pi\left(w\right)$. This contradicts
our assumption that $\rk M=\pi\left(w\right)$.
\end{proof}
\begin{prop}
\label{prop:critical morphisms coming from Crit(w)}Let $1\ne w\in\F$
be a non-power. For every $H\in\crit\left(w\right)$, there are\linebreak{}
$\left\langle \qak,\xi_{1}-1\right\rangle $ distinct critical morphisms
in $\critak\left(w\right)$ mapping a non-empty subset of the powers
of $w$ to $w$ and then to $H$, and the remaining powers of $w$
to cyclic subgroups of $\left\langle w\right\rangle $.
\end{prop}

As we already know from Proposition \ref{prop:critical-tuple} that
every morphism in $\critak\left(w\right)$ corresponds to exactly
one $H\in\crit\left(w\right)$, Proposition \ref{prop:critical morphisms coming from Crit(w)}
yields that $\left|\critak\left(w\right)\right|=\left\langle \qak,\xi_{1}-1\right\rangle \cdot\left|\crit\left(w\right)\right|$.
\begin{proof}
One can give a direct argument, but we like the following one better.
In the notation of Section \ref{sec:Introduction}, let $\kappa\left(f\right)$
denote the constant corresponding to the class function $f\in\A$
in the equality 
\[
\mathbb{E}_{w}\left[f\right]=\left\langle f,1\right\rangle +\kappa\left(f\right)\cdot\frac{\left|\crit\left(w\right)\right|}{N^{\pi\left(w\right)-1}}+O\left(\frac{1}{N^{\pi(w)}}\right).
\]
We know such equality holds with some $\kappa\left(f\right)\in\mathbb{Q}$
because we already know by \eqref{eq:almost main thm} that such constants
exist for every $f=\qak$.

By a theorem of Frobenius \cite{frobenius1896gruppencharaktere},
already mentioned on page \pageref{simple commutator}, for any irreducible
character $\chi$ of any finite group $G$, $\mathbb{E}_{\left[x,y\right]}\left[\chi\right]=\frac{1}{\dim\chi}$.
By Proposition \ref{prop:irreps as linear basis of A}, every class
function $f\in\A$ is of the form 
\[
f=\sum_{\chi\in\symirr}\langle f,\chi\rangle\chi
\]
with finitely many non-vanishing terms. Reverse-engineering Theorem
\ref{thm:PP15} for $w=\left[x,y\right]$ gives $\pi\left(\left[x,y\right]\right)=2$
and $|\crit\left(\left[x,y\right]\right)|=1$. Hence,
\[
\sum_{\chi\in\symirr}\frac{\langle f,\chi\rangle}{\dim\chi}=\mathbb{E}_{\left[x,y\right]}\left[f\right]=\langle f,1\rangle+\frac{\kappa\left(f\right)}{N}+O\left(\frac{1}{N^{2}}\right).
\]
As explained on page \pageref{family of irreps has degree poly in N},
for every $\chi\in\symirr$, the dimension $\dim\chi$ is a polynomial
function of $N$, which has degree $\ge2$ if if $\chi\neq1,\z_{1}-1$.
Subtracting from the left hand side the summands corresponding to
$\chi=1$ and to $\chi=\xi_{1}-1$ leaves $O\left(N^{-2}\right)$.
Thus

\[
O\left(\frac{1}{N^{2}}\right)=\frac{\kappa\left(f\right)}{N}-\frac{\left\langle f,\z_{1}-1\right\rangle }{N-1}=\frac{\kappa\left(f\right)-\left\langle f,\z_{1}-1\right\rangle }{N}+O\left(\frac{1}{N^{2}}\right).
\]
Thus $\kappa\left(f\right)=\langle f,\z_{1}-1\rangle$.
\end{proof}
This completes the proof of our main result, Theorem \ref{thm:Word-Measure-Character-Bound}.

\subsubsection{Some remarks\label{subsec:Some-remarks}}
\begin{rem}
\label{rem:a1...ak -pure algebraic morphisms}Proposition \ref{prop:critical morphisms coming from Crit(w)}
teaches us that if $\sum i\alpha_{i}\ge2$, then every algebraic morphism
from $\Gamma_{\ak}^{w}$ with codomain of EC (Euler characteristic)
$1-\pi\left(w\right)$, actually originates from an algebraic morphism
from $\Gamma_{1}^{w}$. Similarly, there may be algebraic morphisms
from $\Gamma_{\ak}^{w}$ originating from algebraic morphisms from
a different $\Gamma_{\beta_{1},\ldots,\beta_{\ell}}^{w}$ if there
is a morphism from a subset of the connected components of $\Gamma_{\ak}^{w}$
to $\Gamma_{\beta_{1},\ldots,\beta_{\ell}}^{w}$. Note that such $\beta_{1},\ldots,\beta_{\ell}$
satisfy $\sum j\beta_{j}<\sum i\alpha_{i}$. It is natural to consider
algebraic morphisms from $\Gamma_{\ak}^{w}$ which \emph{cannot} be
constructed in this way. Namely, these are algebraic morphisms not
having any connected component of EC 0 in their codomain and which
\emph{do not} factor through a (non-identity) morphism from $\Gamma_{\ak}^{w}$
with codomain of EC $0$. For the sake of the current Section \ref{subsec:Some-remarks},
call such algebraic morphisms $\left(w;\ak\right)$-pure. 

Let $\eta\colon\Gamma_{\ak}^{w}\to\Sigma$ be an algebraic morphism.
Every connected component $\Sigma_{o}$ of $\Sigma$ with $\eta|_{\eta^{-1}\left(\Sigma\right)}$
not an isomorphism, has the property that every edge of $\Sigma_{o}$
is covered by at least two edges from $\Gamma_{\ak}^{w}$ --- this
is an immediate generalization of \cite[Lemma 4.1]{Puder2015}. This
shows that $\left(w;\ak\right)$-pure morphisms are \emph{reducible},
in the sense of \cite{louder2017stackings}. Thus, \cite[Theorem 1.2]{louder2017stackings}
says that every $\left(w;\ak\right)$-pure morphism has codomain of
EC at most $-\sum i\alpha_{i}$. In light of Conjecture \ref{conj:irreducible-characters}
and its connection to algebraic morphisms as in \eqref{eq:Phi of eta ak as sum of C-alg},
it is plausible to conjecture that a tight bound here should be $\left(1-\pi\left(w\right)\right)\cdot\sum i\alpha_{i}$
(so \cite[Theorem 1.2]{louder2017stackings} yields this conjecture
when $\pi\left(w\right)=2$). In fact, this conjecture about $\left(w;\ak\right)$-pure
morphisms is precisely \cite[Conjecture 1.5]{louder2021uniform}.
\end{rem}

\begin{rem}
Louder's Theorem \ref{thm:Louder} is strengthened in \cite[Theorem 1.11]{louder2018negative}
to the fact that under the same assumptions (except for $\sum n_{i,j}\ge2$
which can be discarded), the following inequality holds:
\end{rem}

\[
\rk J-1\le\sum_{i}\left(\rk H_{i}-1\right)-\left(\left(\sum_{i,j}n_{i,j}\right)-1\right).
\]
This implies that algebraic morphisms $\Gamma_{\ak}^{w}\to\Sigma$
such that $\Sigma$ is connected with $\chi\left(\Sigma\right)<0$,
and where $\Gamma_{\ak}^{w}$ is itself the pullback of $\Sigma\to X_{B}$
and $\Gamma_{1}^{w}\to X_{B}$, satisfy that $\rk\Sigma\ge\pi\left(w\right)-1+\sum_{i}i\alpha_{i}$.
This may be relevant to strengthening Theorem \ref{thm:Word-Measure-Character-Bound}
towards Conjecture \ref{conj:irreducible-characters}.

\section{Expansion of random Schreier graphs: the proof of Theorem \ref{thm:Schreier-Graphs-Near-Ramanujan}\label{sec:Expanders}}

Fix $s\in\mathbb{Z}_{\ge1}$ and assume throughout that $N\ge s$.
Also, fix a basis $B$ of $\F=\F_{r}$. Let $\sigma_{1},\ldots,\sigma_{r}\in S_{N}$
be independent, uniformly random permutations, and let $G=G\left(\sigma_{1},\ldots,\sigma_{r}\right)$
be the $d=2r$-regular Schreier graph depicting the action of $S_{N}$
on $\left(\left[N\right]\right)_{s}$, the set of $s$-tuples of distinct
elements in $\left[N\right]$, with respect to $\sigma_{1},\ldots,\sigma_{r}$.
This is a graph with $\left(N\right)_{s}=N\left(N-1\right)\cdots\left(N-s+1\right)$
vertices. In this section we prove Theorem \ref{thm:Schreier-Graphs-Near-Ramanujan},
stating that the random graph $G$ is a.a.s.~an expander with a spectral
bound as given in \eqref{eq:schreier graphs bound}. Namely, the largest
absolute value of a non-trivial eigenvalue of $A_{G}$, the adjacency
matrix of $G$, satisfies a.a.s.~$\mu\left(G\right)\le2\sqrt{d-1}\cdot\exp\left(\frac{2s^{2}}{e^{2}\left(d-1\right)}\right)$.

We follow the strategy laid out in \cite{Puder2015} and its addendum
\cite{friedman2020nonbacktracking}. The strategy is based on the
trace method together with the results from \cite{PP15}. As in \cite{friedman2020nonbacktracking},
instead of analyzing directly the regular adjacency operator, we analyze
the non-backtracking spectrum and only at the end of the argument
deduce a bound on $\mu\left(G\right)$. 

Denote by $\E$ the set of oriented edges of $G$, namely, each edge
of $G$ appears twice in this set, once with every possible orientation,
so $\left|\E\right|=\left(N\right)_{s}\cdot d$. For $e\in\E$, we
denote by $\overline{e}$ the same edge with the reverse orientation,
and by $h\left(e\right)$ and $t\left(e\right)$ the head and tail
of $e$, respectively. The \emph{Hashimoto} or \emph{non-backtracking}
matrix $B=B_{G}$ is a $\left|\E\right|\times\left|\E\right|$ $0$-$1$
matrix with rows and columns indexed by the elements of $\E$. The
$e,f$ entry is defined by
\[
B_{e,f}=\begin{cases}
1 & \mathrm{if}~t\left(e\right)=h\left(f\right)~\mathrm{and}~f\ne\overline{e},\\
0 & \mathrm{otherwise.}
\end{cases}
\]

The Ihara-Bass formula gives a dictionary between the spectrum of
$B_{G}$ and that of the adjacency matrix $A_{G}$. Every eigenvalue
$\lambda\in\spec\left(A_{G}\right)$ with $\left|\lambda\right|\ge2\sqrt{d-1}$
gives rise to two \emph{real} eigenvalues of $B_{G}$ in $\left[-\left(d-1\right),-1\right]\cup\left[1,d-1\right]$,
while every eigenvalue with $\left|\lambda\right|<2\sqrt{d-1}$ corresponds
to two non-real eigenvalues lying on the circle of radius $\sqrt{d-1}$
around $0$ in $\mathbb{C}$. In both cases, the two eigenvalues of
$B_{G}$ are given by $\frac{\lambda\pm\sqrt{\lambda^{2}-4\left(d-1\right)}}{2}\in\spec\left(B_{G}\right)$.
In particular, the trivial eigenvalue $d\in\spec\left(A_{G}\right)$
corresponds to $1,d-1\in\spec\left(B_{G}\right)$, which are considered
to be trivial eigenvalues of $B_{G}$. In addition, there are $\left(d-2\right)\cdot\left(N\right)_{s}$
additional $\pm1$ eigenvalues. For more details see \cite[Section 2]{friedman2020nonbacktracking}
and the references therein. 

Order the eigenvalues of $B_{G}$ by their absolute value to get
\begin{equation}
d-1=\left|\nu_{1}\right|\ge\left|\nu_{2}\right|\ge\ldots\ge\left|\nu_{2\left(N\right)_{s}}\right|=\left|\nu_{2\left(N\right)_{s}+1}\right|=\ldots=\left|\nu_{d\left(N\right)_{s}}\right|=1.\label{eq:evalues of B}
\end{equation}
We let $\nu\left(G\right)\defi\left|\nu_{2}\right|$ denote the largest
absolute value of a non-trivial eigenvalue. If $\left(N\right)_{s}\ge2$
then $\nu\left(G\right)\in\left[\sqrt{d-1},d-1\right]$. Notice that
if $\nu\left(G\right)>\sqrt{d-1}$, in which case $\nu_{2}$ is real,
then
\begin{equation}
\mu\left(G\right)=\nu\left(G\right)+\frac{d-1}{\nu\left(G\right)}.\label{eq:nu vs. mu of G}
\end{equation}

Our immediate goal is to bound $\nu\left(G\right)$ from above. The
trace of the $t$-th power $B_{G}^{~t}$ of the Hashimoto matrix $B_{G}$
is equal to the number of \emph{cyclically non--backtracking closed
walks} of length $t$ in $G$. As every edge in $G$ is directed and
corresponds to one of $\sigma_{1},\ldots,\sigma_{r}$, a cyclically
non-backtracking closed oriented path of length $t$ corresponds to
a cyclically reduced word of length $t$ in $\left\{ \sigma_{1}^{\pm1},\ldots,\sigma_{r}^{\pm1}\right\} $.
In other words, every such path corresponds to $w\left(\sigma_{1},\ldots,\sigma_{r}\right)$
where $w\in\F=\F_{r}$ is cyclically reduced and of length $t$. Denote
the set of cyclically reduced words of length $t$ in $\F_{r}$ by
$\cyr_{t}\left(\F_{r}\right)$. The number of closed paths corresponding
to a given such $w$ is equal to the number of $s$-tuples in $\left(\left[N\right]\right)_{s}$
fixed by $w\left(\sigma_{1},\ldots,\sigma_{r}\right)\in S_{N}$. Denote
by $\chi_{s}$ the (reducible) character corresponding to this permutation-representation
of $S_{N}$. Then the number of closed paths in $G$ corresponding
to $w$ is $\chi_{s}\left(w\left(\sigma_{1},\ldots,\sigma_{r}\right)\right)$.
We have:
\begin{equation}
\sum_{i=1}^{d\cdot\left(N\right)_{s}}\nu_{i}^{~t}=\mathrm{tr}\left(B_{G}^{~t}\right)=\sum_{w\in\cyr_{t}\left(\F_{r}\right)}\chi_{s}\left(w\left(\sigma_{1},\ldots,\sigma_{r}\right)\right).\label{eq:trace method, no expectations}
\end{equation}
There is an exact formula for the number of such words:
\begin{prop}
\cite[Theorem 1.1]{rivin2010growth}\label{prop:Mann} The number
of cyclically reduced words of length $t$ in $\F_{r}$ is 
\[
\left|\cyr_{t}\left(\F_{r}\right)\right|=\left(2r-1\right)^{t}+r+\left(-1\right)^{t}\left(r-1\right).
\]
\end{prop}

(This is also Proposition 17.2 in \cite{mann2011groups}.) So if $t$
is even, $\left|\cyr_{t}\left(\F_{r}\right)\right|=\left(d-1\right)^{t}+\left(d-1\right)$.
In addition, if $t$ is even, for every real eigenvalue $\nu_{i}$,
the summand $\nu_{i}^{~t}$ is positive. Since every non-real eigenvalue
$\nu_{i}$ lies on $\left\{ z\in\mathbb{C}\colon\left|z\right|=\sqrt{d-1}\right\} $,
the summand $\nu_{i}^{~t}$ in this case has real part at least $-\sqrt{d-1}^{t}$.
Recall also that there is a trivial eigenvalue $\nu_{1}=d-1$ and
that (at least) $\left(N\right)_{s}\cdot\left(d-2\right)+1$ out of
the $\left(N\right)_{s}\cdot d$ eigenvalues are $\pm1$. Hence, for
$t$ even we have
\begin{eqnarray*}
\mathrm{tr}\left(B_{G}^{~t}\right) & = & \left(d-1\right)^{t}+\sum_{i=2}^{2\left(N\right)_{s}-1}\nu_{i}{}^{t}+\left(N\right)_{s}\cdot\left(d-2\right)+1.
\end{eqnarray*}
Taking real parts we have
\begin{eqnarray*}
\mathrm{Re}\left[\nu_{2}\left(\Gamma\right)^{t}\right] & = & \mathrm{tr}\left(B_{G}^{~t}\right)-\left(d-1\right)^{t}-\sum_{i=3}^{2\left(N\right)_{s}-1}\mathrm{Re}\left[\nu_{i}{}^{t}\right]-\left(N\right)_{s}\cdot\left(d-2\right)-1\\
 & \le & \left[\sum_{w\in\cyr_{t}\left(\F_{r}\right)}\chi_{s}\left(w\left(\sigma_{1},\ldots,\sigma_{r}\right)\right)\right]-\left(d-1\right)^{t}+2\left(N\right)_{s}\sqrt{d-1}^{t}-\left(N\right)_{s}\left(d-2\right)-1\\
 & \stackrel{\mathrm{Prop.~\ref{prop:Mann}}}{=} & \left[\sum_{w\in\cyr_{t}\left(\F_{k}\right)}\left[\chi_{s}\left(w\left(\sigma_{1},\ldots,\sigma_{r}\right)\right)-1\right]\right]+d-1+2\left(N\right)_{s}\sqrt{d-1}^{t}-\left(N\right)_{s}\left(d-2\right)-1\\
 & \stackrel{N\ge S}{\le} & \left[\sum_{w\in\cyr_{t}\left(\F_{k}\right)}\left[\chi_{s}\left(w\left(\sigma_{1},\ldots,\sigma_{r}\right)\right)-1\right]\right]+2\left(N\right)_{s}\sqrt{d-1}^{t}.
\end{eqnarray*}
Taking expectations we obtain
\begin{equation}
\mathbb{E}\left[\mathrm{Re}\left[\nu_{2}\left(\Gamma\right)^{t}\right]\right]\le\left[\sum_{w\in\cyr_{t}\left(\F_{k}\right)}\left(\mathbb{E}_{w}\left[\chi_{s}\right]-1\right)\right]+2\left(N\right)_{s}\sqrt{d-1}^{t}.\label{eq:step one}
\end{equation}

We can finally use our main results from the current paper. For $N\ge s$,
the action of $S_{N}$ on $\left(\left[N\right]\right)_{s}$ is transitive,
and so the expected number of fixed points is $\left\langle \chi_{s},1\right\rangle =1$.
Corollary \ref{cor:general sym function} therefore gives
\begin{equation}
\mathbb{E}_{w}\left[\chi_{s}\right]-1=\left\langle \chi_{s},\xi_{1}-1\right\rangle \cdot\frac{\left|\crit\left(w\right)\right|}{N^{\pi\left(w\right)-1}}+O\left(\frac{1}{N^{\pi\left(w\right)}}\right).\label{eq:first approx of E-1}
\end{equation}
To proceed, we estimate the number of words in $\cyr_{t}\left(\F_{r}\right)$
of a given primitivity rank, and then provide a bound of the big-$O$
term in \eqref{eq:first approx of E-1} in a uniform manner across
all words of a given length and a given primitivity rank. The first
of these tasks is given by \cite{Puder2015}:
\begin{thm}
\label{thm:primitivity-rank-bounds}\cite[Proposition 4.3 and Theorem 8.2]{Puder2015}
For every $r\geq2$ and $m\in\left\{ 1,\ldots,r\right\} $,

\begin{equation}
\limsup_{t\to\infty}\left[\sum_{w\in\cyr_{t}\left(\F_{r}\right)\colon\ \pi(w)=m}\left|\crit(w)\right|\right]^{1/t}=\max\left(\sqrt{2r-1},2m-1\right).\label{eq:counting words of given pi}
\end{equation}
\end{thm}

\begin{rem}
\label{rem:counting does not cover 1 and primitive}These counting
results in \cite{Puder2015} are stated for reduced, but not necessarily
cyclically reduced, words. However, the proof also applies to the
slightly smaller set of cyclically reduced words, and, besides, we
only use here the inequality $\le$ which obviously follows from the
original statements in \cite{Puder2015}. We also remark that for
$m\in\left\{ 2,\ldots,r\right\} $, the equality \eqref{eq:counting words of given pi}
holds with ordinary limit instead of limsup, and for $m=1$, it holds
with ordinary limit on even values of $t$.

Theorem \ref{thm:primitivity-rank-bounds} does not cover the cases
$\pi\left(w\right)=0$ and $\pi\left(w\right)=\infty$. But $\pi\left(w\right)=0$
if and only if $w=1$ so this is irrelevant in $\cyr_{t}\left(\F_{r}\right)$.
The other extreme, $\pi\left(w\right)=\infty$, holds if and only
if $w$ is primitive, in which case $\mathbb{E}_{w}\left[\chi_{s}\right]=\mathbb{E}_{\mathrm{unif}}\left[\chi_{s}\right]=1$,
so these words contribute nothing to the summation \eqref{eq:step one}.
(The exponential growth rate of primitive words is $2r-3$ -- see
\cite{PuderWu2014}.)
\end{rem}

The second task, of a uniform bound on the big-$O$ term in Corollary
\ref{cor:general sym function} and in \eqref{eq:first approx of E-1},
is given by the following proposition.
\begin{prop}
\label{prop:char-uniform-bound}Let $f\in\A$ be a class function.
Then there are constants $A,D\geq1$ such that for every word of length
$t$, and any $N>(At)^{2}$,
\begin{equation}
\left|\mathbb{E}_{w}\left[f\right]-\left\langle f,1\right\rangle -\left\langle f,\xi_{1}-1\right\rangle \cdot\frac{\left|\crit(w)\right|}{N^{\pi(w)-1}}\right|\leq\frac{\left(A\cdot t\right)^{2\pi(w)+2D}}{N^{\pi(w)-1}\cdot\left(N-\left(At\right)^{2}\right)}.\label{eq:bound on big O}
\end{equation}
\end{prop}

We first prove a version of Proposition \ref{prop:char-uniform-bound}
for the monomials $\qak$ -- see Lemma \ref{lem:uniform bound for qak}.
Using the notation of Section \ref{sec:Proof-of-Theorem}, recall
that $\mathbb{E}_{w}\left[\qak\right]=\Phi_{\eta_{\ak}^{w}}$, and
that\linebreak{}
$\Phi_{\eta}=\sum_{\left(\eta_{1},\eta_{2}\right)\in\decomp\left(\eta\right)}L_{\eta_{2}}^{B}$.
By Proposition \ref{prop:rational-function-for-L} and Corollary \ref{cor:Phi, L, R, C all rational},
the functions $\Phi_{\eta}$ and $L_{\eta}^{B}$ are equal to a rational
expression in $N$ (with rational coefficients) for every large enough
$N$. In particular, these functions are given by power series in
$\frac{1}{N}$. We shall use the following lemma.
\begin{lem}
\label{lem:L-uniform-bound}Let $\eta:\Gamma\to X_{B}$ be a $B$-surjective
morphism of multi core graphs. Assume that $\left|V\left(\Gamma\right)\right|\le T$
and that $\left|E\left(\Gamma\right)\right|\le T$. Then the coefficient
of $N^{\chi(\Gamma)-p}$ in the power series expansion of $L_{\eta}^{B}(N)$
is bounded in absolute value by $T^{2p}$.
\end{lem}

\begin{proof}
By Proposition \ref{prop:L-cover}, 
\[
L_{\eta}^{B}\left(N\right)=\sum_{t\geq0}\sum_{j_{0}\geq0;j_{1},...,j_{t}\geq1}(-1)^{t+\sum_{i=0}^{t}j_{i}}\left[V(\Gamma)\right]_{j_{0}}^{\eta}\cdot\left[E(\Gamma)\right]_{j_{1}}^{\eta}\cdot...\cdot\left[E(\Gamma)\right]_{j_{t}}^{\eta}N^{\chi(\Gamma)-\sum_{i=0}^{t}j_{i}}.
\]
The coefficient of $N^{\chi(\Gamma)-p}$ is thus

\begin{equation}
b_{p}\defi\sum_{t\geq0}\sum_{j_{0}\geq0;j_{1},...,j_{t}\geq1\colon\sum j_{i}=p}\left(-1\right)^{t+\sum j_{i}}\left[V(\Gamma)\right]_{j_{0}}^{\eta}\cdot\left[E(\Gamma)\right]_{j_{1}}^{\eta}\cdot...\cdot\left[E(\Gamma)\right]_{j_{t}}^{\eta}.\label{eq:inside proof of uniform bound for L}
\end{equation}
We proceed by induction on $p$ and ignore the signs in \eqref{eq:inside proof of uniform bound for L}.
For $p=0$, $b_{0}=1$. Note that 
\[
\left[V(\Gamma)\right]_{j}^{\eta},\left[E(\Gamma)\right]_{j}^{\eta}\leq\binom{T}{2}^{^{j}}\leq\frac{T^{2j}}{2^{j}},
\]
since any permutation counted by these numbers is the product of $j$
cycles of length $2$, and the number of vertices and edges of the
graph is bounded by $T$. Therefore, the $t=0$ term of \eqref{eq:inside proof of uniform bound for L}
is bounded by $\frac{T^{2p}}{2^{p}}$. For $t\ge1$ we put aside the
term $\left[E(\Gamma)\right]_{j_{t}}^{\eta}$ to obtain 
\begin{eqnarray*}
b_{p} & = & \frac{T^{2p}}{2^{p}}+\sum_{j=1}^{p}\left[E(\Gamma)\right]_{j}^{\eta}\cdot\sum_{t\geq1}\sum_{j_{0}\geq0;j_{1},...,j_{t-1}\geq1\colon\sum j_{i}=p-j}\left(-1\right)^{t+j+\sum j_{i}}\left[V(\Gamma)\right]_{j_{0}}^{\eta}\cdot\left[E(\Gamma)\right]_{j_{1}}^{\eta}\cdot...\cdot\left[E(\Gamma)\right]_{j_{t-1}}^{\eta}\\
 & = & \frac{T^{2p}}{2^{p}}+\sum_{j=1}^{p}\left[E(\Gamma)\right]_{j}^{\eta}\cdot\left(-1\right)^{j}b_{p-j}.
\end{eqnarray*}
By induction we get

\[
|b_{p}|\leq\frac{T^{2p}}{2^{p}}+\sum_{j=1}^{p}\left[E(\Gamma)\right]_{j}^{\eta}\cdot\left|b_{p-j}\right|\leq\frac{T^{2p}}{2^{p}}+\sum_{j=1}^{p}\frac{T^{2j}}{2^{j}}\cdot T^{2p-2j}=T^{2p}\cdot\left(\frac{1}{2^{p}}+\sum_{j=1}^{p}\frac{1}{2^{j}}\right)=T^{2p}.
\]
\end{proof}
\begin{lem}
\label{lem:uniform bound for qak}Let $w\in\cyr_{t}\left(\F_{r}\right)$
and fix $k\ge1$ and $\ak\ge0$, not all zeros. Let $T=t\cdot\sum i\alpha_{i}$
denote the number of edges and the number of vertices in the multi
core graph $\Gamma_{\ak}^{w}$ and let $D=\sum\alpha_{i}$ denote
the number of connected components in this graph. Then, for all $N>T^{2}$,
\[
\left|\mathbb{E}_{w}\left[\qak\right]-\left\langle \qak,1\right\rangle -\left\langle \qak,\xi_{1}-1\right\rangle \cdot\frac{\left|\crit(w)\right|}{N^{\pi(w)-1}}\right|\leq\frac{T^{2\pi(w)+2D}}{N^{\pi(w)-1}\cdot\left(N-T^{2}\right)}.
\]
\end{lem}

\begin{proof}
By Proposition \ref{prop:rational-function-for-L} and Corollary \ref{cor:Phi, L, R, C all rational},
$\mathbb{E}_{w}\left[\qak\right]$ is equal to a rational expression
in $N$ with rational coefficients for large enough $N$ (in fact,
$N\ge T$ suffices), and by Theorem \ref{thm:Word-Measure-Character-Bound}
its degree is zero. In particular, it is equal to a power series in
$\frac{1}{N}$. Denote
\[
\mathbb{E}_{w}\left[\qak\right]=\sum_{p=0}^{\infty}\frac{a_{p}}{N^{p}}.
\]
By Theorem \ref{thm:Word-Measure-Character-Bound}, $a_{0}=\left\langle \qak,1\right\rangle $,
$a_{1}=\ldots=a_{\pi\left(w\right)-2}=0$, and $a_{\pi\left(w\right)-1}=\left\langle \qak,\xi_{1}-1\right\rangle \cdot\left|\crit\left(w\right)\right|$.
In this notation, our goal is to bound $\sum_{p=\pi\left(w\right)}^{\infty}\frac{a_{p}}{N^{p}}$.
We claim that for every $p\ge\pi\left(w\right)$ we have $\left|a_{p}\right|\le T^{2p+2D}$.
Assuming this inequality,
\begin{eqnarray*}
\sum_{p=\pi\left(w\right)}^{\infty}\frac{a_{p}}{N^{p}} & \le & \sum_{p=\pi\left(w\right)}^{\infty}\frac{T^{2p+2D}}{N^{p}}=\frac{T^{2\pi(w)+2D}}{N^{\pi(w)}}\cdot\frac{1}{1-\frac{T^{2}}{N}}=\frac{T^{2\pi(w)+2D}}{N^{\pi(w)-1}\cdot\left(N-T^{2}\right)},
\end{eqnarray*}
as required. It remains to prove that $\left|a_{p}\right|\le T^{2p+2D}$.
Assume without loss of generality that $\eta_{\ak}^{w}\colon\Gamma_{\ak}^{w}\to X_{B}$
is onto -- otherwise, work in a free factor of $\F_{r}$ generated
by a suitable subset of $B$. As 
\begin{equation}
\mathbb{E}_{w}\left[\qak\right]=\sum_{\left(\eta_{1},\eta_{2}\right)\in\decomp\left(\eta_{\ak}^{w}\right)}L_{\mathcal{\eta}_{2}}^{B}.\label{eq:Phi as sum of L}
\end{equation}
For every decomposition $\left(\eta_{1},\eta_{2}\right)$ as in \eqref{eq:Phi as sum of L},
$\mathrm{Im}\left(\eta_{1}\right)$ has at most $T$ vertices and
at most $T$ edges, so by Lemma \ref{lem:L-uniform-bound},
\[
L_{\eta_{2}}^{B}\left(N\right)\le N^{\chi\left(\mathrm{Im}\left(\eta_{1}\right)\right)}\sum_{q=0}^{\infty}\frac{T^{2q}}{N^{q}}.
\]
In particular, the coefficient of $N^{-p}$ is $T^{2p+2\chi\left(\mathrm{Im}\eta_{1}\right)}$
or $0$ if $\chi\left(Im\eta_{1}\right)<-p$. Summing these over all
decompositions gives
\begin{eqnarray*}
\left|a_{p}\right| & \le & \sum_{c=-p}^{0}\sum_{\left(\eta_{1},\eta_{2}\right)\in\decomp\left(\eta_{\ak}^{w}\right)\colon~\chi\left(\mathrm{Im}\left(\eta_{1}\right)\right)=c}T^{2p+2c}.\\
 & \le & \sum_{c=-p}^{0}\binom{T}{2}^{D-c}\cdot T^{2p+2c}\le\sum_{c=-p}^{0}\frac{T^{2D-2c}}{2^{D-c}}\cdot T^{2p+2c}=T^{2D+2p}\cdot2^{-D}\sum_{c=-p}^{0}2^{c}\le T^{2D+2p},
\end{eqnarray*}
where the second inequality is by Proposition \ref{prop:upper bound on B-norm}
and the fact that $\Gamma_{\ak}^{w}$ has $D$ components.
\end{proof}

\begin{proof}[Proof of Proposition \ref{prop:char-uniform-bound}]
 By definition, every $f\in\A$ is a finite linear combination of
the form
\begin{equation}
f=\sum_{k,\ak}\beta_{\ak}\qak\label{eq:f as sum of qak}
\end{equation}
with $\beta_{\ak}\in\mathbb{R}$. All the terms in the bounded expression
in \eqref{eq:bound on big O} are linear, so it is bounded by the
corresponding linear combinations of bounds from Lemma \ref{lem:uniform bound for qak}.
Set $D$ to be the maximal value of $\sum\alpha_{i}$ over the non-vanishing
monomials in \eqref{eq:f as sum of qak}, and $A_{0}$ to be the maximal
value of $\sum i\alpha_{i}$. Then
\begin{eqnarray*}
\left|\mathbb{E}_{w}\left[f\right]-\left\langle f,1\right\rangle -\left\langle f,\xi_{1}-1\right\rangle \cdot\frac{\left|\crit(w)\right|}{N^{\pi(w)-1}}\right| & \le & \sum_{k,\ak}\left|\beta_{\ak}\right|\cdot\frac{\left(A_{o}t\right)^{2\pi(w)+2D}}{N^{\pi(w)-1}\cdot\left(N-\left(A_{o}t\right)^{2}\right)}\\
 & \le & \frac{\left(At\right)^{2\pi(w)+2D}}{N^{\pi(w)-1}\cdot\left(N-\left(At\right)^{2}\right)}
\end{eqnarray*}
with $A=A_{o}\cdot\max\left(\sum_{k,\ak}\left|\beta_{\ak}\right|,1\right)$.
\end{proof}
We now have all the ingredients needed to prove Theorem \ref{thm:Schreier-Graphs-Near-Ramanujan}.
\begin{proof}[Proof of Theorem \ref{thm:Schreier-Graphs-Near-Ramanujan}]
 Recall our notation of $\chi_{s}$ from the beginning of this section
and that $\left\langle \chi_{s},1\right\rangle =1$. The value of
$\left\langle \chi_{s},\xi_{1}-1\right\rangle $ is $s$: this is
a suitable Kostka number \cite[Proposition 7.18.7]{Stanley1999enumerative2},
but any constant suffices for our needs. From Proposition \ref{prop:char-uniform-bound}
it now follows that there are $A,D\ge1$ with

\[
\mathbb{E}_{w}\left[\chi_{s}\right]\leq1+\frac{1}{N^{\pi(w)-1}}\left(s\cdot\left|\crit(w)\right|+\frac{\left(At\right)^{2\pi(w)+2D}}{N-A^{2}t^{2}}\right).
\]
We now bound the summation from \eqref{eq:step one} (see also Remark
\ref{rem:counting does not cover 1 and primitive}): 
\begin{eqnarray*}
\sum_{w\in\cyr_{t}\left(\F_{k}\right)}\left(\mathbb{E}_{w}\left[\chi_{s}\right]-1\right) & = & \sum_{m=1}^{r}\sum_{w\in\cyr_{t}\left(\F_{k}\right)\colon\pi\left(w\right)=m}\left(\mathbb{E}_{w}\left[\chi_{s}\right]-1\right)\\
 & \le & \sum_{m=1}^{r}\frac{1}{N^{m-1}}\sum_{w\in\cyr_{t}\left(\F_{k}\right)\colon\pi\left(w\right)=m}\left(s\cdot\left|\crit(w)\right|+\frac{\left(At\right)^{2\pi(w)+2D}}{N-A^{2}t^{2}}\right)\\
 & \le & \sum_{m=1}^{r}\frac{1}{N^{m-1}}\sum_{w\in\cyr_{t}\left(\F_{k}\right)\colon\pi\left(w\right)=m}s\cdot\left|\crit(w)\right|\left(1+\frac{\left(At\right)^{2r+2D}}{N-A^{2}t^{2}}\right)\\
 & = & \left(1+\frac{\left(At\right)^{2r+2D}}{N-A^{2}t^{2}}\right)s\sum_{m=1}^{r}\frac{1}{N^{m-1}}\sum_{w\in\cyr_{t}\left(\F_{k}\right)\colon\pi\left(w\right)=m}\left|\crit(w)\right|\\
 & \stackrel{\mathrm{Thm}~\ref{thm:primitivity-rank-bounds}}{\le} & \left(1+\frac{\left(At\right)^{2r+2D}}{N-A^{2}t^{2}}\right)s\sum_{m=1}^{r}\frac{1}{N^{m-1}}\left[\max\left(\sqrt{2r-1},2m-1\right)+\varepsilon\right]^{t},
\end{eqnarray*}
where the last inequality holds for every $\varepsilon>0$ and every
large enough $t$. So under the same assumptions on $\varepsilon$
and $t$, we get from \eqref{eq:step one}
\begin{equation}
\mathbb{E}\left[\mathrm{Re}\left[\nu_{2}\left(G\right)^{t}\right]\right]\le2\left(N\right)_{s}\sqrt{d-1}^{t}+\left(1+\frac{\left(At\right)^{2r+2D}}{N-A^{2}t^{2}}\right)s\sum_{m=1}^{r}\frac{1}{N^{m-1}}\left[\max\left(\sqrt{2r-1},2m-1\right)+\varepsilon\right]^{t}.\label{eq:intermediate bound on nu}
\end{equation}
We will soon take $t$ to be a function of $N$ so that as $N\to\infty$,
$N^{1/t}\to c$ for a constant $c$ specified below. Then for every
$\varepsilon>0$ and every large enough $N$, 
\[
\left(1+\frac{\left(At\right)^{2r+2D}}{N-A^{2}t^{2}}\right)s\cdot2\left(r+1\right)\le\left(1+\varepsilon\right)^{t}.
\]
Because the right hand side of \eqref{eq:intermediate bound on nu}
is at most $\left(r+1\right)$ times the maximal summand (among the
$r+1$ summands), we get that for every $\varepsilon>0$ and large
enough $N$,
\begin{align}
 & \mathbb{E}\left[\mathrm{Re}\left[\nu_{2}\left(G\right)^{t}\right]\right]\le\nonumber \\
 & ~~~~\left[\left(1+\varepsilon\right)\cdot\max\left(\left\{ N^{s/t}\sqrt{d-1}\right\} \cup\left\{ \frac{2m-1}{N^{\left(m-1\right)/t}}\,\middle|\,2m-1\in\left[\sqrt{d-1},d-1\right]\right\} \right)\right]^{t},\label{eq:step three}
\end{align}
where we used the observation that if $2m-1<\sqrt{d-1}$ then the
term corresponding to $m$ in \eqref{eq:intermediate bound on nu}
is $\frac{\sqrt{d-1}^{t}}{N^{\left(m-1\right)}}$, and is thus strictly
smaller than the first term $2\left(N\right)_{s}\sqrt{d-1}^{t}$.
A simple analysis yields that, at least for large values of $d$,
the optimal value of $t=t\left(N\right)$ is such that 
\[
N^{1/t}\to e^{\frac{2}{e\sqrt{d-1}}}
\]
as $N\to\infty$. Whenever $2m-1\in\left[\sqrt{d-1},d-1\right]$,
write $m=\beta\sqrt{d-1}$ with $\beta>\frac{1}{2}$. If $N^{1/t}\to e^{\frac{2}{e\sqrt{d-1}}}$
then 
\begin{eqnarray*}
\frac{2m-1}{N^{\left(m-1\right)/t}} & = & \frac{2\beta\sqrt{d-1}-1}{\left(N^{1/t}\right)^{\beta\sqrt{d-1}-1}}<N^{1/t}\sqrt{d-1}\cdot\frac{2\beta}{\left(N^{1/t}\right)^{\beta\sqrt{d-1}}}\\
 & \approx & N^{1/t}\sqrt{d-1}\cdot\frac{2\beta}{e^{2\beta/e}}\le N^{1/t}\sqrt{d-1},
\end{eqnarray*}
where the last inequality follows as $\frac{2\beta}{e^{2\beta/e}}\le1$
with equality if and only if $\beta=e/2$. Therefore, with this value
of $t$, we obtain from \eqref{eq:step three} that for every $\varepsilon>0$,
\[
\mathbb{E}\left[\mathrm{Re}\left[\nu_{2}\left(G\right)^{t}\right]\right]\le\left[\left(1+\varepsilon\right)\cdot\sqrt{d-1}\cdot N^{s/t}\right]^{t}\approx\left[\left(1+\varepsilon\right)\cdot\sqrt{d-1}\cdot e^{\frac{2s}{e\sqrt{d-1}}}\right]^{t}
\]
for every large enough $N$. Recall that if $\nu_{2}\left(G\right)$
is non-real, then it has absolute value $\sqrt{d-1}$, and so we always
have $\mathrm{Re}\left[\nu_{2}\left(G\right)^{t}\right]\ge-\sqrt{d-1}^{t}$
for $t$ even. Therefore, for $x=\frac{2s}{e\sqrt{d-1}}$, 
\[
\mathrm{Prob}\left\{ \nu\left(G\right)\ge\left(1+2\varepsilon\right)\cdot e^{x}\sqrt{d-1}\right\} \cdot\left[\left(1+2\varepsilon\right)e^{x}\sqrt{d-1}\right]^{t}-\sqrt{d-1}^{t}\le\mathbb{E}\left[\mathrm{Re}\left[\nu_{2}\left(G\right)^{t}\right]\right]\le\left[\left(1+\varepsilon\right)e^{x}\sqrt{d-1}\right]^{t},
\]
which yields that for every $\varepsilon>0$
\[
\mathrm{Prob}\left\{ \nu\left(G\right)\ge\left(1+2\varepsilon\right)\cdot\sqrt{d-1}\cdot e^{x}\right\} \underset{N\to\infty}{\to}0.
\]
Finally, by \eqref{eq:nu vs. mu of G}, when $\nu\left(G\right)>\sqrt{d-1}$,
we have that $\mu\left(G\right)=\nu\left(G\right)+\frac{d-1}{\nu\left(G\right)}$,
and as $e^{x}+e^{-x}<2e^{x^{2}/2}$ for $x>0$, we conclude that
\[
\mathrm{Prob}\left\{ \mu\left(G\right)\ge2\sqrt{d-1}\cdot e^{\frac{2s^{2}}{e^{2}\left(d-1\right)}}\right\} \underset{N\to\infty}{\to}0.
\]
\end{proof}
\begin{rem}
\label{rem:how conjecture implies strong expansion, details}For a
fixed value of $r$ (or equivalently $d$), our bound on $\mu\left(G\right)$
gets weaker as $s$ grows. Assuming Conjecture \ref{conj:irreducible-characters},
we could improve this bound to be independent of $s$. Indeed, we
could decompose the character $\chi_{s}$ into a sum of irreducible
characters. Each character of degree $\theta\left(N^{m}\right)$ corresponds
to $O(N^{m})$ eigenvalues, and if indeed $\mathbb{E}_{w}\left[\chi\right]=O(N^{m\left(1-\pi\left(w\right)\right)})$,
we could choose $t$ separately for each $m$, such that $N^{m/t}\to e^{\frac{2}{e\sqrt{d-1}}}$,
and obtain the bound $\nu\left(G\right)<\sqrt{d-1}\cdot e^{\frac{2}{e\sqrt{d-1}}}$
a.a.s.~as $N\to\infty$. 
\end{rem}

\appendix

\section{Conjugacy separability of free groups\label{sec:conjugacy-seperability}}

We use our results to give a simple proof of the known fact that free
groups are conjugacy separable. 
\begin{defn}
A group $G$ is said to be \emph{conjugacy separable} if for every
two distinct conjugacy classes $g^{G}\ne h^{G}$ of elements of $G$
there exists some homomorphism $G\overset{\phi}{\to}Q$ to a finite
group $Q$ such that $\phi(g)^{Q}\neq\phi(h)^{Q}$.
\end{defn}

It is known that finitely generated free groups are conjugacy separable
-- see, for example, \cite[p.~278]{baumslag1965residual} or \cite[Prop.~I.4.8]{Lyndon1977}.
We give another simple proof of this fact.
\begin{prop}
Finitely generated free groups are conjugacy separable.
\end{prop}

Our proof uses a similar plan to \cite[Cor.~3.16]{wise2000subgroup},
proving separation assuming that $u^{G}\neq v^{G},(v^{-1})^{G}$,
and then separating an element from its inverse using an odd-order
quotient. For completeness, we include all the details of the proof.
\begin{proof}
Assume that $u,v\in\F=\F_{r}$, both not the identity element, are
conjugate under every homomorphism to a finite group $Q$. We write
$u,v$ as maximal powers in the free group, that is, $u=u_{0}^{~k},v=v_{0}^{~m}$,
where $u_{0},v_{0}$ are non-powers and $k,m\in\mathbb{Z}_{\ge1}$.
Recall that $\tau\left(k\right)$ marks the number of positive divisors
of $k$. We assume without loss of generality that $\tau(k)\geq\tau(m)$.

Using $Q=S_{N}$, we see that for every $r$ permutations $\sigma_{1},\dots,\sigma_{r}\in S_{N}$,
the resulting permutations $u\left(\sigma_{1},\ldots,\sigma_{r}\right)$
and $v\left(\sigma_{1},\ldots,\sigma_{r}\right)$ are conjugate. In
particular, they have the same number of fixed points. It follows
that 

\[
\Phi_{\left\{ \left\langle u\right\rangle ^{\F},\left\langle v\right\rangle ^{\F}\right\} \to\F}=\mathbb{E}_{\sigma_{1},\ldots,\sigma_{r}\in S_{N}}\left[\#\mathrm{fix}\big(u\big)\cdot\#\mathrm{fix}\big(v\big)\right]=\mathbb{E}\left[\#\mathrm{fix}\left(u\right)^{2}\right]=\Phi_{\left\{ \left\langle u\right\rangle ^{\F},\left\langle u\right\rangle ^{\F}\right\} \to\F}.
\]
As a simple consequence of Theorem \ref{thm:character-first-bound}
we obtain\footnote{In fact, while \eqref{eq:Phi of u,u} can be derived directly from
Theorem \ref{thm:character-first-bound}, it is a much simpler fact.
Using left Möbius inversion, one can easily see that for $\eta:\Gamma\to\Delta$,
$L_{\eta}^{B}=N^{\chi(\Gamma)}\big(1+O\big(\frac{1}{N}\big)\big)$,
and therefore one simply needs to count decompositions of $\left\{ \left\langle u\right\rangle ^{\F},\left\langle u\right\rangle ^{\F}\right\} \to\F$
as $\left\{ \left\langle u\right\rangle ^{\F},\left\langle u\right\rangle ^{\F}\right\} \to\H\to\F$
with $\chi\left(\H\right)=0$. A similar argument applies to computing
the free coefficient of $\Phi_{\left\{ \left\langle u\right\rangle ^{\F},\left\langle v\right\rangle ^{\F}\right\} \to\F}$.
The method for solving such counting problems, albeit not explicitly
with multi core graphs involving (powers of) different words, appears
already in \cite{nica1994number} and especially in \cite{Linial2010}.}, 
\begin{equation}
\Phi_{\left\{ \left\langle u\right\rangle ^{\F},\left\langle u\right\rangle ^{\F}\right\} \to\F}=\uexp\left[\xi_{k}^{2}\right]+O\left(\frac{1}{n}\right)=\tau\left(k\right)^{2}+\sum_{d|k}d+O\left(\frac{1}{n}\right)\geq\tau(k)^{2}+1+O\left(\frac{1}{n}\right).\label{eq:Phi of u,u}
\end{equation}
On the other hand, the decompositions of $\left\{ \left\langle u\right\rangle ^{\F},\left\langle v\right\rangle ^{\F}\right\} \to\H\to\F$
with $\chi\left(\H\right)=0$ and $c\left(\H\right)=2$ (two connected
components) are in one to one correspondence with pairs of positive
roots of $u$ and $v$. Therefore, there are $\tau(k)\cdot\tau(m)\leq\tau(k)^{2}$
such decompositions. It follows that there is at least one decomposition
$\left\{ \left\langle u\right\rangle ^{\F},\left\langle v\right\rangle ^{\F}\right\} \to\H\to\F$
with $\chi\left(\H\right)=0$ and $c\left(\H\right)=1$, namely $\H=\left\{ \left\langle w\right\rangle ^{\F}\right\} $
for some $1\ne w\in\F$. Therefore, some conjugates of $u$ and $v$
belong to $\left\langle w\right\rangle $, which yields that $\left\langle u_{0}\right\rangle $
and $\left\langle v_{0}\right\rangle $ are conjugate, i.e., $u_{0}$
is conjugate to either $v_{0}$ or $v_{0}^{-1}$. We now prove that
$k=m$. Indeed, 

\[
\Phi_{\left\{ \left\langle u\right\rangle ^{\F},\left\langle v\right\rangle ^{\F}\right\} \to\F}=\tau\left(k\right)\cdot\tau\left(m\right)+\sum_{d|gcd(k,m)}d+O\left(\frac{1}{n}\right),
\]
and therefore,

\[
\tau(k)^{2}+\sum_{d|k}d=\tau(k)\cdot\tau(m)+\sum_{d|gcd(k,m)}d.
\]
It follows that $k=m$, and so $u$ is conjugate to either $v$ or
$v^{-1}$.

We finish the proof by showing that there is no $1\ne u\in\F$ such
that $u$ and $u^{-1}$ are conjugate in every finite quotient of
$\F$. Assume otherwise. By \cite[Thm.~6.1.9]{robinson2012course}
free groups are residually $p$-finite for every prime number $p$.
In particular, this is the case for any odd $p$, so there is a finite
quotient $\phi:\F\to P$ where $P$ is a $p$-group of odd order,
and $\phi(u)\neq1$. By our assumption, there is an element $x\in P$
such that $\phi(u)^{-1}=x\phi(u)x^{-1}$. Then $x^{2}\phi(u)x^{-2}=x\phi(u)^{-1}x^{-1}=\phi(u)$,
so $x^{2}$ lies in the centralizer $C_{P}(\phi(u))$. However, as
$P$ has odd order, so does $x$ and therefore $x\in\left\langle x^{2}\right\rangle \le C_{P}(\phi(u))$.
In particular, $\phi(u)^{-1}=x\phi(u)x^{-1}=\phi(u)$, and as $P$
has odd order, $\phi(u)=1$, in contradiction. 
\end{proof}
\begin{rem}
This proof is even simpler for non-powers: in this case there are
no non-trivial decompositions of the morphisms $\left\{ \left\langle u\right\rangle \right\} \to\F,\left\{ \left\langle v\right\rangle ^{\F}\right\} \to\F$
with Euler characteristic $0$.
\end{rem}

\section{The ring of class functions\label{sec:class-functions}}

Recall that for any $N$ and any permutation $\sigma\in S_{N}$, $\xi_{k}\left(\sigma\right)$
denotes the number of fixed points of $\sigma^{k}$, and $a_{t}\left(\sigma\right)$
denotes the number of $t$-cycles in $\sigma$. As in Section \ref{sec:Introduction},
we consider $\A=\mathbb{Q}\left[\xi_{1},\xi_{2},\ldots\right]$, the
ring of formal polynomials in the countably many variables $\xi_{k}$.
Every element of $\A$ is a class function defined on $S_{N}$ for
every $N$. Note that as class functions on $S_{N}$,

\[
\z_{k}=\sum_{t|k}t\cdot a_{t}.
\]
Therefore, $\A$ can be equivalently defined as $\A=\mathbb{Q}\left[a_{1},a_{2},\ldots\right]$.
The following proposition shows that for any $f,g\in\A$, the inner
product $\left\langle f,g\right\rangle _{S_{N}}$ stabilizes for large
enough $N$, and therefore our definition in Section \ref{sec:Introduction}
of $\left\langle f,g\right\rangle $ (as the constant value obtained
for large $N$) makes sense.
\begin{prop}
\label{prop:inner product stabilizes}For every two class functions
$f,g\in\A$, and for all large enough $N$, $\left\langle f,g\right\rangle _{S_{N}}$
is independent of $N$.
\end{prop}

\begin{proof}
By the previous paragraph, $f$ and $g$ are equal to polynomials
in the $a_{t}$'s. Thus, it is enough to prove the proposition when
$f$ and $g$ are monomials in the $a_{t}$'s. Note that $\left\langle f,g\right\rangle _{S_{N}}=\left\langle fg,1\right\rangle _{S_{N}}$,
so it is enough to show that for every monomial $m$ in the $a_{t}$'s,
$\left\langle m,1\right\rangle _{S_{N}}$ stabilizes. But \cite[Theorem 7]{diaconis1994eigenvalues}
states that for every $b_{1},\ldots,b_{k}\in\mathbb{Z}_{\ge0}$, and
for every\footnote{There is a typo in the original statement of \cite[Theorem 7]{diaconis1994eigenvalues},
where it says $N\ge\sum_{t=1}^{k}ta_{t}$ instead.} $N\ge\sum_{t=1}^{k}tb_{t}$
\[
\left\langle a_{1}^{b_{1}}\cdots a_{k}^{b_{k}},1\right\rangle _{S_{N}}=\prod_{t=1}^{k}\mathbb{E}\left[Z_{t}^{b_{t}}\right],
\]
where $Z_{t}$ is Poisson with parameter $\frac{1}{t}$. In particular,
$\left\langle a_{1}^{b_{1}}\cdots a_{k}^{b_{k}},1\right\rangle _{S_{N}}$
is constant for $N\ge\sum_{t=1}^{k}tb_{t}$.
\end{proof}
It is well known that there is a natural correspondence between partitions
$\lambda$ of $N$ and irreducible representations of $S_{N}$. We
denote the character corresponding to $\lambda$ by $\chi^{\lambda}$.
Recall our notation from Section \ref{sec:Introduction} and particularly
Section \ref{subsec:Similar-phenomena-in} of $\left|\lambda\right|$
(the sum of blocks in $\lambda$), of $\chi^{\lambda}\left(\rho\right)$
where $\rho\vdash\left|\lambda\right|$ (the value of $\chi^{\lambda}$
on permutations with cycle structure $\rho$) and of $z_{\lambda}\defi\prod_{r}r^{\alpha_{r}}\alpha_{r}!$,
where $\lambda$ has $\alpha_{r}$ parts of size $r$. Also recall
from Section \ref{sec:Introduction} that every partition $\lambda$
gives rise to a family of irreducible characters $\chi=\left\{ \chi_{N}\right\} _{N\ge\left|\lambda\right|+\lambda_{1}}$,
and that $\symirr$ denotes the family of such families of irreducible
characters.
\begin{prop}
\label{prop:irreps as linear basis of A}Every $\chi\in\symirr$ corresponds
to an element of $\A$, namely, $\chi_{N}$ and this element of $\A$
coincide as class functions on $S_{N}$ for every $N\ge\left|\lambda\right|+\lambda_{1}$.
Moreover, the elements of $\A$ corresponding to the elements of $\symirr$
constitute a linear basis of $\A$.
\end{prop}

\begin{proof}
Our proof relies on results from \cite{macdonald1998symmetric}. For
a partition $\lambda$, denote by $\ell(\lambda)$ the number of parts
in $\lambda$. If $\rho\vdash k$ and $\sigma\vdash m$, denote by
$\rho\cup\sigma$ the partition of $m+k$ obtained at the disjoint
union of parts of $\rho$ and $\sigma$. Also denote
\[
\binom{a}{\lambda}\defi\prod_{r}\binom{a_{r}}{\alpha_{r}\left(\lambda\right)}=\prod_{r}\frac{a_{r}\cdot\left(a_{r}-1\right)\cdots\left(a_{r}-\alpha_{r}\left(\lambda\right)+1\right)}{\alpha_{r}\left(\lambda\right)!},
\]
where $\alpha_{r}\left(\lambda\right)$ is the number of parts of
size $r$ in $\lambda$. Now let $\chi=\left\{ \chi_{N}\right\} _{N\ge\left|\lambda\right|+\lambda_{1}}\in\symirr$
be the family of irreducible characters corresponding to the partition
$\lambda$. According to \cite[Example I.7.14]{macdonald1998symmetric},
for every $N\ge\left|\lambda\right|+\lambda_{1}$, the class function
$\chi_{N}$ on $S_{N}$ is equal to 
\begin{equation}
\chi_{N}=\sum_{\rho,\sigma~\colon\ \left|\rho\right|+\left|\sigma\right|=\left|\lambda\right|}\frac{\left(-1\right)^{\ell\left(\sigma\right)}\cdot\chi^{\lambda}\left(\rho\cup\sigma\right)}{z_{\sigma}}\cdot\binom{a}{\rho},\label{eq:formula for chi-lambda}
\end{equation}
where the sum is over all partitions $\rho$ and $\sigma$, including
the empty partitions (of $0$), with $\left|\rho\right|+\left|\sigma\right|=\left|\lambda\right|$.
See Example \ref{exa:how the formula works} below. In particular,
\eqref{eq:formula for chi-lambda} shows that indeed $\chi$ coincides
with a certain element of $\A$ for every $N\ge\left|\lambda\right|+\lambda_{1}$.

Now fix $k\in\mathbb{Z}_{\ge1}$, and consider all partitions $\left\{ \lambda\vdash q\,\middle|\,0\le q\le k\right\} $.
The number of such partitions is $p\left(0\right)+p\left(1\right)+\ldots+p\left(k\right)$.
For large enough $N$, all these partitions give rise to distinct
irreducible characters of $S_{N}$, and are, in particular, linearly
independent class functions. On the other hand, the formula \eqref{eq:formula for chi-lambda}
shows that as elements in $\A$, they are spanned by the monomials
\[
a_{1}^{m_{1}}a_{2}^{m_{2}}\cdots a_{k}^{m_{k}}
\]
with $m_{1}+2m_{2}+\ldots+km_{k}\le k$. These are precisely the possible
cycle-structures of elements in $S_{0},S_{1},\ldots,S_{k}$, and therefore
there are $p\left(0\right)+p\left(1\right)+\ldots+p\left(k\right)$
such monomials, which span a linear subspace of $\A$ of dimension
$p\left(0\right)+p\left(1\right)+\ldots+p\left(k\right)$. We conclude
that this subspace is spanned by the $\chi\in\symirr$ corresponding
to $\lambda$ with $\left|\lambda\right|\le k$, and thus that the
elements in $\symirr$ form a linear basis of $\A$.
\end{proof}
\begin{rem}
Proposition \ref{prop:irreps as linear basis of A} yields that every
class function $f\in\A$ is a linear combination of the elements of
$\symirr$. Together with the orthogonality of irreducible characters,
this gives another proof of Proposition \ref{prop:inner product stabilizes}.
\end{rem}

We end this appendix by illustrating how the formula \eqref{eq:formula for chi-lambda}
works.
\begin{example}
\label{exa:how the formula works}Consider the partition $\lambda=(1)$,
of a single element. It gives rise to the family of irreducible characters
in the second row of Table \ref{tab:Some-families-of-irreps}. The
character $\chi^{1}$ is the trivial character on the trivial group
$S_{1}$. In the sum \eqref{eq:formula for chi-lambda}, $\left(\rho;\sigma\right)$
are either $\left(1;\emptyset\right)$ or $\left(\emptyset;1\right)$,
and

\[
\chi_{N}=\frac{\left(-1\right)^{0}\cdot1}{1}\binom{a}{1}+\frac{\left(-1\right)^{1}\cdot1}{1}\binom{a}{0}=a_{1}-1=\z_{1}-1.
\]

Next, consider the partition $\lambda=(1,1)$. It gives rise to the
family of irreducible characters in the fourth row of Table \ref{tab:Some-families-of-irreps}.
The character $\chi^{1,1}$ is the sign character on $S_{2}$. In
the sum \eqref{eq:formula for chi-lambda}, $\left(\rho;\sigma\right)$
are either $\left(2;\emptyset\right)$, $\left(1,1;\emptyset\right)$,
$\left(1;1\right)$, $\left(\emptyset;1,1\right)$ or $\left(\emptyset;2\right)$,
and so
\begin{eqnarray*}
\chi_{N} & = & \frac{\left(-1\right)^{0}\cdot\left(-1\right)}{1}\binom{a}{2}+\frac{\left(-1\right)^{0}\cdot1}{1}\binom{a}{1,1}+\frac{\left(-1\right)^{1}\cdot1}{1}\binom{a}{1}+\frac{\left(-1\right)^{2}\cdot1}{2}\binom{a}{0}+\frac{\left(-1\right)^{1}\cdot\left(-1\right)}{2}\binom{a}{0}\\
 & = & -\binom{a_{2}}{1}+\binom{a_{1}}{2}-\binom{a_{1}}{1}+\frac{1}{2}+\frac{1}{2}=\frac{\left(a_{1}-1\right)\left(a_{1}-2\right)}{2}-a_{2}.
\end{eqnarray*}
\end{example}

\section{Norm of morphisms: the proof of Theorem \ref{thm:alg distance  =00003D comb distance}\label{sec:norm-of-morphism-proof}}

In this final section we prove Theorem \ref{thm:alg distance  =00003D comb distance}
which shows that the norm and the $B$-norm of a morphism are, in
principle, identical. More concretely, if $\eta\colon\Gamma\to\Delta$
is a morphism of $B$-labeled multi core graphs, $\Sigma=\mathrm{Im}\left(\eta\right)$
and $\xymatrix{\Gamma\ar@{->>}[r]^{\overline{\eta}} & \Sigma\ar@{^{(}->}[r]^{\iota} & \Delta}
$ is the decomposition of $\eta$ to a surjective and an injective
morphisms, then 
\begin{equation}
\left\Vert \eta\right\Vert =\left\Vert \overline{\eta}\right\Vert _{B}+\left[\chi\left(\Sigma\right)-\chi\left(\Delta\right)\right].\label{eq:apndx B - what we need to prove}
\end{equation}
The following proof generalizes the ideas in \cite[Section 3]{Puder2014},
which dealt only with connected core graphs.
\begin{proof}[Proof of Theorem \ref{thm:alg distance  =00003D comb distance}]
 Clearly, any component $\Delta'$ of $\Delta$ which does not meet
$\eta\left(\Gamma\right)$, adds $-\chi\left(\Delta'\right)$ to both
sides of \eqref{eq:apndx B - what we need to prove}, so we may ignore
such components altogether and assume that $\mathrm{Im}\left(\eta\right)$
meets every component of $\Delta$.

Recall Definition \ref{def:basis independent norm}, and consider
the set of all possible sequences 
\begin{equation}
\left(\beta_{1},\ldots,\beta_{\left\Vert \eta\right\Vert }\right)~\colon~\beta_{\left\Vert \eta\right\Vert }\circ\ldots\circ\beta_{1}=\eta\label{eq:sequence of immedaite moprhisms}
\end{equation}
of length $\left\Vert \eta\right\Vert $ of immediate morphisms, such
that the composition of the sequence gives $\eta$. For each such
sequence, consider a sequence of non-negative integers 
\[
\left(h_{1},\ldots,h_{\left\Vert \eta\right\Vert }\right),
\]
defined as the number of edges in the codomain of $\beta_{i}$ \emph{not
}covered by edges from its domain. Namely, if $\beta_{i}\colon\Sigma_{i-1}\to\Sigma_{i}$,
then $h_{i}=\left|E\left(\Sigma_{i}\right)\right|-\left|E\left(\mathrm{Im}\beta_{i}\right)\right|$.
There are two observations to be made now:
\begin{itemize}
\item $h_{i}=0$ if and only if $\beta_{i}$ is $B$-surjective, if and
only if $\beta_{i}$ also corresponds to a ``merging step'' corresponding
to the $B$-norm from Definition \ref{def:norm of partition, surjective morphism}.
\item $h_{i}>0$ if and only if $\beta_{i}$ is $B$-injective.
\end{itemize}
In the case of an immediate morphism of the first type (of the two
types described in Definition \ref{def:basis independent norm}),
these observations are explained in \cite[Section 3]{Puder2014}.
The reasoning in the case of immediate morphisms of the second type
is very similar.

Among all sequences as in \eqref{eq:sequence of immedaite moprhisms},
consider one with minimal corresponding integer sequence with respect
to the lexicographic order. We will next prove that in such a sequence
it is impossible to have $h_{i}>0$ and $h_{i+1}=0$. This will imply
that $\eta$ can be obtained as a sequence $\left(\beta_{1},\ldots,\beta_{k}\right)$
of merging-steps from Definition \ref{def:norm of partition, surjective morphism},
followed by a sequence $\left(\beta_{k+1},\ldots,\beta_{\left\Vert \eta\right\Vert }\right)$
of embeddings. This break-up of $\eta$ thus exactly corresponds to
the decomposition $\xymatrix{\Gamma\ar@{->>}[r]^{\overline{\eta}} & \Sigma\ar@{^{(}->}[r]^{\iota} & \Delta}
$ of $\eta$ to a surjective $\overline{\eta}$ and an injective $\iota$.
So $\left\Vert \overline{\eta}\right\Vert _{B}\le k=\left\Vert \overline{\eta}\right\Vert $,
and knowing the converse inequality from \eqref{eq:ineq of norms for surj morphism},
we get $\left\Vert \overline{\eta}\right\Vert _{B}=k$. As $\iota$
is injective, it is also free (Proposition \ref{prop:properties of free morphisms}\eqref{enu:Every-injective-morphism is free}),
and by Lemma \ref{lem:distance-lower-bound}, $\left\Vert \iota\right\Vert =\chi\left(\Sigma\right)-\chi\left(\Delta\right)$.
All in all
\[
\left\Vert \eta\right\Vert =k+\left\Vert \iota\right\Vert =\left\Vert \overline{\eta}\right\Vert _{B}+\chi\left(\Sigma\right)-\chi\left(\Delta\right),
\]
as required.

It remains to prove that in the minimal sequence $\left(h_{1},\ldots,h_{\left\Vert \eta\right\Vert }\right)$,
we cannot have $h_{i}>0$ and $h_{i+1}=0$. Let $\beta_{i}\colon\Sigma_{i-1}\to\Sigma_{i}$
be an immediate morphism and $h_{i}=\left|E\left(\Sigma_{i}\right)\right|-\left|E\left(\mathrm{Im}\beta_{i}\right)\right|$
the corresponding integer. The two types of immediate morphisms from
Definition \ref{def:basis independent norm} have the following geometric
realizations. A step of the first type, where $H^{\F}$ is replaced
by $\left\langle H,j\right\rangle ^{\F}$, is obtained geometrically
by adding a cycle spelling the word $j$ at the vertex $v$ at which
$H$ is based (so $\piol\left(\Sigma,v\right)=H$), and folding. A
step of the second type, where $H^{\F}$ and $H'^{\F}$ are replaced
by $\left\langle jHj^{-1},H'\right\rangle ^{\F}$, is obtained geometrically
by adding a path spelling the word $j$ starting at $v'$ and ending
at $v$, and folding. In both cases, if $h_{i}>0$, the excessive
edges of $\Sigma_{i}\setminus\mathrm{Im}\beta_{i}$ form either a
path, a cycle or a balloon. If $h_{i}=0$, this process can also be
obtained by gluing together two suitable vertices of $\Sigma_{i-1}$
and folding. 

Now assume that $h_{i}>0$ and $h_{i+1}=0$, denote by $p=\Sigma_{i}\setminus\mathrm{Im}\beta_{i}$
the (open) path, cycle or balloon in $\Sigma_{i}\setminus\mathrm{Im}\beta_{i}$,
and consider a pair of vertices of $\Sigma_{i}$ that are glued together
to obtain $\beta_{i+1}$ (with folding). We claim that we may ``exchange''
the order of these two steps and get a pair of integers which is lexicographically
smaller that $\left(h_{i},h_{i+1}\right)$. Indeed, this is certainly
true if both vertices are not on $p$, in which case we may first
merge them and only then add the path\textbackslash cycle corresponding
to $\beta_{i}$: this results in the pair $\left(0,h'\right)$ (with
$h'<h_{i}$, although this is immaterial). If one or two of the merged
vertices are on $p$, we can easily find a step which, algebraically,
is equivalent to merging them, and which can be performed on $\Sigma_{i-1}$
with corresponding $h$ strictly smaller than $h_{i}$ (and then perform
the step corresponding algebraically to $\beta_{i}$). This is illustrated
in Figure \ref{fig:how exchanging steps can reduce integer sequence}.
\end{proof}
\begin{figure}
\includegraphics[scale=0.6]{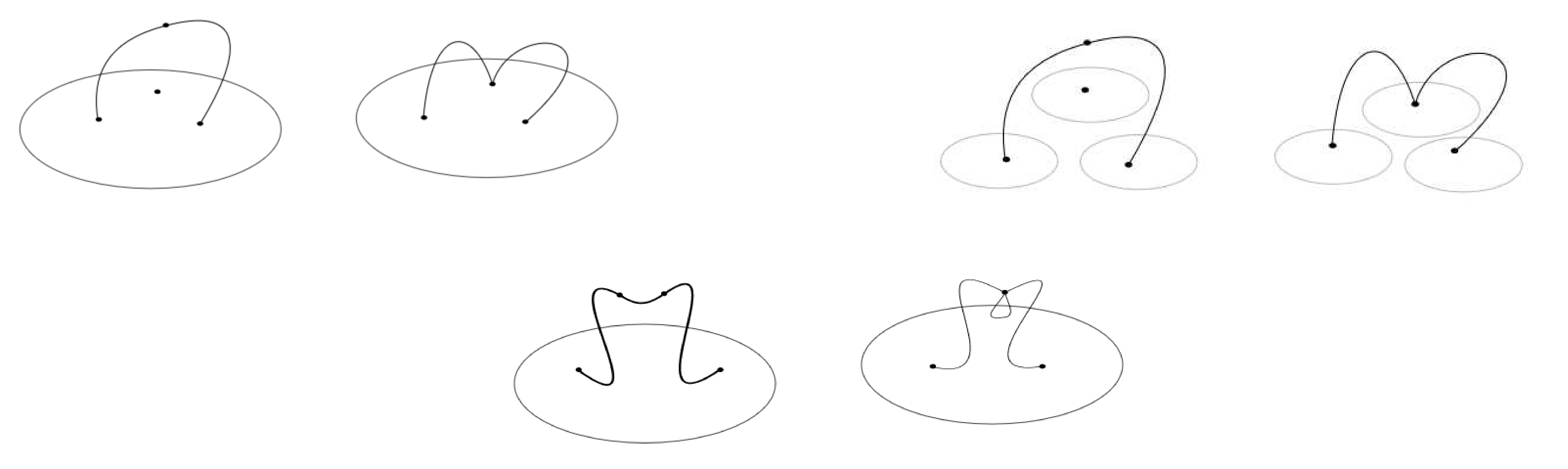}

\caption{\label{fig:how exchanging steps can reduce integer sequence}In every
pair of figures, the one on the left shows the path $p$ of length
$h_{i}>0$ which corresponds to the immediate morphism $\beta_{i}$,
and two vertices whose merging corresponds to the immediate morphism
$\beta_{i+1}$ (with $h_{i+1}=0$). The right figure in every pair
shows how the same final result can be obtained by first performing
a step which is equivalent to merging the two vertices and only then
performing a step equivalent to $\beta_{i}$. Making this change results
in a lexicographically smaller pair of integers.}
\end{figure}

\section*{Glossary}
\begin{center}
\begin{tabular}{|>{\centering}m{0.2\columnwidth}|>{\centering}m{0.37\columnwidth}|>{\centering}m{0.18\columnwidth}|>{\centering}m{0.25\columnwidth}|}
\hline 
 &  & Reference & Remarks\tabularnewline[\doublerulesep]
\hline 
$\F$ & free group of rank $r$ &  & \tabularnewline[\doublerulesep]
\hline 
$\mathbb{E}_{w}$ & expectation w.r.t.~the $w$-measure &  & \tabularnewline[\doublerulesep]
\hline 
$\mathbb{E}_{\mathrm{unif}}$ & expectation w.r.t.~the uniform measure &  & \tabularnewline[\doublerulesep]
\hline 
$\xi_{k}\left(\sigma\right)$ & number of fixed points in the permutation $\sigma^{k}$ & Equation \eqref{eq:xi k} & \tabularnewline[\doublerulesep]
\hline 
$a_{t}\left(\sigma\right)$ & number of $t$-cycles in the permutation $\sigma$ &  & \tabularnewline[\doublerulesep]
\hline 
$\pi\left(w\right)$ & primitivity rank of $w$ & Definition \prettyref{def:primitivity-rank} & \tabularnewline[\doublerulesep]
\hline 
$\crit\left(w\right)$ & set of critical subgroups of $w$ & Definition \prettyref{def:primitivity-rank} & \tabularnewline[\doublerulesep]
\hline 
$\A$ & the algebra $\mathbb{Q}\left[\xi_{1},\xi_{2},\ldots\right]$ & page \pageref{the ring A} & \tabularnewline[\doublerulesep]
\hline 
$\left\langle f,g\right\rangle $ & stable value of $\left\langle f,g\right\rangle _{S_{N}}$ & page \pageref{<f,g>} & $f,g\in\A$\tabularnewline[\doublerulesep]
\hline 
$\symirr$ & stable irreducible characters of $\left\{ S_{N}\right\} _{N}$ & page \pageref{stable irreps} & subset of $\A$\tabularnewline[\doublerulesep]
\hline 
$B=\left\{ b_{1},\ldots,b_{r}\right\} $ & a fixed basis of $\F$ &  & \tabularnewline[\doublerulesep]
\hline 
\end{tabular}
\par\end{center}

\begin{center}
\begin{tabular}{|>{\centering}m{0.2\columnwidth}|>{\centering}m{0.37\columnwidth}|>{\centering}m{0.18\columnwidth}|>{\centering}m{0.25\columnwidth}|}
\hline 
 &  & Reference & Remarks\tabularnewline[\doublerulesep]
\hline 
$\mocc\left(\F\right)$ & the category of multisets of conjugacy classes of non-trivial f.g.~subgroups
of $\F$ &  & for the morphisms see Definition \ref{def:morphism in MOCC}\tabularnewline[\doublerulesep]
\hline 
$\mcg_{B}\left(\F\right)$ & the category of $B$-labeled multi core graphs & Definition \ref{def:multi-core-graphs} & for the morphisms see Definition \ref{def:morphism of multi CG}\tabularnewline[\doublerulesep]
\hline 
$\piol\left(\Gamma\right)$ & the multiset in $\mocc\left(\F\right)$ corresponding to the multi
core graph $\Gamma$ & Section \ref{subsec:Multi-core-graphs} & \tabularnewline[\doublerulesep]
\hline 
$\Gamma_{B}\left(\H\right)$ & the multi core graph corresponding to the multiset $\H\in\mocc\left(\F\right)$ & Section \ref{subsec:Multi-core-graphs} & \tabularnewline[\doublerulesep]
\hline 
$\rk\Gamma=\rk\H,\chi\left(\Gamma\right)=\chi\left(\H\right)$, $c\left(\Gamma\right)=c\left(\H\right)$ & sum of ranks of subgroups in $\H$, Euler characteristic of $\Gamma$,
$\left|\H\right|$ & Definition \ref{def:chi, c, rank} & $\H=\piol\left(\Gamma\right)$; $\rk+\chi=c$\tabularnewline[\doublerulesep]
\hline 
$\Phi_{\eta}\left(N\right)$ & the expected number of lifts of $\eta$ to a random $N$-cover of
$\Delta$ & Definition \ref{def:Phi}, Proposition \ref{prop:geometric meaning for Phi} & $\eta\colon\Gamma\to\Delta$ is a morphism in $\mcg_{B}\left(\F\right)$\tabularnewline[\doublerulesep]
\hline 
$X_{B}$ & the bouquet in $\mcg_{B}\left(\F\right)$ representing $\left\{ \F^{\F}\right\} $ &  & \tabularnewline[\doublerulesep]
\hline 
$\Gamma\stackrel{*}{\to}\Delta$ & a free morphism of multi core graphs & Definition \ref{def:free morphisms} & \tabularnewline[\doublerulesep]
\hline 
$\left\Vert \eta\right\Vert _{B}$ & $B$-norm of the $B$-surjective morphism $\eta$ & Equation \eqref{eq:B-norm of a morphism} & \tabularnewline[\doublerulesep]
\hline 
$\left\Vert \eta\right\Vert $ & norm of the morphism $\eta$ & Definition \ref{def:basis independent norm} & \tabularnewline[\doublerulesep]
\hline 
$\chimax\left(\eta\right)$ & maximal $\chi\left(\Sigma\right)$ of all decompositions of $\eta\colon\Gamma\to\Delta$
as $\xymatrix{\Gamma\ar[r]_{\mathrm{alg}} & \Sigma\ar[r]_{*} & \Delta}
$ & Definition \ref{def:chi-max} & \tabularnewline[\doublerulesep]
\hline 
$\crit\left(\eta\right)$ & set of critical decompositions of $\eta$ & Definition \ref{def:chi-max} & \tabularnewline[\doublerulesep]
\hline 
$\decomp\left(\eta\right)$, $\decompt\left(\eta\right)$ & decompositions of $\eta$ to pairs/triples of $B$-surjective morphisms & Definition \ref{def:set of decompositions} & $\eta$ is $B$-surjective\tabularnewline[\doublerulesep]
\hline 
$L^{B},R^{B},C^{B}$ & Möbius inversions of $\Phi$ in the category of $B$-surjective morphisms & Section \ref{subsec:Basis-dependent-M=0000F6bius} & \tabularnewline[\doublerulesep]
\hline 
$\algdecomp\left(\eta\right)$, $\algdecompt\left(\eta\right)$ & decompositions of $\eta$ to pairs/triples of algebraic morphisms & Definition \ref{def:set of decompositions} & $\eta$ is algebraic\tabularnewline[\doublerulesep]
\hline 
$L^{\mathrm{alg}},R^{\mathrm{alg}},C^{\mathrm{alg}}$ & Möbius inversions of $\Phi$ in the category of algebraic morphisms & Section \ref{subsec:Algebraic-M=0000F6bius-Inversion} & \tabularnewline[\doublerulesep]
\hline 
$\Gamma_{\ak}^{w}$ & the multiset of cycles corresponding to $\qak$ in $\mcg_{B}\left(\F\right)$ & Example \ref{exa:main thm in terms of Phi} & \tabularnewline[\doublerulesep]
\hline 
$\eta_{\ak}^{w}$ & the morphism from $\Gamma_{\ak}^{w}$ to $X_{B}$ & Example \ref{exa:main thm in terms of Phi} & \tabularnewline[\doublerulesep]
\hline 
$\chiak$ & maximal \emph{negative} $\chi\left(\Sigma\right)$ of all algebraic
morphisms $\Gamma_{\ak}^{w}\to\Sigma$ & Definition \ref{def:chi-max for multiset of cycles} & \tabularnewline[\doublerulesep]
\hline 
$\crit_{\ak}$ & critical morphisms realizing $\chiak$ & Definition \ref{def:chi-max for multiset of cycles} & \tabularnewline[\doublerulesep]
\hline 
\end{tabular}
\par\end{center}

\bibliographystyle{alpha}
\bibliography{word_measures_on_symmetric_groups}

\noindent Liam Hanany, School of Mathematical Sciences, Tel Aviv University,
Tel Aviv, 6997801, Israel\\
\texttt{liamhanany@mail.tau.ac.il }~\\

\noindent Doron Puder, School of Mathematical Sciences, Tel Aviv University,
Tel Aviv, 6997801, Israel\\
\texttt{doronpuder@gmail.com}
\end{document}